\documentclass[12pt]{amsart}
\usepackage{amssymb}
\usepackage{amsbsy}
\usepackage{amscd}

\usepackage[mathscr]{eucal}

\usepackage[utf8]{inputenc}

\DeclareSymbolFont{rsfs}{U}{rsfs}{m}{n}
\DeclareSymbolFontAlphabet{\mathrsfs}{rsfs}

\usepackage{a4wide}
\usepackage[all]{xy}
\usepackage{tikz-cd}
\usepackage{tikz}
\usetikzlibrary{arrows,decorations.pathmorphing,backgrounds,positioning,fit,petri}
\usetikzlibrary{decorations.markings}
\usetikzlibrary{patterns,calc}

\usepackage{verbatim}
\usepackage{version}
\usepackage{color}
\definecolor{darkspringgreen}{rgb}{0.09, 0.45, 0.27}
\definecolor{deepjunglegreen}{rgb}{0.0, 0.29, 0.29}
\newenvironment{NB}{
\color{red}{\bf NB}. \footnotesize
}{}
\newenvironment{NB2}{
\color{blue}{\bf NB2}. \footnotesize
}{}
\newenvironment{NB3}{
\color{purple}{\bf NB3}. \footnotesize
}{}
\excludeversion{NB}
\excludeversion{NB2}
\excludeversion{NB3}
\makeatletter

\usepackage{url}
\usepackage{ifmtarg}

\usepackage{xr-hyper}
\usepackage[
pdftex,
bookmarks=true,
colorlinks=true,
citecolor=deepjunglegreen,
filecolor=deepjunglegreen,
debug=true,
unicode,
pdfnewwindow=true]{hyperref}
\usepackage{cleveref}

\crefname{Theorem}{Theorem}{Theorems}
\crefformat{Theorem}{Theorem~#2#1#3}
\crefname{section}{\S}{\S\S}
\crefformat{section}{\S#2#1#3} 
\crefmultiformat{section}{\S\S#2#1#3}{, #2#1#3}{, #2#1#3}{, #2#1#3}
\crefname{Lemma}{Lemma}{Lemmas}
\crefformat{Lemma}{Lemma~#2#1#3}
\crefname{Proposition}{Proposition}{Propositions}
\crefformat{Proposition}{Proposition~#2#1#3}
\crefname{Corollary}{Corollary}{Corollaries}
\crefformat{Corollary}{Corollary~#2#1#3}
\crefname{Definition}{Definition}{Definitions}
\crefformat{Definition}{Definition~#2#1#3}
\crefname{Remark}{Remark}{Remarks}
\crefformat{Remark}{Remark~#2#1#3}
\crefname{Remarks}{Remark}{Remarks}
\crefformat{Remarks}{Remark~#2#1#3}
\crefname{Conjecture}{Conjecture}{Conjectures}
\crefformat{Conjecture}{Conjecture~#2#1#3}
\crefname{figure}{Figure}{Figure}
\crefformat{figure}{Figure~#2#1#3}

\crefname{appendix}{Appendix}{Appendices}
\crefformat{appendix}{\S#2#1#3}
\crefformat{enumi}{#2#1#3}

\crefname{equation}{}{}

\usepackage{enumitem}
\usepackage{youngtab}
\usepackage{extpfeil}
\usepackage{multirow,bigdelim}
\usepackage{array}

\renewcommand{\thesubsection}{\thesection(\@roman\c@subsection)}
\newenvironment{aenume}{%
  \begin{enumerate}%
  }{\end{enumerate}}
\newcounter{number}
\makeatother
\newtheorem{Theorem}[equation]{Theorem}
\newtheorem{Corollary}[equation]{Corollary}
\newtheorem{Lemma}[equation]{Lemma}
\newtheorem{Proposition}[equation]{Proposition}

\theoremstyle{definition}

\newtheorem{Conjecture}[equation]{Conjecture}

\theoremstyle{remark}
\newtheorem{Remark}[equation]{Remark}

\newtheorem*{Claim}{Claim}

\numberwithin{equation}{section}

\newcommand{\defeq}{\overset{\operatorname{\scriptstyle def.}}{=}}
\newcommand{\CC}{{\mathbb C}}
\newcommand{\ZZ}{{\mathbb Z}}
\newcommand{\QQ}{{\mathbb Q}}
\newcommand{\RR}{{\mathbb R}}
\newcommand{\proj}{{\mathbb P}}

\newcommand{\SL}{\operatorname{\rm SL}}

\newcommand{\GL}{\operatorname{GL}}

\newcommand{\algsl}{\operatorname{\mathfrak{sl}}} 

\newcommand{\gl}{\operatorname{\mathfrak{gl}}}

\newcommand{\Spec}{\operatorname{Spec}\nolimits}

\newcommand{\End}{\operatorname{End}}
\newcommand{\Hom}{\operatorname{Hom}}

\newcommand{\Ker}{\operatorname{Ker}}

\newcommand{\Ima}{\operatorname{Im}}

\newcommand{\codim}{\mathop{\text{\rm codim}}\nolimits}

\newcommand{\tr}{\operatorname{tr}}

\newcommand{\id}{\operatorname{id}}

\renewcommand{\MR}[1]{}

\newcommand{\vin}[1]{\operatorname{i}(#1)} 
\newcommand{\vout}[1]{\operatorname{o}(#1)} 

\newcommand{\bN}{\mathbf N}

\newcommand{\tslabar}{\mathbin{
\setbox0=\hbox{/\!\!/\!\!/}\rule[0.4\ht0]{\wd0}{.3\dp0}\kern-\wd0\box0}}

\newcommand{\aff}{\mathrm{aff}}

\newcommand{\Gr}{\mathrm{Gr}}

\newcommand{\cR}{\mathcal R}

\newcommand{\cK}{\mathcal K}
\newcommand{\cO}{\mathcal O}

\makeatletter
\newcommand{\cA}[1][{}]{%
  \@ifmtarg{#1}%
  {\mathcal A}
  {\mathcal A(#1)}
}
\newcommand{\cAh}[1][{}]{%
  \@ifmtarg{#1}%
  {\mathcal A_\hbar}
  {\mathcal A_\hbar(#1)}
}
\makeatother

\newcommand{\fA}{\mathfrak A}

\newcommand{\po}{\ar@{}[dr]|{\text{\pigpenfont R}}}
\newcommand{\pb}{\ar@{}[dr]|{\text{\pigpenfont J}}}
\newcommand{\pp}{\ar@{}[dr]|{\text{\pigpenfont P}}}

\newcommand{\cM}{\mathcal M}

\newcommand{\BA}{{\mathbb{A}}}

\newcommand{\CG}{{\mathscr{G}}}

\newcommand{\oW}{\overline{\mathcal{W}}{}}
\newcommand{\CZ}{{\mathscr{Z}}}

\newcommand{\fg}{{\mathfrak{g}}}

\newcommand{\fM}{{\mathfrak{M}}}

\newcommand{\tn}{{\mathrm{new}}}
\newcommand{\xl}{{\color{purple}
x}}
\newcommand{\xp}{{
\sigma}}
\newcommand{\bv}{{\mathbf v}} 
\newcommand{\bw}{{\mathbf w}} 

\newcommand{\GV}{\mathcal G}

\newcommand{\hf}{\hfill}
\makeatletter
\newdimen\y@inside
\y@inside=\y@boxdim
\advance\y@inside by -.2\y@linethick
\definecolor{halfgray}{gray}{0.5}
\newcommand{\rf}{\color{halfgray}
\rule[-.2\y@inside]{\y@inside}{\y@inside}
}
\makeatother
\newcommand{\graysquare}{\color{halfgray}\blacksquare}

\DeclareSymbolFont{symbolsC}{U}{pxsyc}{m}{n}
\SetSymbolFont{symbolsC}{bold}{U}{pxsyc}{bx}{n}
\DeclareFontSubstitution{U}{pxsyc}{m}{n}
\DeclareMathSymbol{\medcirc}{\mathbin}{symbolsC}{7}
\newcommand{\circfilled}{\tikz\draw[gray,fill=gray] (0,0) circle (.7ex);}

\newcommand{\vG}{{G^\vee}}
\newcommand{\bG}{\mathbf G}

\begin{document}

\title[Cherkis bow varieties and affine Lie algebras of type $A$]{
Towards geometric Satake correspondence for Kac-Moody algebras
-- Cherkis bow varieties and affine Lie algebras of type $A$
}
\author[H.~Nakajima]{Hiraku Nakajima}
\address{Research Institute for Mathematical Sciences,
Kyoto University, Kyoto 606-8502,
Japan}
\address{Kavli Institute for the Physics and Mathematics of the Universe (WPI),
  The University of Tokyo,
  5-1-5 Kashiwanoha, Kashiwa, Chiba, 277-8583,
  Japan
}
\email{hiraku.nakajima@ipmu.jp}

\subjclass[2000]{}
\begin{abstract}
  We give a provisional construction of the Kac-Moody Lie algebra
  module structure on the hyperbolic restriction of the intersection
  cohomology complex of the Coulomb branch of a framed quiver gauge
  theory, as a refinement of the conjectural geometric Satake
  correspondence for Kac-Moody algebras proposed in an earlier paper
  with Braverman, Finkelberg in 2019. This construction assumes several
  geometric properties of the Coulomb branch under the torus
  action. These properties are checked in affine type A, via the
  identification of the Coulomb branch with a Cherkis bow variety
  established in a joint work with Takayama.
\end{abstract}

\maketitle

\setcounter{tocdepth}{2}

\section*{Introduction}

Let $Q = (Q_0,Q_1)$ be a quiver without edge loops and
$\mathfrak g_{\mathrm{KM}}$ be the corresponding symmetric Kac-Moody
Lie algebra.
Let $\cM(\lambda,\mu)$ be the Coulomb branch of the framed quiver
gauge theory associated with dimension vectors specified by a dominant
weight $\lambda$ and a weight $\mu$ with $\mu\le\lambda$, defined as
an affine algebraic variety by Braverman, Finkelberg, and the author
\cite{2016arXiv160103586B}. (See also the earlier paper
\cite{2015arXiv150303676N} for motivation and references to physics
literature.)
In the subsequent paper \cite[\S3(viii)]{2016arXiv160403625B} it was
conjectured that there is a geometric construction of an integrable
highest weight $\mathfrak g_{\mathrm{KM}}$-module structure on the
direct sum (over $\mu$) via $\cM(\lambda,\mu)$:
Recall $\cM(\lambda,\mu)$ is equipped with an action of the torus
$T = (\CC^\times)^{Q_0}$, the Pontryagin dual of the fundamental group
of the gauge group, which is $\ZZ^{Q_0}$ in this case.
The conjecture in \cite{2016arXiv160403625B} refines and generalizes
the earlier conjecture by Braverman-Finkelberg \cite{braverman-2007},
which uses instanton moduli spaces on $\RR^4/(\ZZ/\ell\ZZ)$ in the
affine case. The Coulomb branch $\cM(\lambda,\mu)$ for an affine type
with dominant $\mu$ is conjectured to be instanton moduli
spaces. (This is proved for affine type $A$.)

Let $\Phi$ denote the hyperbolic restriction functor
(\cite{Braden,MR3200429}) with respect to a generic one parameter
subgroup in $T$. Let us apply it to the intersection cohomology
complexes $\mathrm{IC}$ of $\cM(\lambda,\mu)$ with coefficients in
$\QQ$. It is conjectured that $\Phi$ is hyperbolic semismall in the
sense of \cite[3.5.1]{2014arXiv1406.2381B}, and the fixed point set is
either empty or a single point. Hence
$\mathcal V_\mu(\lambda)\defeq \Phi(\mathrm{IC}(\cM(\lambda,\mu)))$ is
a vector space.
The main part of the conjecture states that
$\mathcal V(\lambda) = \bigoplus_\mu \mathcal V_\mu(\lambda)$ has a
structure of an integrable highest weight
$\mathfrak g_{\mathrm{KM}}$-module $V(\lambda)$ with the highest weight
$\lambda$ so that $\mathcal V_\mu(\lambda)$ is a weight space with
weight $\mu$.
It is regarded as the geometric Satake correspondence for the
Kac-Moody Lie algebra $\mathfrak g_{\mathrm{KM}}$, as a generalization
of the usual geometric Satake for a finite dimensional complex
reductive group due to Lusztig, Ginzburg, Beilinson-Drinfeld and
Mirkovi\'c-Vilonen
\cite{Lus-ast,1995alg.geom.11007G,Beilinson-Drinfeld,MV2}. (See also
\cite{fnkl_icm} for a review of the conjecture.)
In particular, the hyperbolic restriction functor was used in
Mikovi\'c-Vilonen's realization of representations. Ginzburg's
realization of $V(\lambda)$ by the cohomology over the affine
Grassmannian has no analog in general $\fg_{\mathrm{KM}}$ setting at
the moment when this paper is written.

In this paper, we give a provisional construction of the
$\mathfrak g_{\mathrm{KM}}$-module structure, assuming several
geometric properties of $\cM(\lambda,\mu)$. This is a refinement of
the conjecture in \cite{2016arXiv160403625B}, as well as its
supporting evidence since these geometric properties are technical in
nature, and not mysterious unlike the
$\mathfrak g_{\mathrm{KM}}$-module structure.
We then check the properties when $\mathfrak g_{\mathrm{KM}}$ is
of affine type $A$, using the identification of relevant Coulomb
branches with Cherkis bow varieties proved by Takayama and the author
\cite{2016arXiv160602002N}.
Bow varieties are symplectic reduction, and easier to handle than
Coulomb branches. We also use Hanany-Witten transition (see
\cref{subsec:HWtransition}) at various points. It is an isomorphism
between two bow varieties. This technique is useful by the following
reason.
When a bow variety satisfies a balanced condition (see 
\cref{subsec:coulomb-branch}), it is isomorphic to Coulomb
branch $\cM(\lambda,\mu)$.
A fixed point component of a balanced bow variety is another bow
variety, but it does not necessarily satisfy the balanced
condition. We then show that it is isomorphic to one with the balanced
condition by Hanany-Witten transition.

The idea of the construction is simple. The
$\mathfrak g_{\mathrm{KM}}$-structure should be compatible with
restriction to a Levi subalgebra $\mathfrak l$, and realized by the
hyperbolic restriction functor with respect to a one parameter
subgroup $\chi_{\mathfrak l}$ corresponding to the Levi
subalgebra. When the one parameter subgroup is generic, the Levi
subalgebra is Cartan, and we recover the above construction. This
compatibility is well-known for the usual geometric Satake
correspondence and is a key ingredient of the construction. Therefore
we define operators $e_i$, $f_i$, $h_i$ corresponding to $i\in Q_0$ by
using the hyperbolic restriction with respect to $\chi_i$ for the Levi
subalgebra $\mathfrak l_i$ and the reduction to the case $A_1$. It is
easy to prove the conjecture in the $A_1$ case. The check of the
defining relations on $e_i$, $f_i$, $h_i$, say $[e_i, f_j] = 0$ for
$i\neq j$, is reduced to rank $2$ cases. By considering tensor
products as explained below, it is enough to check them when $\lambda$
is a fundamental weight. For $\algsl(3)$ relevant bow varieties are
affine spaces, and we check them by direct computation. We also
realize the embedding
$\widehat{\algsl}(n)\to \widehat{\mathfrak{gl}}(\infty)$ by a variant
of a family $\underline{\cM}(\lambda,\mu)$ below. This argument covers
the case $\widehat{\algsl}(2)$. Since we consider only affine types,
these are enough.

Unlike in \cite{MV2} we take $\QQ$ as field of coefficients. We believe that
some of the arguments survive even in positive characteristic, but we
leave the study for the future.

Suppose that $Q$ is of finite type, and hence
$\mathfrak g_{\mathrm{KM}}$ is a finite dimensional complex simple Lie
algebra $\mathfrak g$ of type $ADE$. Then $\cM(\lambda,\mu)$ is
isomorphic to a transversal slice to an orbit in the closure of
another orbit in the affine Grassmannian when $\mu$ is dominant
\cite{2016arXiv160403625B}. The group $G^\vee$ for the affine
Grassmannian is Langlands dual to the simply-connected $G$ with Lie
algebra $\mathfrak g$. (Hence representations of $G$ are nothing but
representations of $\mathfrak g$.)
This is one of the reasons why we expect the geometric Satake
correspondence for $\mathfrak g_{\mathrm{KM}}$ via $\cM(\lambda,\mu)$.
Moreover the hyperbolic restriction
$\Phi(\mathrm{IC}(\cM(\lambda,\mu)))$ is naturally identified with one
in the affine Grassmannian even for non-dominant $\mu$ by a recent
result by Krylov \cite{2017arXiv170900391K}.
Therefore the $\mathfrak{g}$-module structure is induced from the
usual geometric Satake correspondence.

From this point of view, our construction above resembles the
definition of Kashiwara crystal structure on the set of irreducible
components of Mirkovi\'c-Vilonen cycles by Braverman-Gaitsgory
\cite{bragai}. It is also similar to Vasserot's construction
\cite{MR1895920} of a $\mathfrak g$-module structure.
The main difference between these constructions and ours is a
construction of an isomorphism between multiplicity spaces appearing
in the hyperbolic restriction with respect to $\chi_i$, which we will
explain in \cref{subsec:new_conj}. In our case, the isomorphism is
given by the factorization property of Coulomb branches. This
isomorphism comes for free or is unnecessary in the usual setting
\cite{bragai,MR1895920}.
%
See \cref{subsec:Krylov} for a comparison between our construction and
the usual one.

After the author gave a talk on this work at Sydney, B.~Webster
explained to him an approach to a construction of a
$\mathfrak g_{\mathrm{KM}}$-module structure via symplectic
duality. It is not clear to the author that how much can be said in
this approach at the time this paper is written. The construction in
this paper is nothing to do with the symplectic dual side, which is a
quiver variety, where a geometric construction of
$\fg_{\mathrm{KM}}$-modules was given in \cite{Na-quiver,Na-alg}. See
\cite[\S3(viii)]{2016arXiv160403625B} for parallel explanation of two
constructions.

After the first version of this paper was written, the author and
Weekes generalize the definition of the Coulomb branch of a framed
quiver gauge theory to a quiver with \emph{symmetrizer}
\cite{2019arXiv190706552N}. Our geometric Satake correspondence is
naturally generalized to cover the case when $\fg_{\mathrm{KM}}$ is
\emph{symmetrizable} but not necessarily symmetric. The type of the
quiver with symmetrizer is Langlands dual $\fg_{\mathrm{KM}}^\vee$ as
in the usual geometric Satake correspondence.

The paper is organized as follows. In \cref{sec:conj} we formulate
conjectures on geometric properties of Coulomb branches under the
torus action. In \cref{sec:weights-affine-lie} we fix notation for
weights of affine Lie algebras.
In \cref{sec:bow} we review the quiver description and important
properties of bow varieties studied in \cite{2016arXiv160602002N}.
\cref{sec:torus} is the heart of this paper and is devoted to study of
torus action on bow varieties.
In \cref{sec:construction} we use results in \cref{sec:torus} to
define a $\mathfrak{g}_{\mathrm{KM}}$ structure on the hyperbolic
restriction for affine type $A$.
In \cref{sec:Maya} we parametrize torus fixed points in bow varieties
when they are smooth. Fixed points are in bijection to Maya diagrams
which appear in the infinite wedge space.

\subsection*{Notation}

The symmetric group of $n$ letters is denoted by $\mathfrak S_n$.

Let $J_k$
\begin{NB}
(resp. $e_k$)    
\end{NB}%
denote the regular nilpotent Jordan matrix of
size $k$
\begin{NB}
 (resp.\ the $k$-th coordinate vector of size $k$)    
\end{NB}%
:
\begin{equation*}
  J_k = 
  \begin{bmatrix}
    0 & 1 & 0 & \dots & 0 \\
    0 & 0 & 1 & \dots & 0 \\
    \vdots && \ddots & \ddots & 0 \\
    0 && \dots & 0 & 1\\
    0 && \dots && 0
  \end{bmatrix}
  \begin{NB}
  ,
  \qquad
  e_k =   \begin{bmatrix}
    0 \\ \vdots \\ 0 \\ 1
  \end{bmatrix}
  \end{NB}%
.
\end{equation*}

For an irreducible algebraic variety $X$ we denote by $\mathrm{IC}(X)$
its intersection cohomology complex associated with the trivial rank
$1$ local system on its regular locus with rational coefficients.

\subsection*{Acknowledgments}

The author thanks A.~Braverman, M.~Finkelberg, and D.~Muthiah for
useful discussion. He also thanks the referees for their careful reading.

The result of this paper was reported at several places including
the Third Pacific Rim Mathematical Association (PRIMA) Congress 
in August 2017,
the Banff International Research Station in August 2017,
the University of Sydney in December 2017,
and Korea Institute for Advanced Study in December 2017.
The author is grateful to the hospitality of institutes and organizers
of workshops.

The research of HN was supported in part by the World Premier
International Research Center Initiative (WPI Initiative), MEXT,
Japan, and by JSPS Grant Number 16H06335.

\section{Conjectures}\label{sec:conj}

\subsection{Earlier conjectures}

We first recall conjectures from \cite{2016arXiv160403625B}.

Let $Q=(Q_0,Q_1)$ and $\fg_{\mathrm{KM}}$ be as in Introduction.
Let $\cM(\lambda,\mu)$ be the Coulomb branch associated with a
dominant weight $\lambda$ and a weight $\mu$ with $\mu\le\lambda$.
It has a structure of a Poisson variety such that its regular locus is
symplectic (see \cite[\S3(iv) and
Prop.~6.15]{2016arXiv160103586B}). It is conjectured that there are
only finitely many symplectic leaves. This conjecture is known when
$\fg_{\mathrm{KM}}$ is affine type $A$. See \cref{thm:stratification}
below. We assume it from now, and we call symplectic leaves
\emph{strata} of $\cM(\lambda,\mu)$.

Let $T$ be the torus acting on $\cM(\lambda,\mu)$ as in
Introduction.
Let us take a generic 1-parameter family $\chi\colon\CC^\times\to T$
and consider a diagram
\begin{equation*}
  \mathrm{pt} = \cM(\lambda,\mu)^T \xleftarrow{p}\fA_\chi(\lambda,\mu)
  \xrightarrow{j} \cM(\lambda,\mu),
\end{equation*}
where $\fA_\chi(\lambda,\mu)$ is the attracting set with respect to
$\chi$. When there is no fear of confusion, we denote it simply by $\fA$.
Here $j$ is the inclusion, and $p$ is the map given by taking the
limit $\rho(t)$ for $t\to 0$.
The hyperbolic restriction functor $\Phi = p_*j^!$ is defined on the
derived category of equivariant constructible sheaves.

\begin{Conjecture}[\protect{\cite[Conj.~3.25]{2016arXiv160403625B}}]
  \label{conj:old}
  (1)
  $\cM(\lambda,\mu)^{T} = \cM(\lambda,\mu)^{\chi(\CC^\times)}$
  is either a single point or empty.
  
  (2) Intersections of $\fA_\chi(\lambda,\mu)$ with symplectic leaves
  of $\cM(\lambda,\mu)$ are lagrangian. Hence $\Phi$ is hyperbolic
  semi-small for the intersection cohomology complex
  $\mathrm{IC}(\cM(\lambda,\mu))$. In particular,
  $\mathcal V_\mu(\lambda) \defeq \Phi(\mathrm{IC}(\cM(\lambda,\mu)))$
  is a vector space.

  (3) The direct sum
  $\mathcal V(\lambda)=\bigoplus \mathcal V_\mu(\lambda)$ has a
  structure of an integrable highest weight
  $\mathfrak g_{\mathrm{KM}}$-module $V(\lambda)$.
\end{Conjecture}

\subsection{New conjectures}\label{subsec:new_conj}

Let us introduce several notations in order to state conjectural
geometric properties and the construction of the
$\mathfrak g_{\mathrm{KM}}$-module structure in more detail.

Set
${|\underline\bv|} = \sum_{i\in Q_0} \bv_i$ and
$\BA^{\underline\bv} = \BA^{|\underline\bv|}/ \prod_{i\in Q_0} \mathfrak
S_{\bv_i}$, where $\lambda - \mu = \sum_i \bv_i \alpha_i$ with simple
roots $\alpha_i$. We consider $\BA^{\underline\bv}$ as the
configuration space of $Q_0$-colored points in $\BA$.
We have the factorization morphism
$\varpi\colon\cM(\lambda,\mu)\to \BA^{\underline\bv}$
(\cite[(3.17)]{2016arXiv160103586B}). This was denoted by $\Psi$ in
the context of bow varieties \cite{2016arXiv160602002N}, and played
fundamental roles in the analysis of Coulomb branches and their
identification with bow varieties. In particular, it enjoys the
factorization property that says $\cM(\lambda,\mu)$ factorizes over an
open subset of disjoint configurations. (See \cref{subsec:factor} for
a brief review.)

Let $w_{i,r}$, $\mathsf y_{i,r}$ ($i\in Q_0$, $r=1,\dots,\bv_i$) be
regular functions on
$\cM(\lambda,\mu)\times_{\BA^{\underline\bv}} \BA^{|\underline{\bv}|}$
introduced in \cite[\S3(iii)]{2016arXiv160403625B} ($\mathsf y_{i,r}$
was denoted by $\overline{\mathsf y}_{i,r}$ there) and
\cite[\S6.8.1]{2016arXiv160602002N}. They induce a birational
isomorphism $
\begin{tikzcd}
  \cM(\lambda,\mu)\times_{\BA^{\underline\bv}} \BA^{|\underline{\bv}|}
  \arrow[dashed]{r}{\approx}
  & (\BA\times \mathbb G_m)^{|\underline{\bv}|}
\end{tikzcd}
$. The factorization morphism $\varpi$ corresponds to the projection
to the first factor $\BA$. In other words, it is given by
$\left((w_{i,r})_r\bmod \mathfrak S_{\bv_i}\right)_i$.
See \cref{subsec:coord} for the definition of $w_{i,r}$,
$\mathsf y_{i,r}$ in terms of bow varieties.

We take the one parameter subgroup $\chi$ of $T$ from the `negative'
Weyl chamber, i.e., $\chi(t) = (t^{m_j})_{j\in Q_0}$ with $m_j < 0$
for all $j\in Q_0$. We consider the corresponding hyperbolic
restriction functor $\Phi$ with respect to $\chi$. Choosing
$i\in Q_0$, we take another one parameter subgroup $\chi_i$ so that
$\chi_i(t) = (t^{m_j})$ with $m_i = 0$, $m_j < 0$ for $j\neq i$. This
$\chi_i$ lives at the boundary of the chamber containing $\chi$. We
then consider the fixed point set $\cM(\lambda,\mu)^{\chi_i}$ with
respect to $\chi_i$.

\begin{Conjecture}\label{conj:1}
  \textup{(1)} The fixed point set $\cM(\lambda,\mu)^{\chi_i}$ is
  either empty or isomorphic to a Coulomb branch
  $\cM_{A_1}(\lambda',\mu')$ of an $A_1$ type framed quiver gauge
  theory with weights $\lambda'$, $\mu'$, where
  $\mu' = \langle \mu, h_i\rangle$.

  \textup{(2)} The intersection of $\cM(\lambda,\mu)^{\chi_i}$ with a
  stratum is either empty or a stratum
  $\cM_{A_1}^{\mathrm{s}}(\kappa',\mu')$ of
  $\cM_{A_1}(\lambda',\mu')$. (Here
  $\cM_{A_1}^{\mathrm{s}}(\kappa',\mu')$ is the smooth locus of
  $\cM_{A_1}(\kappa',\mu')$. The stratification of
  $\cM_{A_1}(\lambda',\mu')$ was determined in
  \cite[\S7.5]{2016arXiv160602002N}.)
  \begin{NB}
    The following follows from the description of strata in
    \cite[\S7.5]{2016arXiv160602002N} for finite type $A$:

    Moreover the factorization morphism of $\cM_{A_1}(\lambda',\mu')$
    is restricted to that of $\cM_{A_1}^{\mathrm{s}}(\kappa',\mu')$.
  \end{NB}%

  \textup{(3)} The restrictions $w_{j,r}$,
  $\mathsf y_{j,r}|_{\cM(\lambda,\mu)^{\chi_i}}$ are zero for
  $j\neq i$, and are equal to the corresponding functions on
  $\cM_{A_1}(\lambda',\mu')$ or zero for $j=i$.
  In particular, the restriction of the $i$-th component of the
  factorization morphism $\varpi$ of $\cM(\lambda,\mu)$ to
  $\cM_{A_1}(\lambda',\mu')$ is equal to the factorization morphism of
  $\cM_{A_1}(\lambda',\mu')$ up to adding $0$.
\end{Conjecture}

This conjecture will be shown for affine type $A$ in \cref{thm:smaller}. 

The condition (3) uniquely fixes the isomorphism
$\cM(\lambda,\mu)^{\chi_i}\cong \cM_{A_1}(\lambda',\mu')$ since the
restrictions of $w_{i,r}$, $\mathsf y_{i,r}$ give the birational
coordinates on $\cM_{A_1}(\lambda',\mu')$.

Once we prove that $\mathcal V(\lambda)\cong V(\lambda)$, $\lambda'$
is determined as follows. Let $\mathfrak l_i$ be the $i$-th Levi
subalgebra of $\fg_{\mathrm{KM}}$. We consider $\mathfrak l_i$-modules
with highest weight $\mu+\bv\alpha_i$ ($\bv\in\ZZ_{\ge 0}$) which
appear in the restriction of the integrable highest weight module
$V(\lambda)$. Then $(\lambda'-\mu')/2$ is the maximum among such
$\bv$.
When \cref{conj:1} will be discussed, this would not be clear. See the
proof of \cref{thm:smaller}.

\begin{Remark}\label{rem:char}
  After the first version of this paper was written,
  we find that the above $\lambda'=\mu'+2\bv$ is characterized as
  \begin{equation}\label{eq:h:2}
    \cM(\lambda,\mu+\bv\alpha_i)^T \neq \emptyset,\quad
    \cM(\lambda,\mu+(\bv+1)\alpha_i)^T = \emptyset
  \end{equation}
  during a discussion with D.~Muthiah. This condition is phrased in
  terms of Coulomb branches, and is equivalent to the above
  representation theoretic one by \cref{conj:old}(1).
  The condition \eqref{eq:h:2} is checked for affine type
  $A$. See \cref{rem:char2}.
\end{Remark}

Since $\chi_i$ lives in the boundary of a chamber containing $\chi$,
the hyperbolic restriction $\Phi$ factors as
$\Phi = \Phi^i\circ \Phi_i$. Here $\Phi_i$ is the hyperbolic
restriction with respect to $\chi_i$, and $\Phi^i$ is the hyperbolic
restriction with respect to $\chi$, restricted to the fixed point set
$\cM(\lambda,\mu)^{\chi_i}$. Assuming \cref{conj:old}(2) (and also
that for type $A_1$, which was already proved in
\cite[Prop.~7.33]{2016arXiv160602002N}), we see that $\Phi_i$ is
hyperbolic semismall in the sense of
\cite[3.5.1]{2014arXiv1406.2381B}, hence it sends
$\mathrm{IC}(\cM(\lambda,\mu))$ to a semisimple perverse sheaf. See
the argument in \cite[App.~A]{2014arXiv1406.2381B}.

We further conjecture
\begin{Conjecture}\label{conj:2}
  $\Phi_i(\mathrm{IC}(\cM(\lambda,\mu)))$ is a direct sum of
  $\mathrm{IC}(\cM_{A_1}(\kappa',\mu'))$ with various $\kappa'$ with
  $\mu'\le\kappa'\le\lambda'$.
\end{Conjecture}

This conjecture means that we do not have \emph{nontrivial} local
systems on the regular locus of $\cM_{A_1}(\kappa',\mu')$. For affine
type $A$, it follows from the smoothness of deformation of
$\cM(\lambda,\mu)$, and compatibility between \cref{conj:1} and the
deformation of $\cM(\lambda,\mu)$ explained later. See the proof of
\cref{lem:nontrivial}.
By \cref{conj:2}
\begin{equation}\label{eq:11}
  \Phi_i 
  (\mathrm{IC}(\cM(\lambda,\mu)))
  \cong\bigoplus_{\kappa'} M^{\lambda,\mu}_{\kappa',\mu'}\otimes
  \mathrm{IC}(\cM_{A_1}(\kappa',\mu'))
\end{equation}
for vector spaces $M^{\lambda,\mu}_{\kappa',\mu'}$, called
\emph{multiplicity spaces}. Hence we also deduce
\begin{equation*}
  \mathcal V_\mu(\lambda) =
  \Phi(\mathrm{IC}(\cM(\lambda,\mu)))
  \cong \bigoplus_{\kappa'} M^{\lambda,\mu}_{\kappa',\mu'} \otimes
  \Phi^i 
  (\mathrm{IC}(\cM_{A_1}(\kappa',\mu')))
\end{equation*}
as $\Phi = \Phi^i\circ \Phi_i$.

The factor $\Phi^i 
(\mathrm{IC}(\cM_{A_1}(\kappa',\mu')))$ is
$\mathcal V_{\mu'}(\kappa')$ for the finite $A_1$ case. Therefore it
should be the weight space of a finite dimensional irreducible
$\algsl(2)$-module. We indeed construct operators
$\Phi^i(\mathrm{IC}(\cM_{A_1}(\kappa',\mu'))) \xtofrom[e]{f}
\Phi^i(\mathrm{IC}(\cM_{A_1}(\kappa',\mu'-2)))$ in \cref{thm:A1}.

The factorization property of the Coulomb branch says that there is a
$\chi_i$-equivariant isomorphism between open subsets of
$\cM(\lambda,\mu)\times(\BA\times\mathbb G_m)$ and
$\cM(\lambda,\mu-\alpha_i)$ after a base change. See \cref{thm:fact}.
Here $\chi_i$ acts trivially on $\BA\times\mathbb G_m$.
The multiplicity spaces $M^{\lambda,\mu}_{\kappa',\mu'}$,
$M^{\lambda,\mu-\alpha_i}_{\kappa',\mu'-2}$ are determined by
restriction to open subsets. Therefore we have an isomorphism
$M^{\lambda,\mu}_{\kappa',\mu'} \cong
M^{\lambda,\mu-\alpha_i}_{\kappa',\mu'-2}$. (See
\cref{prop:naturaliso}.)
We define operators $e_i$, $f_i$ 
\begin{equation*}
  \mathcal V_\mu(\lambda) \xtofrom[e_i]{f_i}
  \mathcal V_{\mu-\alpha_i}(\lambda),
\end{equation*}
as $(\text{the above isomorphism})\otimes (e, f \text{ for $A_1$
  case})$.

The multiplicity space $M^{\lambda,\mu}_{\kappa',\mu'}$ is
$\Hom_{\mathfrak l_i}(V_{\mathfrak l_i}(\tilde{\kappa}'), \mathcal
V(\lambda))$ once $\mathcal V(\lambda)$ is equipped with
an $\mathfrak l_i$-module as above, where
$\tilde{\kappa}' = \mu + (\kappa'-\mu')\alpha_i/2$. Here
$V_{\mathfrak l_i}(\tilde{\kappa}')$ is a finite dimensional
irreducible representation of $\mathfrak l_i$ with highest weight
$\tilde{\kappa}'$. Note $\tilde{\kappa}'$ and hence
$\Hom_{\mathfrak l_i}(V_{\mathfrak l_i}(\tilde{\kappa}'), \mathcal
V(\lambda))$ are unchanged under simultaneous shifts
$\mu\to \mu - \alpha_i$, $\mu' \to \mu' - 2$. Thus the construction of
the isomorphism
$M^{\lambda,\mu}_{\kappa',\mu'}\cong
M^{\lambda,\mu-\alpha_i}_{\kappa',\mu'-2}$ is crucial in the
definition of a $\fg_{\mathrm{KM}}$-module structure on
$\mathcal V(\lambda)$.
This isomorphism is obscure in the usual geometric Satake
correspondence, as we mentioned in the Introduction.

  


\subsection{Tensor products}

Next, we consider the realization of tensor products in this framework. We
recall \cite[Conj.~3.27]{2016arXiv160403625B} and give its refinement
as in the previous subsection.

Take a decomposition $\lambda = \lambda^1+\lambda^2$ into a sum of two
dominant weights $\lambda^1$, $\lambda^2$. Then it gives a one
parameter subgroup in the flavor symmetry group of the quiver gauge
theory. (See \cite[\S3(viii)]{2016arXiv160103586B}.)
This gives rise a family $\underline{\cM}(\lambda,\mu)\to \BA^1$
parametrized by the affine line $\BA^1$ together with
$\pi\colon\underline{\widetilde\cM}(\lambda,\mu)\to\underline{\cM}(\lambda,\mu)$
which is expected to be a small birational morphism, and the second
family $\underline{\widetilde\cM}(\lambda,\mu)\to\BA^1$ is
topologically trivial over $\BA^1$.
We further conjecture that the fixed point set
$\underline{\cM}(\lambda,\mu)^\chi$ is a union of finitely many copies
of $\BA^1$, glued at the origin such that components correspond, in
bijection, to a decomposition $\mu = \mu^1 + \mu^2$ with
$\mathcal V_{\mu^1}(\lambda^1)$, $\mathcal V_{\mu^2}(\lambda^2)\neq 0$.
Finally, we conjecture that there are isomorphisms
\begin{equation}\label{eq:14}
  \Phi\circ \pi_*(\mathrm{IC}(\widetilde{\cM}(\lambda,\mu)))
 \cong \psi\circ\Phi(\mathrm{IC}(\underline{\cM}(\lambda,\mu)))
    \cong  \bigoplus_{\mu=\mu^1+\mu^2} \mathcal V_{\mu^1}(\lambda^1)
    \otimes \mathcal V_{\mu^2}(\lambda^2),
\end{equation}
where $\widetilde{\cM}(\lambda,\mu)$ is the fiber of
$\underline{\widetilde\cM}(\lambda,\mu)$ over $0$, and $\psi$ is the
nearby cycle functor. The first isomorphism is a consequence of the
triviality of $\underline{\widetilde\cM}(\lambda,\mu)\to\BA^1$ and the
commutativity of the nearby cycle and hyperbolic restriction functors.
(See \cref{rem:comm} below for a comment on the commutativity.)

Since the second isomorphism in \eqref{eq:14} was stated in
\cite[Conj.~3.27(2)]{2016arXiv160403625B} without reason, let us give
explanation. Take a fiber $\cM^{\nu^{\bullet,\CC}}(\lambda,\mu)$ of
$\underline{\cM}(\lambda,\mu)$ over $1\in\BA^1$. The parameter
$\nu^{\bullet,\CC}$ is determined by a nonzero complex number
$\dot\nu$ which corresponds to the second weight $\lambda^2$, while
$0$ corresponds to $\lambda^1$. See \cref{subsec:deformed-case} for
detail. Then
\begin{Conjecture}
  \textup{(1)} The fixed point set
  $\cM^{\nu^{\bullet,\CC}}(\lambda,\mu)^\chi$ is bijective to
  $\bigsqcup_{\mu^1+\mu^2 = \mu}
  \cM(\lambda^1,\mu^1)^\chi\times\cM(\lambda^2,\mu^2)^\chi$. (The
  bijection realizes the above bijection between irreducible components
  of $\underline{\cM}(\lambda,\mu)^\chi$ and decomposition.)
  
  \textup{(2)} The image of a fixed point in
  $\cM^{\nu^{\bullet,\CC}}(\lambda,\mu)^\chi$ under the factorization
  morphism $\varpi$ is supported at $0$ and $\dot\nu$. The
  factorization gives an isomorphism
  $\cM^{\nu^{\bullet,\CC}}(\lambda,\mu)$ and
  $\cM(\lambda^1,\mu^1)\times\cM(\lambda^2,\mu^2)$ in neighborhoods of
  fixed points.
\end{Conjecture}

The factorization isomorphism in (2) gives the second isomorphism in
\eqref{eq:14}. We will prove this conjecture in \cref{prop:deformed}
for affine type $A$.

We thus have
\begin{equation}\label{eq:h:5}
  \mathcal V(\lambda) = \bigoplus_\mu \Phi(\mathrm{IC}(\cM(\lambda,\mu)))\to
  \bigoplus_\mu \Phi\circ\pi_*(\mathrm{IC}(\widetilde{\cM}(\lambda,\mu)))
  \xrightarrow[\eqref{eq:14}]{\cong}
  \mathcal V(\lambda^1)\otimes\mathcal V(\lambda^2),
\end{equation}
where the middle arrow is induced from the inclusion of the direct
summand
$\mathrm{IC}(\cM(\lambda,\mu))\hookrightarrow
\pi_*(\mathrm{IC}(\widetilde{\cM}(\lambda,\mu)))$. The composite is
conjectured (\cite[Conj.~3.27(3)]{2016arXiv160403625B}) to be a
homomorphism of $\mathfrak g_{\mathrm{KM}}$-modules.

We refine this conjecture as follows. We apply the above construction
of the $\mathfrak g_{\mathrm{KM}}$-module structure via $\chi_i$ to
$\pi_*(\mathrm{IC}(\widetilde{\cM}(\lambda,\mu)))$. The construction
is compatible with the inclusion of the direct summand, hence the
middle arrow in \eqref{eq:h:5} is a
$\mathfrak g_{\mathrm{KM}}$-homomorphism.

We then prove that $\xrightarrow[\eqref{eq:14}]{\cong}$ in
\eqref{eq:h:5} is a $\mathfrak g_{\mathrm{KM}}$-homomorphism by a
reduction to the $A_1$ case, which is easy to show by direct
computation.
The $A_1$ reduction is guaranteed by a geometric property of the fixed
point set in the deformed case, as will be shown in
\cref{prop:tensor}:
\begin{Conjecture}
  Let us take $\chi_i$ as above. The fixed point set
  $\cM^{\nu^{\bullet,\RR}}(\lambda,\mu)^{\chi_i}$ is a union of $A_1$
  type Coulomb branches where the parameter is induced from
  the original parameter $\nu^{\bullet,\RR}$.
\end{Conjecture}

We do not have a good understanding of what we mean by the induced
parameter. At this moment, we just say that the entries of new
parameters, once we forget multiplicities, are entries
$\nu^{\bullet,\RR}_h$ of the original parameter.

We check this for affine type $A$ in \cref{lem:perm}.

\begin{Remark}\label{rem:comm}
  The reference for the commutativity was
  \cite[Prop.~5.4.1]{MR3752464} in
  \cite[Conj.~3.27(2)]{2016arXiv160403625B}. The last part of the
  proof of \cite[Prop.~5.4.1(2)]{MR3752464} works only the case when
  $\widetilde{\cM}(\lambda,\mu)$ is smooth. This is true for affine
  type A, but not in general. We need to use
  $\tilde c_X^* = \tilde c_X^!$. The commutativity, as well as the
  commutativity between the nearby cycle functor and the isomorphism
  $p_*j^! \cong p^-_! (j^-)^{*}$ are proved in \cite{MR3912059}. Here
  $p^-$, $j^-$ is the projection and the inclusion for the attracting
  set with respect to the opposite $1$-parameter subgroup
  $\chi^{-1}$. The isomorphism is a main theorem of \cite{Braden}, and
  used here as semisimplicity of
  $\Phi_i(\mathrm{IC}(\cM(\lambda,\mu)))$.
\end{Remark}

\section{Weights of affine Lie algebras}\label{sec:weights-affine-lie}

We fix our convention on weights of affine Lie algebras of type $A$ in
this section.

We denote the central extension of the loop Lie algebra of $\algsl(n)$
by $\widehat{\algsl}(n)$ while the affine Lie algebra containing the
degree operator $d$ is denoted by $\algsl(n)_\aff$. We also use
versions for $\gl(n)$, which are denoted by $\widehat{\gl}(n)$,
$\gl(n)_\aff$ respectively.

Let us take the Cartan subalgebra $\mathfrak h_{\gl(n)}$ of $\gl(n)$
as the space of diagonal matrices. The weight lattice $P_{\gl(n)}$ of
$\gl(n)$ is $\ZZ^n$, where the $i$-th coordinate vector $e_i$ is
$\mathfrak h_{\gl(n)}\to\CC$ given by taking the $i$-th diagonal entry
of $h\in\mathfrak h_{\gl(n)}$. The weight lattice $P_{\algsl(n)}$ of
$\algsl(n)$ is its quotient $\ZZ^n/\ZZ(1,\dots,1)$, considered as
$n$-tuples of integers $[\lambda_1,\dots,\lambda_n]$ up to
simultaneous shifts.
We let $\alpha_i \defeq e_i - e_{i+1} \bmod \ZZ[1,\dots,1]$, the
$i$-th simple root of $\algsl(n)$.
The $i$-th fundamental weight $\Lambda_i$ of $\algsl(n)$ is
$(\underbrace{1,\dots,1}_{\text{$i$
    times}},\underbrace{0,\dots,0}_{\text{$(n-i)$ times}}) \bmod
\ZZ[1,\dots,1]$.

\begin{NB}
  Then roots of $\algsl(n)$ are $e_i - e_j$ (${i\neq j}$). Take
  $\{ e_i - e_j\}_{i < j}$ as the set of positive roots. Then simple
  roots are $\{ \alpha_i \defeq e_i - e_{i+1}\}_{i=1,\dots,n-1}$. The
  Cartan subalgebra of $\algsl(n)$ is the subspace
  $\sum_{i=1}^n e_i = 0$, and roots, positive roots, and simple roots
  are the same.

  We set
  $h_i =
  \operatorname{diag}(0,\dots,0,\underset{i}{1},\underset{i+1}{-1},0,\dots,0)$
  ($i=1,\dots,n-1$). This is a base of $\mathfrak h_{\algsl(n)}$, and
  is equal to $\alpha_i^\vee$. We have
  $\langle h_i, \alpha_j\rangle = 2\delta_{ij} - \delta_{i,j-1} -
  \delta_{i,j+1}$. This is the Cartan matrix of type $A_{n-1}$.
 
  Let $\gl(n)\otimes\CC[z,z^{-1}]$ denote the loop Lie algebra of
  $\gl(n)$. Its central extension is
  $\widehat{\gl}(n) = \gl(n)\otimes\CC[z,z^{-1}]\oplus \CC c$ with
  \begin{equation*}
    [x\otimes z^k, y\otimes z^l] = [x,y]\otimes z^{k+l} + k\delta_{k,-l}(x,y)c,
  \end{equation*}
  where $(x,y)=\tr(xy)$. Here $[c, x\otimes z^k] = 0$, as $c$ is
  central. The degree operator $d$ is defined by
  $[d, x\otimes z^k] = k x\otimes z^k$, $[d,c] = 0$. We have
  $\gl(n)_\aff = \widehat{\gl}(n)\oplus\CC d$.

  The Cartan subalgebra of $\gl(n)_\aff$ is
  $\mathfrak h_{\gl(n)_\aff} = \mathfrak h_{\gl(n)} \oplus \CC c\oplus
  \CC d$. We extend $\alpha_i$ to $\mathfrak h_{\gl(n)_\aff}$ by $0$
  on $\CC c\oplus \CC d$. We then define
  $\delta\in \mathfrak h_{\gl(n)_\aff}^*$ by
  $\langle \mathfrak h_{\gl(n)}\oplus\CC c, \delta\rangle = 0$,
  $\langle d,\delta\rangle = 1$. We then set
  $\alpha_0 = \delta - \theta$, where $\theta$ is the highest weight,
  i.e., $\theta = e_1 - e_n$. Since
  $e_1 - e_n = \alpha_1 + \alpha_2 + \dots + \alpha_{n-1}$, we have
  $\delta = \alpha_0 + \alpha_1 + \dots + \alpha_{n-1}$.

  We have $h_0 = c - \theta^\vee\otimes 1$. See \cite[(7.4.2)]{Kac}.
\end{NB}%

We denote simple roots of $\algsl(n)_\aff$ by $\alpha_0$, \dots,
$\alpha_{n-1}$. Here the primitive positive imaginary root is
$\delta = \alpha_0+\cdots+\alpha_{n-1}$, hence
$\alpha_0 = \delta - (\alpha_1+\dots+\alpha_n)$.
We denote fundamental weights by $\Lambda_0$, \dots,
$\Lambda_{n-1}$. Our convention is $\langle d, \Lambda_i\rangle = 0$,
$\langle d, \alpha_i\rangle = \delta_{0i}$.
\begin{NB}
  See \cite[\S6.2, \S12.4]{Kac}. In the above convention, we have
  $\langle d,\alpha_i\rangle = 0$ if $i\neq 0$,
  $\langle d,\alpha_0\rangle = \langle d, \delta - \theta\rangle = 1$.

  Note also $\langle c,\alpha_i\rangle = 0$.
\end{NB}%
The weight `lattice' $P_{\algsl(n)_\aff}$ of $\algsl(n)_\aff$ is
$\bigoplus_{i=0}^{n-1} \ZZ \Lambda_i \oplus \CC\delta$. 
(This is not a lattice, but we keep this terminology.)
The weight lattice $P_{\widehat{\algsl}(n)}$ of $\widehat{\algsl}(n)$ is
identified with $\bigoplus_{i=0}^{n-1} \ZZ \Lambda_i$.
\begin{NB}
  Note
  $\langle \mathfrak h_{\algsl(n)}\oplus\CC c, \delta\rangle =
  0$. Hence $\delta$ is $0$ on the Cartan subalgebra of
  $\widehat{\algsl}(n)$.
\end{NB}%
The level of a weight $\lambda$ in $P_{\algsl(n)_\aff}$ (or
$P_{\widehat{\algsl}(n)}$) is $\langle c,\lambda\rangle$. If
$\lambda = \sum_{i=0}^{n-1} \bw_i\Lambda_i (+ a\delta)$, the level is
equal to $\sum_{i=0}^{n-1} \bw_i$. We often fix a level $\ell$, then
the set of level $\ell$ weights is identified with
$\ZZ^n/\ZZ[1,\dots,1]\times\CC$ (or $\ZZ^n/\ZZ[1,\dots,1]$), where
$\bw_i$ with $1\le i\le n$ is read off from $\ZZ^n/\ZZ[1,\dots,1]$,
and $\bw_0$ is given by $\ell - \sum_{i=1}^{n-1} \bw_i$. Namely
$[\lambda_1,\dots,\lambda_n]$ defines
$\bw_1 = \lambda_1 - \lambda_2$,\dots,
$\bw_{n-1} = \lambda_{n-1}-\lambda_n$.
\begin{NB}
And $\bw_0 = \ell - \sum_{i=1}^{n-1}\bw_i = \ell - \lambda_1 + \lambda_n$.
Conversely $\lambda = \sum_{i=0}^n \bw_i\Lambda_i$ gives
$[\sum_{j=1}^{n-1}\bw_j, \sum_{j=2}^{n-1} \bw_j,\dots, \bw_{n-1},0]$.
\end{NB}%
In the same way a level $\ell$ weight of ${\mathfrak{gl}}(n)_\aff$
(resp.\ $\widehat{\mathfrak{gl}}(n)$) is an element in
$\ZZ^n\times\CC$ (resp.\ $\ZZ^n$).

Let $W_\aff$ be the Weyl group of $\algsl(n)_\aff$. It is the
semi-direct product $W\ltimes \ZZ^{n-1}$ of a finite Weyl group (= the
symmetric group of $n$ letters) and the root lattice
$\ZZ^{n-1} \cong \bigoplus_{i=1}^{n-1}\ZZ \alpha_i$.
It acts on the set of level $\ell$ weights of
$P_{\widehat{\algsl}(n)}$, identified with $\ZZ^n/\ZZ[1,\dots,1]$ by
permutation for the $W$ part, and translation by $\ell\ZZ^{n-1}$ for
the root lattice part.
\begin{NB}
  The translation part is multiplied by $\ell$ from \cite[\S6.6]{Kac}.
\end{NB}%
The fundamental alcove is $\{ [\lambda_1,\dots,\lambda_n] \mid
\lambda_1\ge\dots\ge\lambda_n\ge \lambda_1 - \ell\}$.
\begin{NB}
  By \cite[\S6.6]{Kac} the condition is
  $0\le (\lambda,\alpha_i) = \lambda_i - \lambda_{i+1}$ for $i\neq 0$,
  $\ell\ge (\lambda,\theta) = \lambda_1 - \lambda_n$.
\end{NB}%
A level $\ell$ weight $\lambda$ is dominant if and only if the
corresponding $[\lambda_1,\dots,\lambda_n]$ is contained in the
fundamental alcove.

It is convenient to extend $\lambda_i$ to all $i\in\ZZ$ so that
$\lambda_{i+n} = \lambda_i - \ell$. For example, the fundamental
alcove consists of
$\cdots\ge \lambda_0 \ge \lambda_1\ge \cdots \ge \lambda_n \ge
\lambda_{n+1} \ge \cdots$. The $j$-th permutation $s_j$
($j=0,\dots,n-1$) of $W_\aff$ acts on $(\lambda_i)_{i\in\ZZ}$ by
exchanging $\lambda_i$ and $\lambda_{i+1}$ if $i\equiv j\bmod n$.

A level $\ell$ dominant weight of $\widehat\gl(n)$ corresponds to a
\emph{generalized Young diagram with the level $\ell$ constraint} (see
\cite[App.~A]{Na-branching},\cite[\S7.6]{2016arXiv160602002N})
\begin{equation*}
   \text{\scriptsize $n$ rows }\Big\{
   \Yvcentermath1
   \cdots
   \underbrace{\overset{-3/2}{\young(\rf\rf\rf,\rf\rf\rf)}}_{\text{$\ell$ columns}}
   \,
   \overset{-1/2}{\young(\rf\rf\rf,\rf\rf\hf)}
   \,
   \overset{1/2}{\young(\rf\rf\hf,\hf\hf\hf)}
   \,
   \overset{3/2}{\young(\hf\hf\hf,\hf\hf\hf)}
   \cdots,
\end{equation*}
where a box is indexed by $(i,\xp,N)$ with $1\le i\le n$, $1\le \xp\le
\ell$, $N\in\ZZ+1/2$ and we put a gray box $\graysquare$ if
$\ell(N-1/2)+\xp\le \lambda_i$. The above figure is
$[\lambda_1,\lambda_2] = [2, -1]$ for $n=2$, $\ell = 3$.
We define the transpose of a generalized Young diagram by the
transposition of each rectangle. Then we get a sequence
$[{}^t\!\lambda_1,\dots,{}^t\!\lambda_\ell]$ which satisfies
${}^t\!\lambda_1\ge{}^t\!\lambda_2\ge\cdots\ge{}^t\!\lambda_\ell\ge
{}^t\!\lambda_1-n$, i.e., a generalized Young diagram with the level
$n$ constraint. The above example gives
$[{}^t\!\lambda_1,{}^t\!\lambda_2,{}^t\!\lambda_3] = [1,1,-1]$.

Recall that a simultaneous shift
$[\lambda_1,\dots,\lambda_n]\mapsto [\lambda_1-1,\dots,\lambda_n-1]$
does not change the weight of $P_{\widehat{\algsl}(n)}$. It
corresponds to $[{}^t\!\lambda_1,\dots,{}^t\!\lambda_\ell]\mapsto
[{}^t\!\lambda_2,\dots,{}^t\!\lambda_\ell,{}^t\!\lambda_1-n]$.
If we consider $[{}^t\!\lambda_1,\dots,{}^t\!\lambda_\ell]$ as a level
$n$ weight of $P_{\widehat{\algsl}(\ell)}$, it corresponds to a
diagram automorphism $0\to 1\to 2\to\cdots\to (\ell-1)\to 0$.
\begin{NB}
  Note
  ${}^t\lambda_\ell - ({}^t\lambda_1 - n) = n - {}^t\!\lambda_1 +
  {}^t\!\lambda_\ell$ is ${}^t\bw_0$,
  $n - {}^t\!\lambda_2 + {}^!\!\lambda_1 - n = {}^t\!\lambda_1 -
  {}^t\!\lambda_2$ is ${}^t\bw_1$.
\end{NB}%

\begin{NB}
  Note that $\sum_{i=1}^n \lambda_i$ is the number of gray boxes with
  $N > 0$ minus the number of white boxes with $N < 0$. It is clear
  that this is preserved under the transpose.
\end{NB}%

\section{Bow varieties}\label{sec:bow}

\subsection{Definition}\label{subsec:definition}

Let us recall the quiver description of bow varieties in
\cite{2016arXiv160602002N}. It is associated with a bow diagram such
as \cref{fig:bow}.
\begin{figure}[htbp]
  \centering
\begin{equation*}
  \begin{tikzpicture}
    \node (0,0) {$\vdots$};
    \draw[thick] (3,0) arc (0:15:3);
    \node at (15:3) {$\boldsymbol\times$};
    \node at (12:3.6) {$\xl_{n-2}$};
    \draw[thick] (15:3) arc (15:30:3);
    \node at (30:3) {$\boldsymbol\times$};
    \node at (27:3.6) {$\xl_{n-3}$};
    \draw[thick] (30:3) arc (30:45:3);
    \node at (45:3) {$\boldsymbol\medcirc$};
    \node at (45:3.5) {$h_4$};
    \draw[thick] (45:3) arc (45:60:3);
    \node at (60:3) {$\boldsymbol\times$};
    \node at (60:3.5) {$\xl_{n-2}$};
    \draw[thick] (60:3) arc (60:75:3);
    \node at (75:3) {$\boldsymbol\medcirc$};
    \node at (75:3.5) {$h_3$};
    \draw[thick] (75:3) arc (75:90:3);
    \node at (90:3) {$\boldsymbol\times$};
    \node at (90:3.5) {$\xl_{n-1}$};
    \draw[thick] (90:3) arc (90:108:3);
    \node at (98:2.5) {$\scriptstyle R(\zeta')$};
    \node at (108:3) {$\boldsymbol\medcirc$};
    \node at (108:3.5) {$h_2$};
    \draw[thick] (108:3) arc (108:126:3);
    \node at (118:2.5) {$\scriptstyle R(\zeta)$};
    \node at (126:3) {$\boldsymbol\medcirc$};
    \node at (126:3.5) {$h_1$};
    \draw[thick] (126:3) arc (126:144:3);
    \node at (144:3) {$\boldsymbol\times$};
    \node at (144:3.5) {$\xl_0$};
    \draw[thick] (144:3) arc (144:162:3);
    \node at (162:3) {$\boldsymbol\times$};
    \node at (162:3.5) {$\xl_1$};
    \draw[thick] (162:3) arc (162:179.5:3);
\end{tikzpicture}
\end{equation*}
  \caption{bow diagram}
  \label{fig:bow}
\end{figure}
A bow diagram consists of $\boldsymbol\times$, $\boldsymbol\medcirc$
on a circle, and nonnegative integers $R(\zeta)$ for segments $\zeta$
cut by either $\boldsymbol\times$ or $\boldsymbol\medcirc$.
We index $\boldsymbol\times$ as $\xl_0$, $\xl_1$, \dots, $\xl_{n-1}$,
$\boldsymbol\medcirc$ as $h_1$, $h_2$, \dots, $h_\ell$ in
anticlockwise and clockwise orientation respectively. The number $n$
of $\boldsymbol\times$ will be the rank of the affine Lie algebra
$\algsl(n)_\aff$. The number $\ell$ of $\boldsymbol\medcirc$ will be
the level of an integrable highest weight representation.
%
In \cref{fig:bow} only $R$ for segments $\zeta$, $\zeta'$ between
$h_1$ and $h_2$, $h_2$ and $\xl_{n-1}$ are drawn for simplicity.

We assign a vector space $V_\zeta$ for each segment $\zeta$ with
$\dim V_\zeta = R(\zeta)$. We also assign a $1$-dimensional vector
$\CC_{\xl_i}$ for each $\xl_i$.

We assume $n > 1$ throughout this paper except \cref{rem:Hilb,rem:Hilb2}.

We have a complex parameter $\nu^\CC_{h}$ and a real parameter
$\nu^\RR_{h}$ for each $h$ (i.e., $h_\xp$ for $1\le \xp\le \ell$), and
also one additional pair $\nu^\CC_*$, $\nu^\RR_*$.
We use the convention in \cite[(6.3)]{2016arXiv160602002N}, rather
than in \cite[\S2.2]{2016arXiv160602002N}.

A quiver description consists of the following:
\begin{enumerate}
      \item A linear endomorphism $B_\zeta\colon V_\zeta\to V_\zeta$.

      \item\label{item:triangle} Let $\xl$ be $\boldsymbol\times$. Let
    $\vout{\xl}$, $\vin{\xl}$ be the adjacent segments so that
\begin{tikzpicture}[baseline=0pt]
    \draw[decorate, decoration = {segment
        length=2mm, amplitude=.4mm},->] 
    (0,0) -- (2,0);
    \node at (1,0) {$\boldsymbol\times$};
    \node at (0.3,0.3) {$\vout{\xl}$};
    \node at (1.7,0.3) {$\vin{\xl}$};
    \node at (1,0.3) {$\xl$};
\end{tikzpicture}
in the anticlockwise orientation. We assign triple of linear maps
    \begin{gather*}
        A_{\xl}\colon V_{\vout{\xl}}\to V_{\vin{\xl}},
        \\
        a_{\xl}\colon \CC_{\xl}\to V_{\vin{\xl}}, \qquad
        b_{\xl} \colon V_{\vout{\xl}}\to \CC_{\xl}.
    \end{gather*}
    
      \item\label{item:twoways} Let $h$ be $\boldsymbol\medcirc$. Let
    $\vout{h}$, $\vin{h}$ be the adjacent segments so that
\begin{tikzpicture}[baseline=0pt]
    \draw[decorate, decoration = {segment
        length=2mm, amplitude=.4mm},->] 
    (0,0) -- (2,0);
    \node at (1,0) {$\boldsymbol\medcirc$};
    \node at (0.3,0.3) {$\vout{h}$};
    \node at (1.7,0.3) {$\vin{h}$};
    \node at (1,0.4) {$h$};
\end{tikzpicture}
in the anticlockwise orientation. We assign a pair of linear maps
\begin{equation*}
    C_h\colon V_{\vout{h}}\to V_{\vin{h}}, \qquad
    D_h\colon V_{\vin{h}}\to V_{\vout{h}}.
\end{equation*}
\begin{NB}
\item\label{item:triangleNB} Let $\xl$ be $\boldsymbol\times$. Let
    $\zeta^-$, $\zeta^+$ be the adjacent segments so that
\begin{tikzpicture}[baseline=0pt]
    \draw[decorate, decoration = {segment
        length=2mm, amplitude=.4mm},->] 
    (0,0) -- (2,0);
    \node at (1,0) {$\boldsymbol\times$};
    \node at (0.4,0.3) {$\zeta^-$};
    \node at (1.8,0.3) {$\zeta^+$};
    \node at (1,0.3) {$\xl$};
\end{tikzpicture}
in the anticlockwise orientation. We assign triple of linear maps
    \begin{gather*}
        A_{\xl}\colon V_{\zeta^-}\to V_{\zeta^+},
        \\
        a_{\xl}\colon \CC_{\xl}\to V_{\zeta^+}, \qquad
        b_{\xl} \colon V_{\zeta^-}\to \CC_{\xl}.
    \end{gather*}
    
      \item\label{item:twowaysNB} Let $h$ be $\boldsymbol\medcirc$. Let
    $\zeta^-$, $\zeta^+$ be the adjacent segments so that
\begin{tikzpicture}[baseline=0pt]
    \draw[decorate, decoration = {segment
        length=2mm, amplitude=.4mm},->] 
    (0,0) -- (2,0);
    \node at (1,0) {$\boldsymbol\medcirc$};
    \node at (0.4,0.3) {$\zeta^-$};
    \node at (1.7,0.3) {$\zeta^+$};
    \node at (1,0.4) {$h$};
\end{tikzpicture}
in the anticlockwise orientation. We assign a pair of linear maps
\begin{equation*}
    C_h\colon V_{\zeta^-}\to V_{\zeta^+}, \qquad
    D_h\colon V_{\zeta^+}\to V_{\zeta^-}.
\end{equation*}
\end{NB}%
\end{enumerate}
See \cite[Fig.~1]{2016arXiv160602002N}.

We denote the direct sum
$\bigoplus B_\zeta \in \End(\bigoplus V_\zeta)$ by $B$, and similarly
for $a$, $b$, $C$, $D$. However we also denote $B_\zeta$ by $B$ when
$\zeta$ is clear from the context. The same applies for $A_\xl$, etc.

We require the following conditions:
\begin{aenume}
\item Let $\xl$, $\vout{\xl}$, $\vin{\xl}$ as in (\ref{item:triangle})
  above.  As a linear map $V_{\vout{\xl}} \to V_{\vin{\xl}}$ we have
    \begin{equation*}
        (B_{\vin{\xl}} + \delta_{\xl,\xl_0}\nu^\CC_*)
        A_\xl - A_\xl B_{\vout{\xl}} + a_\xl b_\xl = 0.
    \end{equation*}
    They satisfy the two conditions (S1),(S2):
    \begin{description}
          \item[(S1)] There is no nonzero subspace $0\neq S\subset
        V_{\vout{\xl}}$ with $B_{\vout{\xl}}(S) \subset S$, $A(S) = 0 = b(S)$.

          \item[(S2)] There is no proper subspace $T\subsetneq
        V_{\vin{\xl}}$ with $B_{\vin{\xl}}(T) \subset T$, $\Ima A + \Ima a
        \subset T$.
    \end{description}

  \item Let $h$, $\vout{h}$, $\vin{h}$ as in (\ref{item:twoways}). As
    endomorphisms on $V_{\vin{h}}$ and $V_{\vout{h}}$, we have
  \begin{equation*}
      C_hD_h + B_{\vin{h}} = \nu^\CC_{h},
      \qquad
      D_hC_h + B_{\vout{h}} = \nu^\CC_h
  \end{equation*}
  respectively.
  \begin{NB}
    Consider
    \begin{tikzcd}[baseline=0pt]
      V_3 \arrow[shift left=1]{r}{C_2} & 
      V_2 \arrow[shift left=1]{r}{C_1} \arrow[shift left=1]{l}{D_2}
      \arrow[out=120,in=60,loop,looseness=3, "B_2"]
      & V_1 \arrow[shift left=1]{l}{D_1}
    \end{tikzcd}
    with $C_1 = C_{h_1}$, etc. Then $C_2 D_2 + B_2 = \nu_2^\CC$,
    $D_1 C_1 + B_2 = \nu_1^\CC$. Hence
    $C_2 D_2 - D_1 C_1 = \nu_2^\CC - \nu_1^\CC$.
  \end{NB}%

  \begin{NB}
    Consider
    \begin{tikzcd}[baseline=(V.base)]
        |[alias=V]|
      V_{n-1}^{\bw_{n-1}} \arrow{rr}{A_0} \arrow[rd,"b_{n-1}"'] 
      \arrow[out=120,in=60,loop,looseness=3, "B_{n-1}"] 
      && V_0^0
      \arrow[out=120,in=60,loop,looseness=3, "B'_0"] 
      \arrow[shift left=1]{r}{C_{1,0}}
      & V_0^1. \arrow[shift left=1]{l}{D_{1,0}}
      \\
      & \CC \arrow[ru,"a_0"'] &&
    \end{tikzcd}
    We have $D_{1,0} C_{1,0} + B'_0 = \nu_0^\CC$,
    $(B'_0 + \nu_*^\CC) A_0 - A_0 B_{n-1} + a_0 b_{n-1} = 0$. If we replace
    $B'_0 + \nu_*^\CC$ by $B'_0$, we get
    $D_{1,0} C_{1,0} + B'_0 = \nu_0^\CC + \nu_*^\CC$.
    This is the defining equation in \cite[(6.3)]{2016arXiv160602002N}.
  \end{NB}%
  Because of these defining equations, we often omit $B_\zeta$ when
  $\zeta$ has $\boldsymbol\medcirc$ on the boundary.

  \item We say $(A,B,C,D,a,b)$ is \emph{$\nu^\RR$-semistable} if the following conditions are satisfied:
\begin{description}
      \item[($\boldsymbol\nu\bf 1$)] Suppose a graded subspace
    $S = \bigoplus S_\zeta \subset \bigoplus V_\zeta$ invariant under
    $A$, $B$, $C$, $D$ with $b(S) = 0$ is given. We further assume
    that $A_x$ is an isomorphism
    $S_{\vout{\xl}}\xrightarrow[\cong]{A_\xl} S_{\vin{\xl}}$ for all
    \begin{tikzpicture}[baseline=0pt]
      \draw[decorate, decoration = {segment
          length=2mm, amplitude=.4mm},->] 
      (0,0) -- (2,0);
      \node at (1,0) {$\boldsymbol\times$};
      \node at (0.4,0.3) {$\vout{\xl}$};
      \node at (1.7,0.3) {$\vin{\xl}$};
      \node at (1,0.3) {$\xl$};
    \end{tikzpicture}.
    \begin{NB}
      \begin{tikzpicture}[baseline=0pt]
        \draw[decorate, decoration = {snake, segment length=2mm,
          amplitude=.4mm},->] (0,0) -- (1,0); \node at (0.5,0)
        {$\boldsymbol\times$}; \node at (0.2,0.3) {$\zeta^-$}; \node
        at (0.95,0.3) {$\zeta^+$}; \node at (0.5,0.3) {$\xl$};
      \end{tikzpicture}.
    \end{NB}%
    Then
    \begin{NB}
        This was the original:
    \begin{equation*}
        \sum_h \nu_{h}^\RR \dim S_{\vin{h}} \le 0.
    \end{equation*}
    \end{NB}%
    \begin{equation*}
        \nu_*^\RR \dim S_{\vin{\xl_0}} + 
        \sum_h \nu_{h}^\RR (\dim S_{\vin{h}} - \dim S_{\vout{h}})
        \le 0.
    \end{equation*}
    Here $\vin{h}$, $\vout{h}$ are determined by $h$ by the rule in
    (\ref{item:twoways}).
    \begin{NB}
      Consider
    \begin{tikzcd}[baseline=0pt]
      V_3 \arrow[shift left=1]{r}{C_2} & 
      V_2 \arrow[shift left=1]{r}{C_1} \arrow[shift left=1]{l}{D_2}
      &  V_1. \arrow[shift left=1]{l}{D_1}
    \end{tikzcd}
    The contribution of $S_2$ in the above inequality is
    \(
      (\nu^\RR_2 - \nu^\RR_1)  \dim S_2.
    \)

    Also in the above expression,
    $\nu^\RR_* - \nu^\RR_{\bw_{n-1},n-1} + \nu^\RR_{0,i}$ appears for
    the contribution from $S_{n-1}^{\bw_{n-1}}\cong S_0^0$.
    \end{NB}%

      \item[($\boldsymbol\nu\bf 2$)] Suppose a graded subspace
    $T = \bigoplus T_\zeta \subset \bigoplus V_\zeta$ invariant under
    $A$, $B$, $C$, $D$ with $\Ima a\subset T$ is given. We further
    assume that the restriction of $A$ induces an isomorphism
    $V_{\vout{\xl}}/T_{\vout{\xl}}\xrightarrow[\cong]{A_\xl}
    V_{\vin{\xl}}/T_{\vin{\xl}}$ for all
    \begin{tikzpicture}[baseline=0pt]
      \draw[decorate, decoration = {segment
          length=2mm, amplitude=.4mm},->] 
      (0,0) -- (2,0);
      \node at (1,0) {$\boldsymbol\times$};
      \node at (0.4,0.3) {$\vout{\xl}$};
      \node at (1.7,0.3) {$\vin{\xl}$};
      \node at (1,0.3) {$\xl$};
    \end{tikzpicture}.
  Then
    \begin{equation*}
        \nu^\RR_* \codim T_{\vin{\xl_0}} +
        \sum_h \nu_{h}^\RR (\codim T_{\vin{h}} -\codim T_{\vout{h}}) \ge 0.
    \end{equation*}
    
    We say $(A,B,C,D,a,b)$ is \emph{$\nu^\RR$-stable} if we have
    strict inequalities in $(\boldsymbol\nu\bf 1)$,
    $(\boldsymbol\nu\bf 2)$ unless $S_\zeta = 0$, $T_\zeta = V_\zeta$
    for all $\zeta$.
\end{description}
\end{aenume}

We have a natural group action of $\GV := \prod \GL(V_\zeta)$ by
conjugation, which preserves the above conditions.  Let
$\widetilde\cM^{\nu\text{-}\mathrm{ss}}$ (resp.\ $\widetilde\cM^{\nu\text{-}\mathrm{s}}$)
denote the set of $\nu^\RR$-semistable (resp.\ $\nu^\RR$-stable)
points satisfying other conditions (a),(b).
(We understand that the parameter $\nu$ is a pair $(\nu^\CC,\nu^\RR)$.)
We introduce the \emph{$S$-equivalence relation} $\sim$ on the
$\widetilde\cM^{\nu\text{-}\mathrm{ss}}$ by defining $m\sim m'$ if and only if
orbit closures $\overline{\GV m}$ and $\overline{\GV m'}$ intersect in
the $\widetilde\cM^{\nu\text{-}\mathrm{ss}}$. Let $\cM^\nu$ denote
$\widetilde\cM^{\nu\text{-}\mathrm{ss}}/\!\!\sim$. It is a geometric invariant
theory quotient, hence in particular has a structure of a
quasi-projective variety.
Let $\cM^{\nu\text{-}\mathrm{s}}$ denote the open subvariety
of $\cM$ given by $\widetilde\cM^{\nu\text{-}\mathrm{s}}/\GV$.

When $\nu^\CC = 0 = \nu^\RR$, we denote $\cM^\nu$,
$\cM^{\nu\text{-}\mathrm{s}}$ by $\cM$, $\cM^{\mathrm{s}}$ for brevity.
We also understand that $\cM^{\nu^\CC}$ is a bow variety with
vanishing real parameters, i.e., $\nu^\CC$ is understood as an
abbreviation of $(0,\nu^\CC)$.
We also use $\cM^{\nu^\RR}$ for a bow variety with vanishing complex
parameters, i.e., $\nu^\RR$ means $(\nu^\RR,0)$.

We have a projective morphism
\begin{equation}\label{eq:7}
    \pi\colon \cM^\nu \to \cM^{\nu^\CC},
\end{equation}
where $\nu^\CC$ is the complex part of $\nu$, i.e.,
$\nu = (\nu^\RR,\nu^\CC)$, and $\cM^{\nu^\CC}$ is a bow variety with
vanishing real parameters as above.

\begin{Remark}\label{rem:inequ}
Note that the inequality in $(\boldsymbol\nu\bf 1)$ can be rewritten as
\begin{equation*}
  \nu_*^\RR \dim S_{\vin{\xl_0}} +
  \sum_h (\nu_h^\RR - \nu_{h+1}^\RR) \dim S_{\vin{h}} \le 0.
\end{equation*}
The first term can be absorbed in the second term with $h$ such that
$\vin{h}$ is connected to $\vin{\xl_0}$ through triangle parts. For
example, we have
\(
   (\nu_*^\RR + \nu_1^\RR - \nu_2^\RR) \dim S_{\vin{h_1}}
   + \sum_{h\neq 1} (\nu_h^\RR - \nu_{h+1}^\RR) \dim S_{\vin{h}}
\)
in \cref{fig:bow}.
   
From this reformulation, it is clear that an overall shift of
$\nu^\RR_h$ is irrelevant. The same is true for $\nu^\CC_h$ as we can
simultaneously subtract a scalar from all $B_\zeta$'s. Therefore the
total number of real or complex parameters is $\ell$.
\end{Remark}

\subsection{Coulomb branch}\label{subsec:coulomb-branch}

We say that the \emph{balanced condition} is satisfied if
$R({\vout{h}}) = R({\vin{h}})$ for any
\begin{tikzpicture}[baseline=0pt]
    \draw[decorate, decoration = {segment
        length=2mm, amplitude=.4mm},->] 
    (0,0) -- (2,0);
    \node at (1,0) {$\boldsymbol\medcirc$};
    \node at (0.4,0.3) {$\vout{h}$};
    \node at (1.7,0.3) {$\vin{h}$};
    \node at (1,0.4) {$h$};
\end{tikzpicture}.
Then $R(\zeta)$ depends only the arc $\xl_{i-1}\to \xl_{i}$ which
contains $\zeta$. In particular $R$ is determined by an $n$-tuple of
integers $\bv_0$, $\bv_1$, \dots, $\bv_{n-1}$ corresponding to
$\xl_{0}\to\xl_1$, $\xl_1\to\xl_2$, \dots, $\xl_{n-1}\to\xl_{0}$.
Let $\bw_0$, $\bw_1$, \dots, $\bw_{n-1}$ be the numbers of
$\boldsymbol\medcirc$ on the corresponding arcs.
Let $\underline{\bv} = (\bv_0,\dots,\bv_{n-1})$,
$\underline{\bw} = (\bw_0,\dots,\bw_{n-1})$.

\begin{Theorem}[\protect{\cite[\S6]{2016arXiv160602002N}}]\label{thm:NT}
  Suppose the balanced condition is satisfied, and determine
  $\underline{\bv}$, $\underline{\bw}$ as above. Then the
  corresponding bow variety with parameters $\nu^\RR = 0$,
  $\nu^\CC = 0$ is isomorphic to the Coulomb branch of the framed
  affine quiver gauge theory with dimensions $\underline{\bv}$,
  $\underline{\bw}$.
\end{Theorem}

Here the Coulomb branch is one defined in
\cite{2016arXiv160103586B}. We take
\begin{equation*}
  \bN = \bigoplus_{i=0}^{n-1} \Hom(\CC^{\bv_{i-1}},\CC^{\bv_i})
  \oplus \Hom(\CC^{\bw_i},\CC^{\bv_i}),
\end{equation*}
as a representation of $\bG \defeq \prod_{i=0}^{n-1} \GL(\bv_i)$, and
consider the variety
$\cR = \{ ([g],s)\in \bG_\cK/\bG_\cO \times \bN_\cO \mid
g^{-1}s\in\bN_\cO\}$ where $\cO = \CC[[z]]\subset\cK = \CC((z))$. Then
we consider its equivariant Borel-Moore homology group
$H^{\bG_\cO}_*(\cR)$ equipped with the convolution product. The Coulomb
branch $\cM_C$ of the gauge theory associated with $\bG,\bN$ is defined
as the spectrum of the ring $H^{\bG_\cO}_*(\cR)$.

This $\bN$ is equipped with an action of a larger group
$\tilde\bG = (\bG\times T(\underline{\bw}))/\CC^\times\times
\CC^\times_{\mathrm{dil}}$
($T(\underline{\bw}) = \prod_i T^{\bw_i}$), where $\CC^\times$ is the
diagonal scalar in $\bG\times T(\underline{\bw})$ and
$\CC^\times_{\mathrm{dil}}$ acts on $\bN$ by scaling
$\Hom(\CC^{\bv_{n-1}}, \CC^{\bv_0})$.
We can consider the spectrum of $H^{\tilde\bG_\cO}_*(\cR)$, which is a
deformation of $\cM_C$ parametrized by
$\Spec H^*_{\bG_F}(\mathrm{pt}) = \operatorname{Lie}(\bG_F)$, where
$\bG_F = \tilde\bG/\bG$.
\cref{thm:NT} is generalized to this deformation (see
\cite[\S6.8.2]{2016arXiv160602002N}). Therefore it is identified with
$\cM^{\nu^\CC}$ (with vanishing real parameters) where
$\nu^\CC\in \operatorname{Lie}(\bG_F)$ is identified with a tuple of
complex numbers above under the standard coordinates. (The above
definition of bow varieties is slightly modified from one in
\cite[\S2]{2016arXiv160602002N} so that we have the identification of
parameters. See \cite[\S6.2.2]{2016arXiv160602002N}. We also change
the action of $\CC^\times_{\mathrm{dil}}$ from the scaling on
$\bigoplus_{i=0}^{n-1}\Hom(\CC^{\bv_i}, \CC^{\bv_{i+1}})$ to the
scaling only on the factor $i=n-1$.)

Similarly the variety $\cR_{\tilde\bG, \bN}$ defined for the larger
group $\tilde\bG$ (and $\bN$) defines a quasi-projective variety
equipped with a projective morphism to $\cM^{\nu^\CC}$ depending on a
$\QQ$-coweight $\varkappa$ of $\bG_F$. This is identified with $\cM^\nu$
so that this projective morphism coincides with
$\pi\colon\cM^\nu\to \cM^{\nu^\CC}$ in \eqref{eq:7}, when the real
parameters $\nu^\RR$.
(See \cite[\S4]{2018arXiv180511826B}.)
%

In what follows, we will not use the original definition of Coulomb
branches and discuss only bow varieties. Hence we do not explain the
further detail.
\begin{NB}
However dimension vectors $\underline{\bv}$,
$\underline{\bw}$ should be kept in mind.
\end{NB}%

\begin{Remark}\label{rem:reflect}
  If the balanced condition is satisfied, the ordering of parameters
  $\nu_h$ in the arc $\xl_i\to\xl_{i+1}$ is irrelevant. Consider the
  relevant part
  \begin{equation*}
    \begin{tikzcd}[column sep=small,baseline=(V.base)]
      |[alias=V]|
      V_{\vout{\xl_i}} \arrow{rr}{A_{\xl_i}} \arrow[rd,"b_{\xl_i}"'] 
      \arrow[out=120,in=60,loop,looseness=3, "B_{\vout{\xl_i}}"] 
      && 
      V_i^0 \arrow[shift left=1, r]
      \arrow[out=120,in=60,loop,looseness=3, "B_{\vin{\xl_i}}"] 
      & V_i^1 \arrow[shift left=1, l]
      \arrow[shift left=1, r]
      & \cdots \arrow[shift left=1, l]
      \arrow[shift left=1, r]
      & V_i^{\bw_i}\arrow[shift left=1, l]
            \arrow[out=120,in=60,loop,looseness=3, "B_{\vout{\xl_{i+1}}}"] 
      \arrow[rd,"b_{\xl_{i+1}}"'] 
      \arrow{rr}{A_{\xl_{i+1}}}
      && 
      V_{\vin{\xl_{i+1}}}\rlap{ .}
      \arrow[out=120,in=60,loop,looseness=3,"B_{\vin{\xl_{i+1}}}"] 
      \\
      & \CC_{\xl_i} \arrow[ru,"a_{\xl_i}"'] &&&&&
      \CC_{\xl_{i+1}} \arrow[ru,"a_{\xl_{i+1}}"']
      &
    \end{tikzcd}
  \end{equation*}
  We can apply reflection functors \cite{Na-reflect} at $V_i^1$,
  \dots, $V_i^{\bw_i-1}$ to change the ordering of parameters
  $\nu_h$. Since the balanced condition means
  $\dim V_i^0 = \cdots = \dim V_i^{\bw_i}$, dimensions of vector
  spaces are preserved under reflection functors.
\end{Remark}

\subsection{Factorization}\label{subsec:factor}

Let
$\BA^{\underline\bv} = \prod_{i=0}^{n-1} \BA^{\bv_i}/\mathfrak
S_{\bv_i} = \Spec H^*_{\bG}(\mathrm{pt})$.
In the proof of \cref{thm:NT}, the factorization morphism
$\varpi\colon \cM^\nu\to\BA^{\underline\bv}$ played an important
role. Since we will use it later, let us recall its definition and
properties. Suppose that $\cM^\nu$ is a bow variety with the balanced
condition. For each $\xl_i$ we consider $\vin{\xl_i}$ and the
associated linear map $B_{\vin{\xl_i}}$. We count its eigenvalues
with multiplicities and let it as the $i$-th component of $\varpi$.
Since $B_{\vin{\xl_i}}$ and $B_{\vout{\xl_{i+1}}}$ have the same
eigenvalues by the defining equations thanks to the following lemma, we can
also use $B_{\vout{\xl_{i+1}}}$.

\begin{NB}
    Consider
    \begin{tikzcd}[baseline=(V.base)]
        |[alias=V]|
      V_{i-1}^{\bw_{i-1}} \arrow{rr}{A_{i-1}} \arrow[rd,"b_{i-1}"'] 
      \arrow[out=120,in=60,loop,looseness=3, "B_{i-1}"] 
      && V_i^0
      \arrow[out=120,in=60,loop,looseness=3, "B'_i"] 
      \arrow[shift left=1]{r}{C_{1,i}}
      & V_i^1 \arrow[shift left=1]{l}{D_{1,i}}
      \arrow[shift left=1]{r}{C_{2,i}}
      & \cdots \arrow[shift left=1]{l}{D_{2,i}}
      \arrow[shift left=1]{r}{C_{\bw_i,i}}
      & V_i^{\bw_i} \arrow[shift left=1]{l}{D_{\bw_i,i}}
      \arrow[out=120,in=60,loop,looseness=3, "B_i"] 
      \arrow[rd,"b_{i}"'] 
      \arrow{rr}{A_{i}}
      && 
      V^0_{i+1}.
      \arrow[out=120,in=60,loop,looseness=3, "B'_{i+1}"] 
      \\
      & \CC_{\xl_i} \arrow[ru,"a_i"'] &&&&&
      \CC_{\xl_{i+1}} \arrow[ru,"a_{i+1}"']
      &
    \end{tikzcd}
    Then $\operatorname{Spec} B'_i = \nu_{1,i}^\CC - \operatorname{Spec}D_{1,i} C_{1,i} = \nu_{1,i}^\CC - \operatorname{Spec} C_{1,i} D_{1,i}
    = \nu_{2,i}^\CC - \operatorname{Spec} D_{2,i} C_{2,i} = \cdots
    = \nu_{\bw_i, i}^\CC - \operatorname{Spec} C_{\bw_i,i} D_{\bw_i,i}
    = \operatorname{Spec} B_{i}$.
\end{NB}%

\begin{Lemma}\label{lem:traceCD}
  Let $C\colon V \to V'$, $D\colon V'\to V$ be a pair of linear
  maps. We have
  \begin{equation*}
    \tr_{V'} (t + CD)^N = \tr_{V}(t + DC)^N + t^N (\dim V' - \dim V)
  \end{equation*}
  for any $N\in \ZZ_{\ge 0}$.
\end{Lemma}

\begin{NB}
  \begin{proof}
    We have
    \begin{equation*}
      (t + CD)^N = \sum_{i=0}^N \binom{N}{i} t^{N-i} (CD)^i.
    \end{equation*}
    Then $\tr_{V'} (CD)^i = \tr_{V'} (CDCD\cdots CD) = \tr_{V} (DCDC\cdots DC)
    = \tr_{V}(DC)^i$ if $i\neq 0$. The assertion follows.
  \end{proof}
\end{NB}%

Let $\underline{\bv} = \underline{\bv}' + \underline{\bv}''$ be a
decomposition of the dimension vector $\bv$. Let
$(\BA^{\underline{\bv}'}\times \BA^{\underline{\bv}''})_{\text{disj}}$
be the open subset of
$\BA^{\underline{\bv}'}\times \BA^{\underline{\bv}''}$ consisting of
$(w'_{i,1}+\dots+w'_{i,\bv'_i},
w''_{i,1}+\dots+w''_{i,\bv''_i})_{i=0}^{n-1}$ such that
$w'_{i,j} \neq w''_{i,k}$, 
$w'_{i,j} \neq w''_{i\pm 1,k}$,
$w'_{n-1,\sigma}\neq w''_{0,\tau} + \nu_*^\CC$,
$w'_{0,\sigma}\neq w''_{n-1,\tau} + \nu_*^\CC$,
where $\sigma$, $\tau$ run over the set of indices for data at $i$,
$i\pm 1$, $n-1$ or $0$ appropriately.
Note that the last two conditions correspond to the additional term
$\delta_{\xl,\xl_0}\nu^\CC_*$ in the defining equation.

\begin{Theorem}\label{thm:fact}
  \textup{(1)(\cite[Th.~6.9]{2016arXiv160602002N})} $\cM^\nu$
  satisfies the factorization property:
  \begin{equation*}
    \cM^\nu\times_{\BA^{\underline{\bv}}}
    (\BA^{\underline{\bv}'}\times \BA^{\underline{\bv}''})_{\mathrm{disj}}
    \cong
    (\cM^\nu(\underline{\bv}',\underline{\bw}) \times 
    \cM^\nu(\underline{\bv}'',\underline{\bw}))
    \times_{\BA^{\underline{\bv}'}\times \BA^{\underline{\bv}''}}
    (\BA^{\underline{\bv}'}\times \BA^{\underline{\bv}''})_{\mathrm{disj}},
  \end{equation*}
  where $\cM^\nu(\underline{\bv}',\underline{\bw})$,
  $\cM^\nu(\underline{\bv}'',\underline{\bw})$ are bow varieties
  associated with dimension vectors
  $\underline{\bv}'$,$\underline{\bw}$ and
  $\underline{\bv}''$,$\underline{\bw}$ respectively.

  \textup{(2)(\cite[Th.~6.9]{2016arXiv160602002N})} $\cM^\nu$ is
  normal, and all fibers of $\varpi$ have the same dimension.
\end{Theorem}

In fact, the balanced condition is \emph{not} essential in (1) once we
note that eigenvalues of $B_{\vin{\xl_i}}$ and $B_{\vout{\xl_{i+1}}}$
may differ, but differences are determined by $\nu^\CC_h$ and
differences of dimensions of $V_{\vin{h}}$ and $V_{\vout{h}}$ thanks
to \cref{lem:traceCD}.

When there is no fear of confusion and the open subset
$(\BA^{\underline{\bv}'}\times
\BA^{\underline{\bv}''})_{\mathrm{disj}}$ is clear from the context,
we simply write the above isomorphism after the base change as
$\cM^\nu\approx \cM^\nu(\underline{\bv}',\underline{\bw})\times
\cM^\nu(\underline{\bv}'',\underline{\bw})$ for brevity.

In the context of Coulomb branches the factorization morphism $\varpi$
corresponds to $H^*_\bG(\mathrm{pt})\to H^{\bG_\cO}(\cR)$. The
factorization property above follows from the localization theorem in
the equivariant cohomology and was essentially proved in
\cite[\S5]{2016arXiv160103586B}. The factorization of bow varieties is
an essential ingredient for the identification of bow varieties and
Coulomb branches in \cref{thm:NT}.

\subsection{Hanany-Witten transition}\label{subsec:HWtransition}

Let us recall the Hanany-Witten transition of bow varieties, which is
formulated as isomorphisms between bow varieties with adjacent
$\boldsymbol\times$ and $\boldsymbol\medcirc$ swapped
\cite[\S7]{2016arXiv160602002N}.

Consider the following part of bow data:
\begin{equation*}
  \begin{tikzcd}[column sep=small,baseline=(V.base)]
    |[alias=V]|
    V_1 \arrow[shift left=1,rr,"C"]
    \arrow[out=120,in=60,loop,looseness=3, "B_1"]
    &&
    V_2 \arrow[shift left=1,ll,"D"] \arrow[shift left=1,rr,"A"]
    \arrow[out=120,in=60,loop,looseness=3, "B_2"]
    \arrow[dr,"b"']
    &&
    V_3 
    \arrow[out=120,in=60,loop,looseness=3, "B_3"] \\
    &&&
    \CC \arrow[ur,"a"']
    &
  \end{tikzcd}
\qquad
\begin{aligned}[t]
    & CD + B_2 = \nu^\CC,
    \quad DC + B_1 = \nu^\CC,
    \\
    & B_3A - AB_2 + ab = 0.
\end{aligned}
\end{equation*}
We assume the triangle part is not $\xl_0$ for a moment.
We replace this part by
\begin{equation*}
  \begin{tikzcd}[column sep=small,baseline=(V.base)]
    |[alias=V]|
    V_1 \arrow[rr,"A^\tn"]
    \arrow[out=120,in=60,loop,looseness=3, "B_1"]
    \arrow[dr,"b^\tn"']
    &&
    V_2^\tn \arrow[shift left=1,rr,"C^\tn"] 
    \arrow[out=120,in=60,loop,looseness=3, "B_2^\tn"]
    &&
    V_3 \arrow[shift left=1,ll,"D^\tn"] 
    \arrow[out=120,in=60,loop,looseness=3, "B_3"]
    \\
    & \CC \arrow[ur,"a^\tn"']
  \end{tikzcd}
\qquad
\begin{aligned}[t]
    & C^\tn D^\tn + B_3 = \nu^\CC,
    \quad D^\tn C^\tn + B_1 = \nu^\CC,
    \\
    & B_2^\tn A^\tn - A^\tn B_1 + a^\tn b^\tn = 0.
\end{aligned}
\end{equation*}
so that we have a commutative square with the exact middle row:
\begin{equation*}
    \begin{CD}
        @. V_2 @>\alpha =\left[
          \begin{smallmatrix}
              D \\ A \\ b
          \end{smallmatrix}\right]>>
        V_1 \oplus V_3\oplus \CC @>{\beta
        = \left[
          \begin{smallmatrix}
              AC & (B_3-\nu^\CC) & a
            \end{smallmatrix}\right]}>> V_3 @.
        \\
        @. @| @| @AA{C^\tn}A 
        \\
        0 @>>> V_2 @>\alpha
        >>
        V_1 \oplus V_3\oplus \CC @>{\beta^\tn
        }>> V_2^\tn @>>> 0
        \\
           @. @A{C}AA @| @|
        \\
                @. V_1 @>\alpha^\tn =\left[
          \begin{smallmatrix}
              \nu^\CC - B_1 \\ C^\tn A^\tn \\ b^\tn
            \end{smallmatrix}\right]>>
        V_1 \oplus V_3\oplus \CC @>{\beta^\tn
        = \left[
          \begin{smallmatrix}
              A^\tn & -D^\tn & a^\tn
            \end{smallmatrix}\right]}>> V_2^\tn @..
    \end{CD}
\end{equation*}
This gives an isomorphism between bow varieties where adjacent
$\boldsymbol\medcirc$ and $\boldsymbol\times$ are swapped, and the
dimensions of vector spaces are changed by the rule
\begin{equation*}
   \dim V_2 + \dim V_2^\tn = \dim V_1 + \dim V_3 + 1.
\end{equation*}

When the triangle part is $\xl_0$, the defining equations are changed
to $(B_3+\nu^\CC_*)A - A B_2 + ab=0$,
$(B_2^\tn +\nu^\CC_*)A^\tn - A^\tn B_1 + a^\tn b^\tn =0$. Thus $B_3$
and $B_2^\tn$ must be shifted, hence other defining equations must be
changed to
\begin{equation*}
  C^\tn D^\tn + B_3 = \nu^\CC - \nu_*^\CC,
    \quad D^\tn C^\tn + B_1 = \nu^\CC-\nu_*^\CC.
\end{equation*}

We consider two $\CC^\times$-actions on the relevant part. The first
one is the action induced from the weight $1$ action on $\CC$, hence
$a\mapsto t^{-1} a$, $b\mapsto tb$, $a^\tn\mapsto t^{-1}a^\tn$,
$b^\tn \mapsto tb^\tn$, and other data are unchanged.
The second one is $A$, $b\mapsto t A$, $t b$, $A^\tn$,
$b^\tn\mapsto tA^\tn$, $t b^\tn$, and others are unchanged.
See \cite[\S6.9.2]{2016arXiv160602002N}.

The following was not stated in \cite{2016arXiv160602002N}, but clear
from the definition.

\begin{Lemma}\label{lem:hanany-witt-trans_torus}
  The Hanany-Witten transition respects the $(\CC^\times)^2$-action.
\end{Lemma}

\begin{NB}
\begin{Lemma}
  We have the followings for $N\in\ZZ_{\ge 0}$:
\begin{equation*}
  \begin{split}
    & \tr_{V_1} (t + B_1)^N
  \begin{NB2}
      {\scriptscriptstyle
    = \tr_{V_1}(t + \nu^\CC - DC)^N
    = \tr_{V_2} (t + \nu^\CC - CD)^N + (t+\nu^\CC)^N(\dim V_1 - \dim V_2)}
  \end{NB2}%
  = \tr_{V_2} (t + B_2)^N + (t+\nu^\CC)^N(\dim V_1 - \dim V_2), \\
  & \tr_{V_3} (t + B_3)^N
  \begin{NB2}
      {\scriptscriptstyle
    = \tr_{V_3}(t+\nu^\CC - C^\tn D^\tn)^N
    = \tr_{V_2^\tn}(t+\nu^\CC - D^\tn C^\tn)^N
    + (t+\nu^\CC)^N(\dim V_3 - \dim V_2^\tn)}
  \end{NB2}%
  = \tr_{V_2^\tn} (t + B_2^\tn)^N
  + (t+\nu^\CC)^N(\dim V_3 - \dim V_2^\tn).
  \end{split}
\end{equation*}
\end{Lemma}
\end{NB}%
 
Note also that the factorization morphism does not essentially change
under Hanany-Witten transition by \cref{lem:traceCD}. It is because we
can use spectra of $B_1$, $B_3$ for the definition of the
factorization morphism, which are unchanged under Hanany-Witten
transition.

\subsection{Invariants}

Let $h$ be $\boldsymbol\medcirc$ in a bow diagram. Let $N_h$ be
$R(\vin{h}) - R(\vout{h})$ where $\vin{h}$, $\vout{h}$ are as in
(\ref{item:twoways}) of \cref{subsec:definition}. Let $h_\sigma$,
$h_{\sigma+1}$ be consecutive $\boldsymbol\medcirc$. We define
\begin{equation*}
  N(h_\sigma,h_{\sigma+1}) \defeq N_{h_\sigma} - N_{h_{\sigma+1}} +
  (\text{the number of $\boldsymbol\times$ between $h_{\sigma+1}\to h_\sigma$}),
\end{equation*}
where $h_{\sigma+1}\to h_\sigma$ means that on the arc starting from $h_{\sigma+1}$
towards $h_\sigma$ in the anticlockwise direction.

Similarly we define $N_{\xl}$ and $N(\xl_i, \xl_{i+1})$ in the same
way by replacing $\boldsymbol\medcirc$ by $\boldsymbol\times$, and the
anticlockwise direction by clockwise one.

Then $N(h_\sigma,h_{\sigma+1})$, $N(\xl_i,\xl_{i+1})$ are invariant under
Hanany-Witten transition (\cite[Lem.~7.6]{2016arXiv160602002N}).

We have two other invariants
\begin{equation}\label{eq:2}
  -\sum_{\sigma=1}^\ell N_{h_\sigma}^2
  + \sum_{i=0}^{n-1} (R({\vout{\xl_i}}) + R({\vin{\xl_i}})),
  \qquad
  -\sum_{i=0}^{n-1} N_{\xl_i}^2
  + \sum_{\sigma=1}^\ell (R({\vout{h_\sigma}}) + R({\vin{h_\sigma}})),
\end{equation}
where
\begin{tikzpicture}[baseline=0pt]
    \draw[decorate, decoration = {segment
        length=2mm, amplitude=.4mm}] 
    (0,0) -- (2,0);
    \node at (1,0) {$\boldsymbol\times$};
    \node at (0.3,0.3) {$\vout{\xl_i}$};
    \node at (1.8,0.3) {$\vin{\xl_i}$};
    \node at (1,0.4) {$\xl_i$};
\end{tikzpicture}
and
\begin{tikzpicture}[baseline=0pt]
    \draw[decorate, decoration = {segment
        length=2mm, amplitude=.4mm}] 
    (0,0) -- (2,0);
    \node at (1,0) {$\boldsymbol\medcirc$};
    \node at (0.3,0.3) {$\vout{h_\sigma}$};
    \node at (1.8,0.3) {$\vin{h_\sigma}$};
    \node at (1,0.4) {$h_\sigma$};
\end{tikzpicture}
invariant under Hanany-Witten
transition (\cite[Lem.~7.6]{2016arXiv160602002N}).

The following is stated in \cite[Prop.~7.19]{2016arXiv160602002N}, but
it is based on a wrong statement.

\begin{Proposition}\label{prop:corrected}
  There is at most one bow diagram satisfying the balanced condition
  among those obtained by successive applications of Hanany-Witten
  transitions.
\end{Proposition}

\begin{proof}
  Let us suppose a bow diagram satisfying the balanced condition is
  given.
  Then $N(h_{\sigma}, h_{\sigma+1})$ is the number of
  $\boldsymbol\times$ between $h_{\sigma+1}\to h_\sigma$. Hence the
  collection $\{ N(h_\sigma,h_{\sigma+1}) \}_{\sigma=1}^\ell$
  determines the distribution of $\boldsymbol\medcirc$ and
  $\boldsymbol\times$. On the other hand the vector $\underline{\bw}$
  in \cref{subsec:coulomb-branch} is given by the number of
  $\boldsymbol\medcirc$ on the arc $\xl_{i}\to\xl_{i+1}$. Therefore
  $\underline{\bw}$ is determined up to a cyclic permutation. This is
  because the numbering of $\boldsymbol\times$ by $\xl_i$ is not fixed
  by $N(h_{\sigma},h_{\sigma+1})$, but the only ambiguity is given by
  a shift $\xl_i\mapsto \xl_{i+i_0}$ (modulo $n$) for some
  $i_0$. (This ambiguity was overlooked in
  \cite[Prop.~7.19]{2016arXiv160602002N}.)

  But this shift cannot be achieved by Hanany-Witten transitions. By
  Hanany-Witten transitions, $N_{h_\sigma}$ is changed by the number
  of $\boldsymbol\times$ crossing
  $\overset{h_\sigma}{\boldsymbol\medcirc}$ in the anticlockwise
  direction minus the number of $\boldsymbol\times$ crossing in the
  clockwise direction. Therefore in order to keep $N_{h_\sigma}$
  vanishing, those two numbers must be equal. Therefore the numbering
  for the first $\boldsymbol\times$ after
  $\overset{h_\sigma}{\boldsymbol\medcirc}$ (in either direction)
  remains the same. Thus the shift is not possible.
  Thus the numbering of $\boldsymbol\times$ by $\xl_i$ is unique,
  hence $\underline{\bw}$ is determined.

  Next note that $N(\xl_i,\xl_{i+1})$ is the $i$-th entry of
  $\underline{\mathbf u} = \underline{\bw} - C\underline{\bv} = (\bw_i
  + \bv_{i-1} + \bv_{i+1} - 2\bv_i)_{i=0}^{n-1}$
  \cite[Lem.~7.18]{2016arXiv160602002N}. Therefore the collection
  $\{ N(\xl_i,\xl_{i+1}) \}$ and $\underline{\bw}$ determine
  $\underline{\bv}$ up to an addition of a multiple of
  ${}^t(1,1,\dots,1)$. But an addition of ${}^t(1,1,\dots,1)$
  increases two invariants in \eqref{eq:2} by $2n$ and $2\ell$
  respectively. Hence $\bv$, i.e., numbers $R(\zeta)$ on segments are
  determined.
\end{proof}

\begin{NB}
  Let us note that
  \begin{equation*}
    \bv_0 + \bv_1 + \cdots + \bv_{n-1} = 
    - \frac12 \sum_{\sigma=1}^\ell N_{h_\sigma}^2
    + \frac12 \sum_{i=0}^{n-1} (\bv_{\vout{\xl_i}} + \bv_{\vin{\xl_i}})
  \end{equation*}
  for a balanced case.
\end{NB}%

\subsection{Another form}\label{subsec:another-form}

Let us take a bow diagram satisfying the balanced condition, and
define dimension vectors $\underline\bv = (\bv_0,\dots,\bv_{n-1})$,
$\underline\bw = (\bw_0,\dots,\bw_{n-1})$ as in
\cref{subsec:coulomb-branch}. We apply Hanany-Witten transitions
successively so that we separate $\boldsymbol\times$ and
$\boldsymbol\medcirc$ as follows.
\begin{equation}\label{eq:16}
    \begin{tikzpicture}[baseline=(current  bounding  box.center),
        label distance=1pt]
        \draw[rounded corners=15pt] (0,0) rectangle ++(8,2);
        \node[label=below:$\xl_1$,label=above:$\mu_1$] (x1) at (1,0) 
        {$\boldsymbol\times$};
        \node[label=below:$\xl_2$,label=above:$\mu_2$] at (2,0)
        {$\boldsymbol\times$};
        \node[label=below:$h_1$,label=above:${}^t\!\lambda_1$] at (2,2)
        {$\boldsymbol\medcirc$};
        \node[label=below:$\xl_3$,label=above:$\mu_3$] at (3,0)
        {$\boldsymbol\times$};
        \node[label=below:$h_2$,label=above:${}^t\!\lambda_2$] at (3,2)
        {$\boldsymbol\medcirc$};
        \node at (4,-.5) {$\cdots$};
        \node at (4,1.4) {$\cdots$};
        \node[label=below:$\xl_{n-2}$,label=above:$\mu_{n-2}$] at (5,0)
        {$\boldsymbol\times$};
        \node[label=below:$h_{\ell-1}$,label=above:${}^t\!\lambda_{\ell-1}$]
        at (5,2) {$\boldsymbol\medcirc$};
        \node[label=below:$\xl_{n-1}$,label=above:$\mu_{n-1}$] at
        (6,0) {$\boldsymbol\times$};
        \node[label=below:$h_{\ell}$,label=above:${}^t\!\lambda_{\ell}$] at
        (6,2) {$\boldsymbol\medcirc$};
        \node[label=below:$\xl_{0}$,label=above:$\mu_{n}$] at (7,0)
        {$\boldsymbol\times$};
        \node at (8.5,1) {$\bv_{0}$};
        \node at (-1.2,1) {$\bv_{0}+\sum_i i\bw_i$};
    \end{tikzpicture}
\end{equation}
See the proof of \cite[Cor.~7.21]{2016arXiv160602002N}. We do not move
$\boldsymbol\medcirc$ across $\xl_0$, hence the dimension $\bv_0$ next
to $\xl_0$ is unchanged.
Numbers ${}^t\!\lambda_\xp$, $\mu_i$ above $\boldsymbol\medcirc$,
$\boldsymbol\times$ indicate the values of $N_{h_\xp}$, $N_{\xl_i}$
respectively. Two numbers $\bv_0$ and $\bv_0+ \sum_i i\bw_i$ are
dimensions of vector spaces on two segments, between $\xl_0$ and
$h_\ell$, $h_1$ and $\xl_1$ respectively.

The numbers ${}^t\!\lambda_\xp$, $\mu_i$ are $N_{h_\xp}$ and $N_{\xl_i}$
respectively. In order to explain how ${}^t\!\lambda_\xp$, $\mu_i$ are
given in terms of $\underline{\bv}$, $\underline{\bw}$, we introduce
weights of $P_{\algsl(n)_\aff}$, $P_{\widehat{\gl}(n)}$.
We first define two weights $\lambda$, $\mu$ of $P_{\algsl(n)_\aff}$ by
\begin{equation}\label{eq:6}
   \lambda = \sum_{i=0}^{n-1} \bw_i \Lambda_i, \qquad
   \mu = \sum_{i=0}^{n-1} (\bw_i \Lambda_i - \bv_i \alpha_i).
\end{equation}
We have $\langle d, \lambda\rangle = 0$,
$\langle d,\mu\rangle = -\bv_0$. (Note that this is different from the
convention in \cite[\S7.6]{2016arXiv160602002N} by
$-\bv_0\delta$. Since we change $\bv_0$, the current convention is
more natural.)

Let $\ell$ be the level of $\lambda$, which is equal also to the level
of $\mu$. It is $\langle c,\lambda\rangle = \langle c,\mu\rangle$,
where $c$ is the central element in $\widehat{\algsl}(n)$. Concretely
it is equal to $\sum_{i=0}^{n-1} \bw_i$, hence the number of
$\boldsymbol\medcirc$. Therefore we can number $\boldsymbol\medcirc$ as
$h_1$, \dots, $h_\ell$ as in \eqref{eq:16}.

We define two integer vectors $[\lambda_1,\dots,\lambda_n]$,
$[\mu_1,\dots,\mu_n]$ by
\begin{equation*}
  \lambda_i = \sum_{j=i}^{n-1} \bw_j, \quad
  \mu_i = \bv_{n-1} - \bv_0 + \sum_{j=i}^{n-1} \mathbf u_j,
  \begin{NB}
    \lambda_n = 0, \quad \mu_n = \bv_{n-1} - \bv_0
  \end{NB}%
\end{equation*}
where $\mathbf u_i$ is the $i$-th entry of
$\underline{\mathbf u} = \underline{\bw} - C\underline{\mathbf v}$ as
in the proof of \cref{prop:corrected}.
\begin{NB}
  Note that we have $\lambda_n = 0$, which is equivalent to
  ${}^t\!\lambda_1 < n$.
\end{NB}%
We consider them as level $\ell$ weights of
$\widehat{\mathfrak{gl}}(n)$. Note that
$\sum_{i=1}^n \lambda_i = \sum_{i=1}^n \mu_i$
\begin{NB}
  We have $\sum \lambda_i = \sum_{j=1}^{n-1} j\bw_j$. On the other
  hand,
  \begin{equation*}
    \begin{split}
      & \sum \mu_i = n(\bv_{n-1} - \bv_0) + \sum_{j=1}^{n-1} j \mathbf
      u_j
      = n(\bv_{n-1} - \bv_0) + \sum_{j=1}^{n-1}
      j(\bw_j + \bv_{j-1} + \bv_{j+1} - 2\bv_j) \\
      =\; & n(\bv_{n-1} - \bv_0) + \sum_{j=1}^{n-1} j\bw_j +
      \sum_{j=0}^{n-2} (j+1)\bv_j + \sum_{j=2}^n (j-1)\bv_j -
      2\sum_{j=1}^{n-1} j\bv_j \\
      =\; & -n \bv_0 + \sum_{j=1}^{n-1} j\bw_j
      + \sum_{j=0}^{n-1} (j+1)\bv_j + \sum_{j=1}^n (j-1)\bv_j -
      2\sum_{j=1}^{n-1} j\bv_j
      = \sum_{j=1}^{n-1} j\bw_j.
    \end{split}
  \end{equation*}
\end{NB}%
It means that the pairings with the central element
$\operatorname{diag}(1,\dots,1)$ in $\mathfrak{gl}(n)$ (charges) are
the same for $\lambda$ and $\mu$.
\begin{NB}
We have the following:
\begin{equation*}
  \begin{split}
    & \lambda = \sum_{i=1}^{n-1} (\lambda_i - \lambda_{i+1})\Lambda_i
    + (\ell + \lambda_n - \lambda_1)\Lambda_0,
    \begin{NB2}
      = \ell \Lambda_0 +\sum_{i=1}^{n} \lambda_i (\Lambda_i -
      \Lambda_{i-1})
    \end{NB2}%
    \\
    & \mu = \sum_{i=1}^{n-1} (\mu_i - \mu_{i+1})\Lambda_i + (\ell +
    \mu_n - \mu_1) \Lambda_0 - \bv_0 \delta.
  \end{split}
\end{equation*}
\end{NB}%

Note that $\lambda$ is dominant by its definition. Hence it is
contained in the fundamental alcove, i.e.,
\begin{equation*}
    \lambda_1 \ge \lambda_2 \ge \cdots \ge \lambda_n \ge \lambda_1 - \ell.
\end{equation*}
Let us consider the corresponding generalized Young diagram
$[\lambda_1,\dots,\lambda_n]$ and its transpose
$[{}^t\!\lambda_{1},\dots,{}^t\!\lambda_{\ell}]$. The latter is a
generalized Young diagram with the level $n$ constraint, i.e.,
\begin{equation*}
  {}^t\!\lambda_1\ge {}^t\!\lambda_2\ge \cdots \ge {}^t\lambda_\ell
  \ge {}^t\lambda_1 - n.
\end{equation*}
See~\cref{sec:weights-affine-lie}.
Now numbers in \eqref{eq:16} are given by these rules.

Recall that we did not move $\boldsymbol\medcirc$ over $\xl_0$ in this
procedure. Let us move $\overset{h_1}{\boldsymbol\medcirc}$
anticlockwise overall $\xl_i$ including $\xl_0$ to return back to the
same picture as \eqref{eq:16}. But this process changes numbers on
$\boldsymbol\medcirc$, $\boldsymbol\times$, and also dimensions of
vector spaces on two segments. The result is as follows.
\begin{equation}\label{eq:3}
    \begin{tikzpicture}[baseline=(current  bounding  box.center),
        label distance=1pt]
        \draw[rounded corners=15pt] (0,0) rectangle ++(8,2);
        \node[label=below:$\xl_1$,label=above:$\mu_1\!-\!1$] (x1) at (1,0) 
        {$\boldsymbol\times$};
        \node[label=below:$\xl_2$,label=above:$\mu_2\!-\!1$] at (2.2,0)
        {$\boldsymbol\times$};
        \node[label=below:$h_2$,label=above:${}^t\!\lambda_2$] at (2,2)
        {$\boldsymbol\medcirc$};
        \node[label=below:$h_3$,label=above:${}^t\!\lambda_3$] at (3,2)
        {$\boldsymbol\medcirc$};
        \node at (4,-.5) {$\cdots$};
        \node at (4,1.4) {$\cdots$};
        \node[label=below:$h_{\ell}$,label=above:${}^t\!\lambda_{\ell}$]
        at (5,2) {$\boldsymbol\medcirc$};
        \node[label=below:$\xl_{n-1}$,label=above:$\mu_{n-1}\!-\!1$] at
        (5.6,0) {$\boldsymbol\times$};
        \node[label=below:$h_{1}$,label=above:${}^t\!\lambda_{1}\!-\!n$] at
        (6,2) {$\boldsymbol\medcirc$};
        \node[label=below:$\xl_{0}$,label=above:$\mu_{n}\!-\!1$] at (7,0)
        {$\boldsymbol\times$};
        \node at (9.2,1) {$\bv_{0}\!-\!{}^t\!\lambda_1\!+\!n$};
        \node at (-1.5,1) {$\bv_{0}\!+\!\sum_i i\bw_i\!-\!{}^t\!\lambda_1$};
    \end{tikzpicture}
\end{equation}
\begin{NB}
  We have $\sum_i i\bw_i - n = \sum_i (\mu_i - 1) = {}^t\!\lambda_2 +\dots + {}^t\!\lambda_\ell + {}^t\!\lambda_1 - n$.
\end{NB}%
Note that
$[{}^t\!\lambda_1,\dots,{}^t\!\lambda_\ell]\mapsto
[{}^t\!\lambda_2,\dots,{}^t\!\lambda_\ell,{}^t\!\lambda_1-n]$
corresponds to a simultaneous shift
$[\lambda_1,\dots,\lambda_n]\mapsto [\lambda_1-1,\dots,\lambda_n-1]$
(see \cref{sec:weights-affine-lie}). Therefore this process shifts both
$\lambda$ and $\mu$ simultaneously.

\begin{NB}
  Let us check that the invariants in \eqref{eq:2} are unchanged. Let
  us consider the first invariant. Since $N_{h_\sigma}$ is unchanged
  for $\sigma\ge 2$ and $N_{h_1}$ is changed from ${}^t\!\lambda_1$ to
  ${}^t\!\lambda_1 - n$, $\sum_{\sigma} N_{h_\sigma}^2$ is increased by
  $n^2 - 2n{}^t\!\lambda_1$. On the other hand,
  $\bv_{\vout{\xl_1}}+\bv_{\vin{\xl_1}}$ is decreased by
  $2{}^t\!\lambda_1 - 1$, $\bv_{\vout{\xl_2}}+\bv_{\vin{\xl_2}}$ is
  decreased by $2{}^t\!\lambda_1 - 3$, and so on, and finally
  $\bv_{\vout{\xl_0}}+\bv_{\vin{\xl_0}}$ is decreased by
  $2{}^t\!\lambda_1 - 2n + 1$. In total, the decrease in the second sum
  is $2n{}^t\!\lambda_1 - n^2$. So the sum is unchanged.

  Next consider the second invariant. Each $N_{\xl_i}$ is changed from
  $\mu_i$ to $\mu_i - 1$. Therefore $\sum_i N_{\xl_i}^2$ is increased
  by $\sum_i (-2\mu_i + 1) = n - 2\sum_i \mu_i$. On the other hand,
  the sum of dimensions of vector spaces next to the first
  $\boldsymbol\medcirc$ decreases by
  ${}^t\!\lambda_1 + {}^t\!\lambda_2$, the next sum decreases by
  ${}^t\!\lambda_2 + {}^t\!\lambda_3$, and the sum one before the
  last decreases by ${}^t\!\lambda_{\ell-1} + {}^t\!\lambda_\ell$. And
  the last one decreases by
  ${}^t\!\lambda_\ell + {}^t\!\lambda_1 - n$. The total is
  $2\sum{}^t\!\lambda_\sigma - n = 2\sum_i \mu_i - n$.
\end{NB}%

By \cite[Prop.~7.20]{2016arXiv160602002N} (more precisely its dual
version), bow diagrams \eqref{eq:16} can be transformed to a balanced one by
successive applications of Hanany-Witten transition, as
$[{}^t\!\lambda_1,\dots,{}^t\!\lambda_\ell]$ is in the fundamental
alcove. And it is unique by \cref{prop:corrected}.
Dimension vectors are read off from numbers in \eqref{eq:16} as
\begin{equation*}
  \begin{split}
    & \bw_i = \lambda_i - \lambda_{i+1} \quad (1\le i\le n-1), \qquad
    \bw_0 = \lambda_n - (\lambda_1 - \ell), \\
    & \mathbf u_i = \mu_i - \mu_{i+1} \quad (1\le i\le n-1), \qquad
    \mathbf u_0 = \mu_n - (\mu_1 - \ell),
  \end{split}
\end{equation*}
and ambiguity $(\bv_0,\dots,\bv_{n-1}) + \ZZ (1,\dots,1)$ is fixed by
one of invariants in \eqref{eq:2} by
\begin{equation*}
  \bv_0 + \dots + \bv_{n-1} = 
  - \frac12 \sum_{\sigma=1}^\ell {}^t\!\lambda_\sigma^2
        + \frac12 \sum_{i=0}^{n-1} (R({\vout{\xl_i}}) + R({\vin{\xl_i}})).
\end{equation*}

Let us denote the bow variety with the balanced condition with
dimension vectors $\underline\bv$, $\underline\bw$ by
$\cM(\lambda,\mu)$ hereafter, where $\lambda$, $\mu$ are given by
\eqref{eq:6}.
(Note that it was denoted by $\cM(\mu,\lambda)$ in
\cite{2016arXiv160602002N}.)

\subsection{Stratification}

By \cite[Prop.~4.5, Th.~7.26]{2016arXiv160602002N} 
\begin{Theorem}\label{thm:stratification}
    \textup{(1)} Suppose $\ell\neq 1$. We have a stratification
\begin{equation*}
    \cM(\lambda,\mu) =
    \bigsqcup_{\kappa,\underline{k}} \cM^{\mathrm{s}}(\kappa,\mu)
    \times S^{\underline{k}}(\CC^2\setminus\{0\}/(\ZZ/\ell \ZZ)),
\end{equation*}
where $\underline{k} = [k_1,k_2,\dots]$ is a partition and $\kappa$ is
a dominant weight with $\mu\le\kappa\le \lambda - |\underline{k}|\delta$.
The same is true if we replace $\CC^2\setminus\{0\}/(\ZZ/\ell \ZZ)$ by
$\CC^2$ and we allow only $\kappa = \lambda - |\underline{k}|\delta$
when $\ell = 1$.

\textup{(2)} Take a generic real parameter $\nu^\RR$. Then
$\pi\colon \cM^{\nu^\RR}(\lambda,\mu) \to \cM(\lambda,\mu)$ in \eqref{eq:7}
is a semismall resolution with respect to the above stratification,
and all strata are relevant.
\end{Theorem}

Here we understand $\cM^{\nu^\RR}(\lambda,\mu)$ has vanishing complex
parameters.

\begin{NB}
    Is it true that the projection is surjective ? Yes it is by a
    property of Coulomb branches.
\end{NB}%

\begin{NB}

This subsection is incomplete.

\subsection{Multiplication}

We introduce an analog of multiplication of slices in
\cite[2(vi)]{2016arXiv160403625B}.

Suppose that $\underline{\bv}^1$, $\underline{\bw}^1$,
$\underline{\bv}^2$, $\underline{\bw}^2\in\ZZ_{\ge 0}^n$ are given. We
introduce a \emph{multiplication} morphism
\begin{equation*}
    \cM(\underline{\bv}^1,\underline{\bw}^1)\times
    \cM(\underline{\bv}^2,\underline{\bw}^2)
    \to
    \cM(\underline{\bv}^1+\underline{\bv}^2,
    \underline{\bw}^1+\underline{\bw}^2).
\end{equation*}

Let us first apply Hanany-Witten transitions successively to make both
$\cM(\underline{\bv}^1,\underline{\bw}^1)$,
$\cM(\underline{\bv}^2,\underline{\bw}^2)$ into the form
\eqref{eq:16}. 
Let $[\lambda_1^a,\dots,\lambda^a_n]$, $[\mu_1^a,\dots,\mu_n^a]$ be
vectors associated to $\underline{\bv}^a$, $\underline{\bw}^a$ as
above ($a=1,2$). Let $\ell_a = \sum \bw_i^a$.
We have $\ell_a\ge\lambda^a_1\ge\dots\ge\lambda^a_n\ge 0$. We consider
the transpose $[{}^t\!\lambda^a_1,\dots, {}^t\!\lambda^a_{\ell_a}]$,
which satisfies
$n\ge{}^t\!\lambda^a_1\ge\dots\ge{}^t\!\lambda^a_{\ell_a}\ge 0$.
We arrange ${}^t\!\lambda^a_\sigma$ in descending order as
${}^t\!\lambda_1\ge \cdots \ge {}^t\!\lambda_\ell$
($\ell = \ell_1 + \ell_2$). They sit between $0$ and $n$, hence
$[{}^t\!\lambda_1,\dots,{}^t\lambda_\ell]$ is a generalized Young
diagram with the level $n$ constraint. Let
$[\lambda_1,\dots,\lambda_n]$ be its transpose. Note that we have
$\lambda_i = \lambda_i^1+\lambda_i^2$, hence it is the vector
associated with $\underline{\bw}^1+\underline{\bw}^2$.

Consider diagrams of the form \eqref{eq:16} for $\underline{\bv}^a$,
$\underline{\bw}^a$. The number of $\boldsymbol\medcirc$ is
$\ell_a$. We add $(\ell - \ell_a)$
\begin{NB2}
  $= \ell_{3-a}$
\end{NB2}%
$\boldsymbol\medcirc$'s according to the way ${}^t\!\lambda_\sigma$ is
constructed from ${}^t\!\lambda^a_\sigma$. Namely when we number
$\boldsymbol\medcirc$'s from $h_1$ to $h_\ell$, new
$\boldsymbol\medcirc$'s are $h_\sigma$ such that
${}^t\!\lambda_{\sigma}$ is not an entry of
$[{}^t\!\lambda^a_1,\dots,{}^t\!\lambda^a_{\ell_a}]$.
We assign $0$ to the new $\boldsymbol\medcirc$'s. We define a morphism
from the bow variety for the old diagram to the new one as
follows. Suppose that a new $\boldsymbol\medcirc$ is added between $h$
and $h'$:
\begin{equation*}
\begin{tikzpicture}[baseline=-2pt]
    \draw[decorate, decoration = {segment
        length=2mm, amplitude=.4mm}] 
    (0,0) -- (3,0);
    \node at (1,0) {$\boldsymbol\medcirc$};
    \node at (1,0.4) {$h$};
    \node at (2,0) {$\boldsymbol\medcirc$};
    \node at (2,0.4) {$h'$};
\end{tikzpicture}
\leadsto    
\begin{tikzpicture}[baseline=-2pt]
    \draw[decorate, decoration = {segment
        length=2mm, amplitude=.4mm}] 
    (0,0) -- (4,0);
    \node at (1,0) {$\boldsymbol\medcirc$};
    \node at (1,0.4) {$h$};
    \node at (2,0) {$\boldsymbol\medcirc$};
    \node at (3,0) {$\boldsymbol\medcirc$};
    \node at (3,0.4) {$h'$};
\end{tikzpicture}
\end{equation*}
Let $V_1$, $V_2$, $V_3$ be vector spaces on segments in the old
diagram from left to right. We take a copy of $V_2$ and define the
morphism by assigning linear maps as
\begin{equation*}
      \begin{tikzcd}
      V_1 \arrow[shift left=1, r, "C_h"] &
      V_{2} \arrow[shift left=1, r, "C_{h'}"]
      \arrow[shift left=1, l, "D_h"] &
      V_{3} \arrow[shift left=1, l, "D_{h'}"]
      \arrow[r, rightsquigarrow] &
      V_1 \arrow[shift left=1, r, "C_h"] &
      V_{2} \arrow[shift left=1, r, "\id"]
      \arrow[shift left=1, l, "D_h"] &
      V_{2} \arrow[shift left=1, r, "C_{h'}"]
      \arrow[shift left=1, l, "C_h D_h"] &
      V_{3} \rlap{ .} \arrow[shift left=1, l, "D_{h'}"]
    \end{tikzcd}
\end{equation*}
We define the morphism when we add $\boldsymbol\medcirc$ between
$\boldsymbol\medcirc$ and $\boldsymbol\times$ in the same way.
\begin{NB2}
\begin{equation*}
  \xymatrix@C=1.2em{ V_1 \ar@(ur,ul)_{B_1} \ar[rr]^{A} \ar[dr]_{b} &&
  V_2
  \ar@(ur,ul)_{B_2} \ar@<-.5ex>[rr]_{C} && V_3 \ar@(ur,ul)_{B_3}
  \ar@<-.5ex>[ll]_{D}
  \\
  & \CC \ar[ur]_{a} &&&}
  \leadsto
  \xymatrix@C=1.2em{ V_1 \ar@(ur,ul)_{B_1} \ar[rr]^{A} \ar[dr]_{b} &&
  V_2
  \ar@(ur,ul)_{B_2} \ar@<-.5ex>[rr]_{\id} && V_2 \ar@<-.5ex>[rr]_{C}
  \ar@<-.5ex>[ll]_{DC} && V_3 \ar@(ur,ul)_{B_3} \ar@<-.5ex>[ll]_{D}
  \\
  & \CC \ar[ur]_{a} &&&}.
\end{equation*}
\end{NB2}%

We apply this procedure to bow diagrams for $a=1,2$. Note that both
two diagrams have $\ell$ $\boldsymbol\medcirc$'s (and $n$
$\boldsymbol\times$'s). We define a new bow diagram with $\ell$
$\boldsymbol\medcirc$'s and $n$ $\boldsymbol\times$ as well, and
assign numbers ${}^t\!\lambda_1$, \dots, ${}^t\!\lambda_\ell$ to
$\boldsymbol\medcirc$'s, and $\mu^1_i + \mu^2_i$ to $i$-th
$\boldsymbol\times$. We also assign $\bv_0^1 + \bv_0^2$ on the segment
between $\xl_0$ and $h_\ell$.
It is the form \eqref{eq:16} for
$\cM(\underline{\bv}^1+\underline{\bv}^2,
\underline{\bw}^1+\underline{\bw}^2)$.
Now we take direct sum of vector spaces for $a=1$, $2$ on each segment,
and define linear maps as block diagonal matrices.

Let us check that conditions (S1,2) are satisfied.
This operation defines multiplication morphism.
\end{NB}%

\subsection{Birational coordinate system}\label{subsec:coord}

We consider a bow variety with the balanced condition. Suppose that we
have vector spaces and linear maps between $\vin{\xl_i}$ and
$\vout{\xl_{i+1}}$ as in \cref{rem:reflect}.
Recall that the factorization morphism $\varpi$ is given by
eigenvalues with multiplicities of $B_{\vin{\xl_i}}$ or
$B_{\vout{\xl_{i+1}}}$ by \cref{subsec:factor}. Let $w_{i,r}$
($r=1,\dots,\bv_i$) be eigenvalues. Let
$C_{\bw_i\dots 1,i}\colon V^0_i\to V^{\bw_i}_i$ be the composite of
linear maps $V^0_i\to V^1_i \to \dots \to V^{\bw_i}_i$. We define
\begin{equation*}
  \mathsf y_{i,r} \defeq b_{\xl_{i+1}} \prod_{\substack{1\le s\le \bv_i\\ s\neq r}}
  (B_{\vout{\xl_{i+1}}}-w_{i,s}\id)C_{\bw_i\dots 1,i} a_{\xl_i}.
\end{equation*}
Then $w_{i,r}$, $\mathsf y_{i,r}$ are regular functions on
$\cM(\lambda,\mu)\times_{\BA^{\underline{\bv}}}\BA^{|\underline{\bv}|}$.
The isomorphism between the Coulomb branch and the bow variety is
constructed so that $w_{i,r}$, $\mathsf y_{i,r}$ coincide with one
defined in \cite[\S3(iii)]{2016arXiv160403625B} ($\mathsf y_{i,r}$ was
denoted by $\overline{\mathsf y}_{i,r}$ there) and
\cite[\S6.8.1]{2016arXiv160602002N}, as certain homology classes.

Under Hanany-Witten transitions, eigenvalues of $B_{\vout{\xl_i}}$
remain unchanged except $0$ by \cref{lem:traceCD}. It is also obvious
that $\mathsf y_{i,r}$, as a map from $\CC$ at $\xl_{i-1}$ to $\CC$ at
$\xl_i$, is unchanged. Therefore $w_{i,r}$, $\mathbf y_{i,r}$ are
given by the same formula for bow varieties, not necessarily with
balanced condition.

\section{Torus action}\label{sec:torus}

We consider the $T = (\CC^\times)^{Q_0}$-action given by $\pi_1(\bG)\cong \ZZ^n$
\cite[\S3(v)]{2016arXiv160103586B}.
(We formally add the factor $\CC^\times$ even when $\bv_i = 0$ so that
$T$ depends only on $Q$.)
Let $(s_0,\dots,s_{n-1})$ denote
the standard coordinates of $T$, where $s_i$ corresponds to
$\pi_1(\GL(\bv_i))$ at the vertex $i$.
By \cite[\S6.9.2]{2016arXiv160602002N}, the action is given by one
induced by $s_0\cdots s_{i-1}$ on $\CC_{\xl_i}$ ($1\le i\le n-1$) and
$A$, $b$ at the vertex $\xl_0$ are multiplied by $s_0\cdots s_{n-1}$.
\begin{NB}
    We take $t_0 = 1$. Then $t_{i+1} = t_i s_i = s_0\cdots s_i$ by
    induction. We also have
    $t_\delta = t_{n-1} s_{n-1} = s_0\cdots s_{n-1}$.
\end{NB}%
By \cref{lem:hanany-witt-trans_torus} Hanany-Witten transitions are
equivariant under the torus action.

\subsection{Torus fixed points}\label{subsec:torus-fixed-points}

Recall that the $T$-fixed point in $\cM(\lambda,\mu)$ is either a
single point or empty \cite[Prop.~7.30]{2016arXiv160602002N}. 
(Recall also $\cM(\lambda,\mu)$ is $\cM^{\nu}(\lambda,\mu)$ for
$\nu^\CC = \nu^\RR = 0$.)
Let us review the proof as we will study fixed points with respect to
smaller tori by using the same argument. Let us give a slight
improvement simultaneously.

We have a stratification
$\cM(\lambda,\mu) = \bigsqcup \cM^{\mathrm{s}}(\kappa,\mu)\times
S^{\underline{k}}(\CC^2\setminus\{0\}/(\ZZ/\ell\ZZ))$ ($\ell > 1$),
$\bigsqcup \cM^{\mathrm{s}}(\kappa,\mu)\times
S^{\underline{k}}(\CC^2)$ ($\ell = 1$), where $\underline{k}$ is a
partition, and $\kappa$ is a dominant weight between $\mu$ and
$\lambda - |\underline{k}|\delta$.  (See \cref{thm:stratification}.)
This stratification is compatible with the $T$-action.
On the factor $S^{\underline{k}}(\CC^2\setminus\{0\}/(\ZZ/\ell\ZZ))$
or $S^{\underline{k}}(\CC^2)$, the action is induced from the
$\CC^\times$-action on $\CC^2$ given by $t\cdot (x,y) = (tx, t^{-1}y)$
where $t=s_0\cdots s_{n-1}$.
\begin{NB}
    $x$ corresponds to the composite of linear maps going the
    anticlockwise direction. $y$ corresponds to opposite maps.
\end{NB}%
In particular, the $T$-fixed point set is empty for
$S^{\underline{k}}(\CC^2\setminus\{0\}/(\ZZ/\ell\ZZ))$ unless
$\underline{k}=\emptyset$ when $\ell > 1$. When $\ell = 1$, it is the
single point $k[0]$ of $S^{(k)}(\CC^2)$ when $\underline{k}$ only has
a single entry $k$ and is empty otherwise. Thus it is enough to
determine $\cM^{\mathrm{s}}(\kappa,\mu)^T$.

\begin{Proposition}\label{prop:torus-fixed-points}
  The $T$-fixed point set $\cM^{\mathrm{s}}(\kappa,\mu)^T$ is a single
  point if $\kappa = \mu^+$, the dominant weight in the Weyl group
  orbit of $\mu$. It is empty otherwise.
\end{Proposition}

\begin{proof}
\begin{NB}
We have a stratification
$\cM(\lambda,\mu) = \bigsqcup \cM^{\mathrm{s}}(\kappa,\mu)\times
S^{\underline{k}}(\CC^2\setminus\{0\}/(\ZZ/\ell\ZZ))$
($\ell > 1$),
$\bigsqcup \cM^{\mathrm{s}}(\kappa,\mu)\times
S^{\underline{k}}(\CC^2)$
($\ell = 1$), where $\underline{k}$ is a partition, and $\kappa$ is a
dominant weight between $\mu$ and $\lambda - |\underline{k}|\delta$.
(See \cite[Th.~7.26]{2016arXiv160602002N}.)
This stratification is compatible with the $T$-action. On the factor
$S^{\underline{k}}(\CC^2\setminus\{0\}/(\ZZ/\ell\ZZ))$ or
$S^{\underline{k}}(\CC^2)$, the action is induced from the
$\CC^\times$-action on $\CC^2$ given by $t\cdot (x,y) = (tx, t^{-1}y)$
where $t=s_0\cdots s_{n-1}$.
\begin{NB2}
    $x$ corresponds to the composite of linear maps going the
    anticlockwise direction. $y$ corresponds to opposite maps.
\end{NB2}%
In particular, the $T$-fixed point set is empty for
$S^{\underline{k}}(\CC^2\setminus\{0\}/(\ZZ/\ell\ZZ))$, and the single
point $|\underline{k}|0$ for $S^{\underline{k}}(\CC^2)$. Thus it is
enough to determine $\cM^{\mathrm{s}}(\kappa,\mu)^T$.
Let us first ignore the requirement $\kappa\le\lambda$.
\end{NB}%
Let us apply Hanany-Witten transitions successively to separate
$\boldsymbol\medcirc$ and $\boldsymbol\times$. Moreover we move
$\boldsymbol\medcirc$ all over $\boldsymbol\times$ many times so that
$\mu_i\defeq N_{\xl_i}\ge 0$ for all $i$. See
\cref{subsec:another-form}.

Let us take a representative $(A,B,C,D,a,b)$ of a point in
$\cM^{\mathrm{s}}(\kappa,\mu)^T$. We have a homomorphism
$\rho = (\rho_\zeta) \colon T\to \GV = \prod_\zeta \GL(V_\zeta)$ such that
\begin{gather*}
  A_{\xl_i} = \rho_{\vin{\xl_i}}(s)^{-1} A_{\xl_i} \rho_{\vout{\xl_i}}(s)\quad (i\neq 0),
  \qquad
  s_0\cdots s_{n-1} A_{\xl_0} = \rho_{\vin{\xl_0}}(s)^{-1} A_{\xl_0} \rho_{\vout{\xl_0}}(s),\\
  B_{\zeta} = \rho_\zeta(s)^{-1} B_\zeta \rho_\zeta(s),\qquad
  (s_0\cdots s_{i-1})^{-1} a_{\xl_i} = \rho_{\vin{\xl_i}}(s)^{-1} a_{\xl_i}, \\
  s_0\cdots s_{i-1} b_{\xl_i} = b_{\xl_i} \rho_{\vout{\xl_i}}(s) \quad (i\neq 0),\qquad
  s_0\cdots s_{n-1} b_{\xl_0} = b_{\xl_0} \rho_{\vout{\xl_0}}(s),\\
  C_{h_\sigma} = \rho_{\vin{h_\sigma}}(s)^{-1} C_{h_\sigma} \rho_{\vout{h_\sigma}}(s),\qquad
  D_{h_\sigma} = \rho_{\vout{h_\sigma}}(s)^{-1} D_{h_\sigma} \rho_{\vin{h_\sigma}}(s),
\end{gather*}
where
\begin{tikzpicture}[baseline=0pt]
      \draw[decorate, decoration = {segment
          length=2mm, amplitude=.4mm},->] 
      (0,0) -- (2,0);
      \node at (1,0) {$\boldsymbol\times$};
      \node at (0.4,0.3) {$\scriptstyle\vout{\xl_i}$};
      \node at (1.8,0.3) {$\scriptstyle\vin{\xl_i}$};
      \node at (1,0.3) {$\scriptstyle\xl_i$};
\end{tikzpicture},
\begin{tikzpicture}[baseline=0pt]
    \draw[decorate, decoration = {segment
        length=2mm, amplitude=.4mm},->] 
    (0,0) -- (2,0);
    \node at (1,0) {$\boldsymbol\medcirc$};
    \node at (0.4,0.3) {$\scriptstyle\vout{h_\sigma}$};
    \node at (1.7,0.3) {$\scriptstyle\vin{h_\sigma}$};
    \node at (1,0.4) {$\scriptstyle h_\sigma$};
\end{tikzpicture}.
We consider the weight space decomposition of $V$ with respect to
$\rho$. Then $A_{\xl_i}$ ($i\neq 0$), $B_\zeta$, $C_{h_\sigma}$ and
$D_{h_\sigma}$ preserve weight spaces, while $A_{\xl_0}$ shifts weights
by $(s_0\cdots s_{n-1})^{-1}$. And $a_{\xl_i}$ sends $\CC_{\xl_i}$ to the
$s_0\cdots s_{i-1}$ weight space, $b_{\xl_i}$ is $0$ on weight spaces
other than $s_0\cdots s_{i-1}$ ($i\neq 0$), $s_0\cdots s_{n-1}$ ($i=0$).
In particular, we classify weights to $n$ classes
$(s_0\cdots s_{i-1})(s_0\cdots s_{n-1})^{\ZZ}$ ($i=0,\dots,n-1$) so
that $\CC_{\xl_i}$ can be `communicated' with only weight spaces in the
$i$-th class. Thus the data is a direct sum of $n$ pieces.

Let us consider the direct summand for $\CC_{\xl_i}$ and 
the corresponding bow diagram. Since $a_{\xl_j}$, $b_{\xl_j}$ vanish for
$j\neq i$, $A_{\xl_j}$ is an isomorphism thanks to the condition
(S1,2). Then we can identify $V_{\vout{\xl_j}}$ with $V_{\vin{\xl_j}}$ so
that we may assume that the bow diagram has only one
$\boldsymbol\times$. Moreover we can unwind the circle to a line as
$A_{\xl_0}$ shifts weight by $(s_0\cdots s_{n-1})^{-1}$. Thus the bow
diagram is
\begin{equation}\label{eq:4}
\begin{tikzpicture}[baseline=(current  bounding  box.center)]
  \node[label=below:$\xl_{i}$,label=above:$\mu_i$] at (0,0)
  {$\vphantom{j^X}\boldsymbol\times$};
  \node[label=below:$h_0$] at (1,0)
        {$\vphantom{j^X}\boldsymbol\medcirc$};
  \node[label=below:$h_1$] at (-1,0)
        {$\vphantom{j^X}\boldsymbol\medcirc$};
  \node[label=below:$h_{-1}$] at (2,0)
        {$\vphantom{j^X}\boldsymbol\medcirc$};
  \node[label=below:$h_2$] at (-2,0)
        {$\vphantom{j^X}\boldsymbol\medcirc$};
  \node[label=below:$\cdots$] at (3,0)
        {$\vphantom{j^X}$};
  \node[label=below:$\cdots$] at (-3,0)
        {$\vphantom{j^X}$};
  \node[label=below:$h_{-n}$] at (4,0)
        {$\vphantom{j^X}\boldsymbol\medcirc$};
  \node[label=below:$h_m$] at (-4,0)
        {$\vphantom{j^X}\boldsymbol\medcirc$};
  \node[label=below:$\cdots$] at (5,0)
        {$\vphantom{j^X}\boldsymbol\medcirc$};
  \node[label=below:$\cdots$] at (-5,0)
        {$\vphantom{j^X}\boldsymbol\medcirc$};
  \draw[-] (-6,0) -- (6,0);
\end{tikzpicture}.
\end{equation}
Note that $\mu_i$ remains the same as one for the original bow
diagram, as $A_{\xl_i}$ is an isomorphism in other summands.
In particular, the above $\mu_i$ is $\ge 0$ as we have assumed so in the original bow diagram.

By the necessary condition for $\fM_0^{\mathrm{reg}}\neq\emptyset$ in
\cite[Lem.~8.1]{Na-quiver}, \cite[Lem.~4.7]{Na-alg} we have
$N(h_\sigma,h_{\sigma+1}) \ge 0$ for any $\sigma$. (To show
$N(h_0, h_1) \ge 0$, we use the Hanany-Witten transition. See the
proof of \cite[Th.~7.26]{2016arXiv160602002N} for detail.) On the
other hand, $\sum N(h_\sigma,h_{\sigma+1}) = 1$ by
definition. Therefore $N(h_\sigma,h_{\sigma+1})\neq 0$ at most one
$\sigma$. Since $N_{h_\sigma} = 0$ if $|\sigma|$ is sufficiently
large, we have $N_{h_1} = \cdots = N_{h_{\mu_i}} = 1$, other
$N_{h_\sigma} = 0$.
\begin{NB}
if
$\mu_i \ge 0$ or $N_{h_0} = \cdots = N_{h_{-\mu_i}} = 1$, other
$N_{h_\sigma} = 0$ if $\mu_i \le 0$.  
\end{NB}%
Thus the data looks like
\begin{equation}\label{eq:8}
    \begin{tikzcd}[column sep=small]
        \CC \arrow[shift left=1]{r} & 
        \CC^2 \arrow[shift left=1]{l} \arrow[shift left=1]{r} &
        \cdots \arrow[shift left=1]{l} \arrow[shift left=1]{r} &
        \CC^{\mu_i}\arrow[shift left=1]{l}
        \arrow[out=120,in=60,loop,looseness=3, "B_{\vout{\xl_i}}"]
        \arrow[dr, "b_{\xl_i}"'] & \\
        &&&& \CC_{\xl_i}
    \end{tikzcd}
\end{equation}
\begin{NB}
    \begin{equation*}
  \xymatrix@C=1.2em{
        \CC \ar@<.5ex>[r] & \CC^2 \ar@<.5ex>[l] \ar@<.5ex>[r]
        & \cdots \ar@<.5ex>[l] \ar@<.5ex>[r]
        & \CC^{\mu_i} \ar@<.5ex>[l] \ar[dr]_{b_{\xl_i}} \ar@(ur,ul)_{B_{\xl_i}}
        &
        \\
        &&&& \CC}
\text{ or }
  \xymatrix@C=1.2em{
    & \CC^{-\mu_i}  \ar@(ur,ul)_{B_{\xl_i}} \ar@<.5ex>[r]
    & \cdots \ar@<.5ex>[l] \ar@<.5ex>[r]
    & \CC^2 \ar@<.5ex>[l] \ar@<.5ex>[r]
        & \CC \ar@<.5ex>[l]
        \\
        \CC \ar[ur]_{a_{\xl_i}} &&&&}
    \end{equation*}
\end{NB}%
\begin{NB}
  according to the sign of $\mu_i$. In either case, $B_{x_i}$ is
  nilpotent by the defining equation. The condition (S1,2) says
  $a_{x_i}$ (resp.\ $b_{x_i}$) is cyclic (resp.\ cocyclic) vector for
  $B_{x_i}$. Hence $a_{x_i}, B_{x_i}$ (resp.\ $b_{x_i}, B_{x_i}$) can
  be set to $e_{-\mu_i}$, $J_{-\mu_i}$ (resp.\ ${}^t e_{\mu_i}$,
  ${}^t J_{\mu_i}$) by conjugation. Other data $C$, $D$ can be also
  fixed by the remaining group action. Therefore the fixed point set
  is a single point.
\end{NB}%
Note that $B_{\vout{\xl_i}}$ is nilpotent by the defining equation. The
condition (S1,2) says $b_{\xl_i}$ is cocyclic vector for
$B_{\vout{\xl_i}}$. Hence $b_{\xl_i}$, $B_{\vout{\xl_i}}$ can be moved to
${}^t e_{\mu_i}$, ${}^t\! J_{\mu_i}$ by conjugation.  Once the action of
$\GL(\mu_i)$ is killed, the remaining data $C_i$, $D_i$ are regarded
as a point of a quiver variety of type $A_{\mu_i-1}$, which is
the nilpotent cone of $\algsl(\mu_i-1)$. See
\cite[\S7]{Na-quiver}. Therefore the fixed point set is a single
point.
Alternatively, we apply Hanany-Witten transitions $\mu_i$ times to move
$\xl_i$ to the left. Then we arrive at the bow variety with all vector
spaces $V_\zeta$ vanish. It is a single point.

Since dimensions of vector spaces are determined by $\mu$, the weight
$\kappa$ with $\cM^{\mathrm{s}}(\kappa,\mu)^T\neq\emptyset$ is
determined uniquely by $\mu$. Let us show that $\kappa = \mu^+$.
Recall that the projection of $\kappa$ to $P_{\widehat{\algsl}(n)}$
can be read off from $N_{h_\sigma}$ ($\sigma=1,\dots,\ell$) as in
\eqref{eq:16}. Namely $\kappa$ is the transpose of the generalized
Young diagram associated with
$N_{h_1}\ge N_{h_2}\ge \cdots \ge N_{h_\ell}\ge N_{h_1} - n$. The
coefficient of $\delta$ is fixed by either of two invariants
\eqref{eq:2}.

First, note that $\kappa$ is unchanged under permutations of
$\mu_i$. It is because vector spaces are direct sums of vector
spaces for $\CC_{\xl_i}$, hence the ordering is not relevant. Thus we
may assume $\mu_1\ge\mu_2\ge\dots\ge\mu_n$.
\begin{NB}
  This process does not change the dimension $\bv_0$ between $\xl_0$
  and $h_\ell$. From this observation, it is clear that this process
  changes $\mu$ by $w \mu$ for $w$ in the finite Weyl group.
\end{NB}%
Next we change $\mu$ to
$(\mu_n+\ell,\mu_2,\dots,\mu_{n-1},\mu_1 - \ell)$. Then all vector
spaces in the upper semicircle of \eqref{eq:16} between $\xl_1$ and
$\xl_0$ changes their dimensions by $\mu_n+\ell - \mu_1$, which is
$\mathbf u_0 = \bw_0 + \bv_1+\bv_{n-1}-2\bv_0$.
Thus this process does not change $\kappa$, and replace $\mu$ by
$s_0\mu$. Here $s_0$ is the simple reflection for the $0$-th simple
root. If we extend $\mu_i$ to $i\in\ZZ$ by $\mu_{i+n} = \mu_i - \ell$
as in \cref{sec:weights-affine-lie}, $s_0$ exchanges $\mu_j$ and
$\mu_{j+1}$ for $j\equiv 0\bmod n$. It is also clear that $s_0$ does
not change $\kappa$ in this description.
\begin{NB}
  Consider the left invariant of \eqref{eq:2} divided by $2$. It is
  \begin{equation*}
    -\frac12 \sum_{\sigma=1}^\ell N_{h_\sigma}^2 +
    n\bv_0 + \frac12 \left(
      (2n-1)\mu_{n} + (2n-3)\mu_{n-1} + \dots + \mu_1
      \right),
  \end{equation*}
  for the diagram in \eqref{eq:16}. By the above process, this is
  added by
  $n\mathbf u_0 + \frac12 (2n-2) (\mu_1-\ell-\mu_n) = \mathbf
  u_0$. Recall this is equal to $\bv_0 + \bv_1 + \dots + \bv_{n-1}$.
  On the other hand, by $s_0$, $\bv_0$ is changed by
  $\bw_0 + \bv_1 + \bv_{n-1} - \bv_0 = \bv_0 + \mathbf u_0$. Other
  $\bv_i$ are unchanged. Therefore this sends $\mu$ to $s_0\mu$.
\end{NB}%
These two operation generate the Weyl group of $\algsl(n)_\aff$. Hence
we may assume that $\mu$ is in the fundamental alcove, i.e.,
$\mu_1\ge\dots\ge\mu_n\ge\mu_1-\ell$. Moreover we can shift $\mu_i$
simultaneously by the process explained in
\cref{subsec:another-form}. So we make $\mu_n = 0$. Then $\mu$
determines a Young diagram with at most $(n-1)$ rows and $\ell$
columns. Then $N_{h_\sigma}$ is the number of rows which have length
more than $\sigma$. Namely $[N_{h_1},\dots,N_{h_\ell}]$ is the
transpose of the Young diagram. Moreover the vector space between
$\xl_0$ and $h_\ell$ is $0$ as $\ell\ge \mu_i$ for any $i$.
\begin{NB}
    Recall that the data is the direct sum of \eqref{eq:8} over
    $i$. If $\mu_i\le\ell$, the vector space between $\xl_0$ and
    $h_\ell$ in \eqref{eq:16} is $0$ for the summand $i$.
\end{NB}%
It means that $\kappa = \mu$.
\end{proof}

\begin{Corollary}\label{cor:torus-fixed-points}
    The followings are equivalent:

\textup{(1)} $\cM(\lambda,\mu)$ has a $T$-fixed point.

\textup{(2)} $\lambda\ge\mu^+$ in the dominance order.

\textup{(3)} $\mu$ is a weight of the integrable highest weight module
with the highest weight $\lambda$.
\end{Corollary}

The equivalence between (2) and (3) is a consequence of
\cite[Prop.~12.5]{Kac}. The equivalence between (1) and (2) follows
from an observation that $\cM(\lambda,\mu)$ contains
$\cM^{\mathrm{s}}(\mu^+,\mu)$ as an stratum if and only if
$\lambda\ge\mu^+$ in the dominance order.
This confirms \cref{conj:old}(1) and a part of (3), that is
$V_\mu(\lambda) = 0 \Leftrightarrow \cM(\lambda,\mu)^T = \emptyset$,
for affine type $A$.

Note also that the above will follow without the combinatorial
argument in \cref{prop:torus-fixed-points}, once we will endow
$\mathcal V(\lambda)$ with a $\mathfrak g_{\mathrm{KM}}$-module
structure and identify it with the integrable highest weight module.

\begin{NB}
\begin{Claim}
  The total dimension of vector spaces between $\xl_0\to h_1$ (in the
  anticlockwise direction) is $\sum_i \mu_i(\mu_i - 1)/2$.
\end{Claim}

\begin{proof}
  Consider the contribution from the direct summand for
  $\CC_{\xl_i}$. If $\mu_i > 0$, the vector space on the right of
  $h_1$ is of dimension $\mu_i$. Hence the left of $h_1$ has dimension
  $\mu_i - 1$, and the next is $\mu_i - 2$, and so on. Hence the total
  is $1 + 2 + \dots + (\mu_i - 1) = \mu_i(\mu_i - 1)/2$. If
  $\mu_i < 0$, the left of $h_0$ has dimension $-\mu_i$, and the next
  is $-\mu_i-1$, and so on. Hence the total sum is
  $1 + 2 + \dots + (-\mu_i) = (1 - \mu_i)(-\mu_i)/2 = \mu_i(\mu_i -
  1)/2$.
\end{proof}

Similarly the total dimension of vector spaces between
$\xl_1\to h_\ell$ (in the clockwise direction) is
$\sum_i \mu_i(\mu_i + 1)/2$. Therefore the difference of dimensions of
vector spaces between $x_1\to h_1$, $x_0\to h_\ell$ is
$\sum_i \mu_i(\mu_i+1)/2 - \sum_i \mu_i(\mu_i-1)/2 = \sum_i \mu_i$
($=\sum_i \lambda_i$).
\end{NB}%

Let us take a generic 1-parameter family $\chi\colon\CC^\times\to T$ and consider a diagram
\begin{equation*}
  \mathrm{pt} = \cM(\lambda,\mu)^T \xleftarrow{p}\fA_\chi(\lambda,\mu)
  \xrightarrow{j} \cM(\lambda,\mu),
\end{equation*}
where $\fA_\chi(\lambda,\mu)$ is the attracting set with respect to
$\rho$. When there is no fear of confusion, we denote it simply by $\fA$.
Here $j$ is the inclusion, and $p$ is the map given by taking the
limit $\rho(t)$ for $t\to 0$.
Then the following confirms
\cite[Conjecture~3.25(2)]{2016arXiv160403625B} for affine type $A$.
\begin{Theorem}[\protect{\cite[Prop.~7.33]{2016arXiv160602002N}}]
  The intersection of $\fA$ with strata in \cref{thm:stratification}
  are lagrangian. In particular, the hyperbolic restriction functor
  $\Phi = p_* j^!$ is hyperbolic semismall.
\end{Theorem}

Let
\begin{equation*}
  \mathcal V_\mu(\lambda) \defeq \Phi(\mathrm{IC}(\cM(\lambda,\mu))).
\end{equation*}
Thanks to the above theorem, this is a vector space. We also set
$\mathcal V(\lambda) = \bigoplus \mathcal V_\mu(\lambda)$.

\begin{Remark}\label{rem:MV}
  Since $\Phi$ is hyperbolic semismall, we have
  \[
    \mathcal V_\mu(\lambda) = \Phi(\mathrm{IC}(\cM(\lambda,\mu)))
    \cong H_{2\dim}(\fA_\chi(\lambda,\mu)),
  \]
  and hence $\mathcal V_\mu(\lambda)$ possesses a basis parametrized
  by irreducible components of of $\fA_\chi(\lambda,\mu)$ of
  $\dim = \dim \cM(\lambda,\mu)/2$ \cite[Prop.~3.10]{MV2}. After
  identifying $\mathcal V_\mu(\lambda)$ with a weight space of an
  integrable highest weight representation of an affine Lie algebra,
  irreducible components are regarded as Mirkovi\'c-Vilonen cycles for
  affine Lie algebras, hence for double affine Grassmannian. This
  generalizes the construction in \cite[\S6]{Na-branching} for
  dominant $\mu$.
\end{Remark}

\subsection{Deformed case}\label{subsec:deformed-case}

Let us next choose a parameter $\nu^{\bullet}$ such that
$\nu_*^{\bullet,\CC} = 0 = \nu_*^{\bullet,\RR}$ and
$\nu^{\bullet,\CC}_h = \nu^{\bullet,\RR}_h$ is either $\dot\nu$ or
$0$, where $\dot\nu$ is a nonzero real number. As before we denote by
$\nu^{\bullet,\CC}$, the complex part of $\nu^\bullet$ and understand
it as one with vanishing real parameters.
This gives us a decomposition $\lambda = \lambda^1+\lambda^2$: recall
$\lambda = \sum \bw_i\Lambda_i$, and $\bw_i$ is the number of
$\boldsymbol\medcirc$'s between $\xl_i$ and $\xl_{i+1}$. Let $\bw^1_i$,
$\bw^2_i$ be numbers of $\boldsymbol\medcirc$'s with $\nu^\CC_h = 0$
and $=\dot\nu$ respectively. Then we have $\bw_i = \bw^1_i + \bw^2_i$
and define $\lambda^1 = \sum \bw^1_i \Lambda_i$,
$\lambda^2 = \sum \bw^2_i\Lambda_i$.
Thanks to \cref{rem:reflect}, the results in this subsection does not
depend on how $\dot\nu$ is distributed to
$\nu^\bullet_h$. They depend only on $\lambda^1$, $\lambda^2$.

The following confirms \cite[Conjecture~3.27(1) and the first half of
(2)]{2016arXiv160403625B} for affine type $A$.

\begin{Proposition}\label{prop:deformed}
  The $T$-fixed points $\cM^{\nu^{\bullet,\CC}}(\lambda,\mu)^T$ are
  finite, and correspond to decomposition $\mu = \mu^1 + \mu^2$ such
  that $\cM(\lambda^1, {\mu^1})^T$, $\cM(\lambda^2,{\mu^2})^T$ are
  nonempty. Moreover the image of
  $\cM^{\nu^{\bullet,\CC}}(\lambda,\mu)^T$ under the factorization
  morphism $\varpi$ is supported at $0$ and $\dot\nu$, and the
  multiplicities are determined by the decomposition
  $\mu = \mu^1+\mu^2$. Therefore around the point corresponding to
  $\mu = \mu^1+\mu^2$, the bow variety
  $\cM^{\nu^{\bullet,\CC}}(\lambda,\mu)$ is isomorphic to a
  neighborhood of the unique $T$-fixed point in
  $\cM(\lambda^1,\mu^1)\times \cM(\lambda^2,\mu^2)$ by the
  factorization.
\end{Proposition}

More explicitly, if $p\in\cM^{\nu^{\bullet,\CC}}(\lambda,\mu)^T$ corresponds
to a decomposition $\mu = \mu^1+\mu^2$, we have
$\varpi(p) = (\bv^1_i [0] + \bv^2_i [\dot\nu])_{i\in Q_0}$
with $\lambda^1 - \mu^1 = \sum \bv^1_i \alpha_i$,
$\lambda^2 - \mu^2 = \sum \bv^2_i\alpha_i$.

\begin{proof}
We argue as in the proof of \cref{prop:torus-fixed-points} to
decompose a fixed point as sum of the data associated with
\eqref{eq:4} over $i$. Each summand looks like
\begin{equation}\label{eq:19}
        \begin{tikzcd}[column sep=small]
        V_m \arrow[shift left=1,r,"C_{m-1}"] & 
        V_{m-1} \arrow[shift left=1,l,"D_{m-1}"] \arrow[shift left=1]{r} &
        \cdots \arrow[shift left=1]{l} \arrow[shift left=1,r,"C_1"] &
        V_1 \arrow[shift left=1,l,"D_1"]
        \arrow[out=120,in=60,loop,looseness=3, "B_1"]
        \arrow[dr, "b"'] \arrow[rr,"A"]
        && V_0 \arrow[out=120,in=60,loop,looseness=3, "B_0"]
        \arrow[shift left=1,r,"C_{-1}"]
        & \cdots \arrow[shift left=1,l,"D_{-1}"] \arrow[shift left=1]{r}
        & V_{1-n}  \arrow[shift left=1,r,"C_{-n}"]
        \arrow[shift left=1,l]
        & V_{-n} \arrow[shift left=1,l,"D_{-n}"]
        \\
        &&&& \CC \arrow[ur, "a"'] &&&&
    \end{tikzcd}
\end{equation}
instead of \eqref{eq:8}. By the defining equation we see that
eigenvalues of $B_0$, $B_1$ are either $0$ or $\dot\nu$.
\begin{NB}
  Consider 
  \(
  \begin{tikzcd}[baseline=0pt]
      \cdots \arrow[shift left=1,r]
      & V_3 \arrow[shift left=1]{r}{C_2} 
      \arrow[shift left=1,l] & 
      V_2 \arrow[shift left=1]{r}{C_1} \arrow[shift left=1]{l}{D_2}
      &  V_1. \arrow[shift left=1]{l}{D_1}
      \arrow[out=120,in=60,loop,looseness=3, "B_1"]
  \end{tikzcd}
  \)
  We have $D_1 C_1 - \nu_1^\CC = C_2 D_2 - \nu_2^\CC$. Hence
  $\tr_{V_1} (t-B_1)^N
  = \tr_{V_1} (t + C_1 D_1 - \nu_1^\CC)^N
  = \tr_{V_2} (t + D_1 C_1 - \nu_1^\CC)^N
  + (t - \nu_1^\CC)^N (\dim V_1 - \dim V_2)
  = \tr_{V_2}(t + C_2 D_2 - \nu_2^\CC)^N
  + (t - \nu_1^\CC)^N (\dim V_1 - \dim V_2)
  = \tr_{V_3}(t + D_2 C_2 - \nu_2^\CC)^N
  + (t - \nu_2^\CC)^N (\dim V_2 - \dim V_3)
  + (t - \nu_1^\CC)^N (\dim V_1 - \dim V_2)$ and so on.
  If $\nu_i^\CC$ is either $0$ or $\dot\nu$, the eigenvalues of $B_1$
  are either $0$ or $\dot\nu$.
\end{NB}%
Let us decompose $V_0 = V_0'\oplus V_0''$, $V_1 = V_1'\oplus V_1''$ by
eigenvalues, the prime for $0$ and the double prime for $\dot\nu$. We
have inherited decomposition $V_h = V_h' \oplus V_h''$ for other
$h$'s. Then we have factorization
$\cM^{\nu^{\bullet,\CC}}\approx {}'\cM^{\nu^{\bullet,\CC}}\times
{}''\cM^{\nu^{\bullet,\CC}}$ around the fixed point. (See the paragraph
after \cref{thm:fact}.)

Take a two way part $h$ and consider the corresponding $C_h$,
$D_h$. If $\nu^\CC_h = 0$ (resp.\ $\dot\nu$), then $C_h$, $D_h$ are
isomorphism on the factor $V''$ (resp.\ $V'$).
\begin{NB}
  $C_h D_h + B_{\vin{h}} = 0$, $D_h C_h + B_{\vout{h}} = 0$. Since $B$
  is $\dot\nu$ plus nilpotent on $V''$, $C_h D_h$ and $D_h C_h$ are
  invertible.
\end{NB}%
Therefore we can absorb the action of $\GL(V_{\vin{h}}'')$ to that of
$\GL(V_{\vout{h}}'')$ by normalizing $C_h|_{V_{\vout{h}}''}$ to the
identity homomorphism. The remaining $D_h|_{V_{\vin{h}}''}$ is fixed
by the defining equation. Thus we can eliminate $V_{\vin{h}}''$.
The same applies for $V'$. After this normalization each factor gives
a fixed point in a bow variety with parameter $\nu = 0$, one
classified in \cref{prop:torus-fixed-points}. Therefore it is a form
in \eqref{eq:8}.
\begin{NB}
  \begin{equation*}
    \begin{tikzcd}
      V_h \arrow[shift left=1, r, "C_h"] &
      V_{h+1} \arrow[shift left=1, r, "C_{h+1}"]
      \arrow[shift left=1, l, "D_h"] &
      V_{h+2} \arrow[shift left=1, l, "D_{h+1}"]
      \arrow[r, leftrightsquigarrow] &
      V_h \arrow[shift left=1, r, "C_h"] &
      V_{h+1} \arrow[shift left=1, r, "\id"]
      \arrow[shift left=1, l, "D_h"] &
      V_{h+1} \arrow[shift left=1, r, "C_{h+1}"]
      \arrow[shift left=1, l, "C_h D_h"] &
      V_{h+2} \rlap{ .} \arrow[shift left=1, l, "D_{h+1}"]
    \end{tikzcd}
  \end{equation*}
\end{NB}%

We return back to the balanced bow variety
$\cM^{\nu^{\bullet,\CC}}(\lambda,\mu)$ by successive applications of
Hanany-Witten transitions. Eigenvalues of $B$'s are preserved
(\cref{lem:traceCD}), hence we have the factorization
$\cM^{\nu^{\bullet,\CC}}(\lambda,\mu) \approx \cM(\lambda^1,\mu^1)\times
\cM(\lambda^2,\mu^2)$
corresponding to the above factorization. Here we eliminate several
summands of $V'$, $V''$ as above.
Since the fixed point corresponds to a fixed point in
$\cM(\lambda^1,\mu^1)\times \cM(\lambda^2,\mu^2)$, it is the one
described in \cref{prop:torus-fixed-points}. In particular we must have
$\cM(\lambda^1,\mu^1)^T$, $\cM(\lambda^2,{\mu^2})^T\neq \emptyset$.
Conversely if $\cM(\lambda^1,\mu^1)^T$,
$\cM(\lambda^2,{\mu^2})^T\neq \emptyset$, we get a fixed point in
$\cM(\lambda,\mu)$ after adding removed summands of $V'$, $V''$.
\end{proof}

We allow scaling of $\nu^{\bullet,\CC}$ in the defining equation to
get families
\begin{equation*}
  \underline{\cM}(\lambda,\mu) = \bigsqcup_{\nu'\in\CC\nu^{\bullet,\CC}}
  \cM^{\nu'}(\lambda,\mu), \qquad
  \underline{\widetilde\cM}(\lambda,\mu) = \bigsqcup_{\nu'\in\CC\nu^{\bullet,\CC}}
  \cM^{\nu^{\bullet,\RR},\nu'}(\lambda,\mu)
\end{equation*}
parametrized by $\CC$, where the fiber at $0$ (resp.\ $1$) is
$\cM(\lambda,\mu)$ (resp.\ $\cM^{\nu^{\bullet,\CC}}(\lambda,\mu)$) for
the first family. The fiber at $0$ of the second family is
$\cM^{\nu^{\bullet,\RR}}(\lambda,\mu)$. We have
$\underline{\pi}\colon\underline{\widetilde\cM}(\lambda,\mu)\to
\underline{\cM}(\lambda,\mu)$ as in \eqref{eq:7}.
Let us denote its fiber over $0$ as
$\pi\colon \cM^{\nu^{\bullet,\RR}}(\lambda,\mu)\to\cM(\lambda,\mu)$.  The latter
is a stratified semismall birational morphism
(\cite[Prop.~4.5]{2016arXiv160602002N}), hence the former is a
stratified small birational morphism.
Moreover $\underline{\widetilde\cM}(\lambda,\mu)$ is a topologically
trivial family. In fact, by hyperK\"ahler rotation, we can consider it
as a family of bow varieties with the same defining equation (the
parameter is $\nu^{\bullet,\RR}$) with varying stability conditions
with parameters in $\RR\operatorname{Re}\nu^{\bullet,\CC}$. But we
only consider submodules whose dimension vectors are perpendicular to
$\nu^{\bullet,\RR}$, hence slopes appearing in inequalities in
($\boldsymbol\nu\bf 1, 2$) in \cref{subsec:definition} are
automatically vanish. Therefore
$\operatorname{Re}\nu^{\bullet,\CC}$-(semi)stability and
$0$-(semi)stability are equivalent.
Therefore the nearby cycle functor $\psi$ for the family
$\underline{\cM}(\lambda,\mu)\to\CC$ sends
$\mathrm{IC}(\underline{\cM}(\lambda,\mu))$ to
$\pi_*(\mathrm{IC}(\cM^{\nu^{\bullet,\RR}}(\lambda,\mu)))$.
Now the remaining half of
\cite[Conjecture~3.27(2)]{2016arXiv160403625B} follows as
\begin{Corollary}\label{cor:tensor}
We have
\begin{equation*}
  \begin{split}
    & \Phi \circ \pi_*(\mathrm{IC}(\cM^{\nu^{\bullet,\RR}}(\lambda,\mu)))
    \cong \psi\circ\Phi (\mathrm{IC}(\underline{\cM}(\lambda,\mu)))
      \\
    \cong \; &
    \bigoplus_{\substack{\mu=\mu^1+\mu^2\\
        \cM(\lambda^1,{\mu^1})^T, \cM(\lambda^2,{\mu^2})^T\neq
        \emptyset}} \Phi(\mathrm{IC}(\cM(\lambda^1,\mu^1))) \otimes
    \Phi(\mathrm{IC}(\cM(\lambda^2,\mu^2))).
  \end{split}
\end{equation*}
The first isomorphism is given by the triviality of the family
$\underline{\widetilde\cM}(\lambda,\mu)\to \CC$ and the commutativity
of the nearby cycle and hyperbolic restriction functors by
\cite{MR3912059}.
The second isomorphism is given by the factorization in
\cref{prop:deformed}.
\end{Corollary}

Recall we denote $\Phi(\mathrm{IC}(\cM(\lambda,\mu)))$ by
$\mathcal V_\mu(\lambda)$. Since
$\pi\colon\cM^{\nu^{\bullet,\RR}}(\lambda,\mu)\to \cM(\lambda,\mu)$ is
an isomorphism over the open locus $\cM^{\mathrm{s}}(\lambda,\mu)$,
the direct image
$\pi_*(\mathrm{IC}(\cM^{\nu^{\bullet,\RR}}(\lambda,\mu)))$ contains
$\mathrm{IC}(\cM(\lambda,\mu))$ as a direct summand with multiplicity
one. Therefore we have natural homomorphisms, inclusion and projection
\(
\mathrm{IC}(\cM(\lambda,\mu))\leftrightarrows
\pi_*(\mathrm{IC}(\cM^{\nu^{\bullet,\RR}}(\lambda,\mu))).
\)
They induce
\begin{equation}\label{eq:9}
  \bigoplus_\mu \Phi(\mathrm{IC}(\cM(\lambda,\mu))) \leftrightarrows
  \bigoplus_\mu\Phi\circ\pi_*(\mathrm{IC}(\cM^{\nu^{\bullet,\RR}}(\lambda,\mu))).
\end{equation}
\begin{NB}
  The original version:
\begin{equation*}
    \mathcal V_\mu(\lambda) \leftrightarrows
  \bigoplus_{\mu=\mu^1+\mu^2}
  \mathcal V_{\mu^1}(\lambda^1) \otimes \mathcal V_{\mu^2}(\lambda^2), \qquad
  \mathcal V(\lambda) \leftrightarrows
  \mathcal V(\lambda^1) \otimes \mathcal V(\lambda^2).
\end{equation*}
\end{NB}%

\subsection{Weyl group action}\label{subsec:Weyl}

Let us consider a real parameter 
with 
$\nu_*^{\circ,\RR} = 0$ but other 
$\nu^{\circ,\RR}_h$ are generic. We consider
$\cM^{\nu^{\circ,\RR}}(\lambda,\mu)\to \cM(\lambda,\mu)$.
As in \cref{cor:tensor} we have
\begin{equation}\label{eq:h:6}
  \Phi\circ\pi_*(\mathrm{IC}(\cM^{\nu^{\circ,\RR}}(\lambda,\mu)))
  \cong
  \bigoplus_{\sum \mu_i^{a_i} = \mu}
  \bigotimes_{i=0}^{n-1} \bigotimes_{a_i=1}^{\bw_i}
  \Phi(\mathrm{IC}(\cM(\Lambda_i,\mu_i^{a_i}))).
\end{equation}

We divide the hamiltonian reduction in the definition of a balanced
bow variety in two steps: the first by $\GL(V_\zeta)$'s when both ends
of $\zeta$ are $\boldsymbol\medcirc$, and the second by the remaining
$\GL(V_\zeta)$'s. In the first step, we obtain products of triangle
parts and quiver varieties of type $A_{\bw_i-1}$ with dimension
vectors $(\bv_i,\bv_i,\dots,\bv_i, \bv_i)$,
$(\bv_i,0,\dots,0,\bv_i)$.
For $\zeta^{\nu^{\circ,\RR}}$ this first step gives the inverse image
of Slodowy slice $\mathcal S((\bw_i-1)^{\bv_i},1^{\bv_i})$ under the
projection $\pi\colon T^*\mathcal F
\to \overline{\mathcal
  N}(\bw_i^{\bv_i})$.  Here
$((\bw_i-1)^{\bv_i},1^{\bv_i})$,
$(\bw_i^{\bv_i})$ denote Jordan type of
nilpotent orbits, and
$T^*\mathcal F
$ is the cotangent bundle of the partial flag variety of flags of
subspaces of $\CC^{\bw_i\bv_i}$ of dimensions
$\bv_i$, $2\bv_i$, \dots,
$(\bw_i-1)\bv_i$.
For $\cM(\lambda,\mu)$ the first step gives the intersection of
Slodowy slice and $\overline{\mathcal N}(\bw_i^{\bv_i})$. (See e.g.,
\cite[\S7.4]{2016arXiv160602002N}.)

The projection
$\pi\colon T^*\mathcal
F
\to \overline{\mathcal N}(\bw_i^{\bv_i})$ is very similar to the
Springer resolution. (In fact, it is exactly the Springer resolution
when $\bv_i = 1$.)
Recall that the Springer representation is a
$\mathfrak S_{\bw_i}$-action on
$\pi_*(\QQ_{T^*\mathcal F}[\dim T^*\mathcal F])$ for $\bv_i = 1$ in
Lusztig's construction \cite{Lu-Green} (see also \cite[\S2.6]{BM}).

\begin{Lemma}
  \textup{(1)} A perverse sheaf
  $\pi_*(\mathrm{IC}(\cM^{\nu^{\circ,\RR}}(\lambda,\mu)))$ is equipped
  with a $\prod_i \mathfrak S_{\bw_i}$-action.

  \textup{(2)} The induced $\prod_i\mathfrak S_{\bw_i}$-action on the
  hyperbolic restriction $
  \Phi\circ\pi_*(\mathrm{IC}(\cM^{\nu^{\circ,\RR}}(\lambda,\mu)))$
  coincides with the permutation among factors $a_i=1,\dots,\bw_i$ with
  common $i$ in \eqref{eq:h:6}.
\end{Lemma}

Taking the sum over $\mu$, we rewrite the right hand side of
\eqref{eq:h:6} as
\begin{equation*}
  \bigotimes_{i=0}^{n-1} \mathcal V(\Lambda_i)^{\otimes \bw_i},
\end{equation*}
as
$\mathcal V(\Lambda_i) = \bigoplus_{\mu_i^{a_i}}
\Phi(\mathrm{IC}(\cM(\Lambda_i,\mu_i^{a_i})))$. Thus this analog of
Springer representation permutes tensor factors.

\begin{proof}
  (1) We follow Lusztig's construction. We consider the family
  $\underline{\widetilde\cM}(\lambda,\mu)\to\BA^{\ell-1}$ where the
  real parameter is the one we already choose $\nu^{\circ,\RR}$.
  Complex parameters $\nu^\CC$ have $\nu^\CC_* = 0$, but other parts
  $\nu_h^\CC$ are arbitrary. The number of parameters is
  $\sum_i \bw_i - 1 = \ell-1$. (Recall that an overall shift of
  $\nu^\CC_h$ is irrelevant.)
  For $\underline{\cM}(\lambda,\mu)\to\BA^{\ell-1}$, we choose complex
  parameter as above while the real parameter $\nu^\RR$ is $0$.
  The product of symmetric groups $\prod_i \mathfrak S_{\bw_i}$ acts
  on $\BA^{\ell-1}$, where $\mathfrak S_{\bw_i}$ permutes
  $\nu_h^\CC$'s in the arc $\xl_i\to\xl_{i+1}$.
  This action can be lifted to $\underline{\cM}(\lambda,\mu)$ by
  reflection functors \cite{Na-reflect} as we have already mentioned
  in \cref{rem:reflect}. Alternatively we use the identification of
  $\underline{\cM}(\lambda,\mu)$ with the spectrum of
  $H^{\tilde\bG_\cO}_*(\cR)$, and observe that it has a quotient by
  $\prod_i \mathfrak S_{\bw_i}$ given by
  $\Spec H^{\overline{\bG}_\cO}_*(\cR)$ with
  $\overline{\bG} = (\bG\times \prod_i\GL(\bw_i))/\CC^\times$. We now
  have a commutative diagram
  \begin{equation*}
    \begin{tikzcd}
      \cM^{\nu^{\circ,\RR}}(\lambda,\mu) \arrow[r,"\tilde{\iota}"]
      \arrow[d,"\pi"] &
      \underline{\widetilde\cM}(\lambda,\mu)\arrow[r]
      \arrow[d,"\underline{\pi}"]
      & \BA^{\ell-1}\arrow[d] \\
      \cM(\lambda,\mu) \arrow[r,"\iota"] &
      \underline{\cM}(\lambda,\mu)/\prod_i \mathfrak S_{\bw_i}
      \arrow[r] & \BA^{\ell-1}/\prod_i \mathfrak S_{\bw_i},
    \end{tikzcd}
  \end{equation*}
  where the leftmost column consists of fibers over $0$.

  Since the left square is Cartesian, we have
  $\iota^*\underline{\pi}_*(\mathrm{IC}(\underline{\widetilde\cM}(\lambda,\mu)))
  \cong \pi_*
  \tilde\iota^*(\mathrm{IC}(\underline{\widetilde\cM}(\lambda,\mu)))$
  by base change.
  Since $\underline{\widetilde\cM}(\lambda,\mu)$ is a topologically
  trivial family as explained in \cref{subsec:deformed-case}, we have
  $\tilde\iota^*(\mathrm{IC}(\underline{\widetilde\cM}(\lambda,\mu)))[1-\ell]
  \cong \mathrm{IC}(\cM^{\nu^{\circ,\RR}}(\lambda,\mu))$.
  On the other hand, $\underline{\pi}$ is a stratified small
  birational morphism again as in \cref{subsec:deformed-case}. Hence
  $\underline{\pi}_*(\mathrm{IC}(\underline{\widetilde\cM}(\lambda,\mu)))$
  is the $\mathrm{IC}$ associated with the local system given by the
  $\prod_i \mathfrak S_{\bw_i}$-covering defined by the the
  restriction of $\underline\pi$ to the open subset of
  $\BA^{\ell-1}/\prod_i \mathfrak S_{\bw_i}$ consisting of disjoint
  configurations. In particular,
  $\underline{\pi}_*(\mathrm{IC}(\underline{\widetilde\cM}(\lambda,\mu)))$,
  and hence also
  $\iota^*\underline{\pi}_*(\mathrm{IC}(
  \underline{\widetilde\cM}(\lambda,\mu)))[1-\ell] =
  \pi_*(\mathrm{IC}(\cM^{\nu^{\circ,\RR}}(\lambda,\mu)))$ is equipped
  with a $\prod_i \mathfrak S_{\bw_i}$-action.

  (2) Since $\Phi$ is a functor, we have a
  $\prod_i \mathfrak S_{\bw_i}$-action on
  $\Phi\circ\pi_*(\mathrm{IC}(\cM^{\nu^{\circ,\RR}}(\lambda,\mu))) =
  \Phi\circ
  \tilde\iota^*(\mathrm{IC}(\underline{\widetilde\cM}(\lambda,\mu)))[1-\ell]
  $. We exchange $\Phi$ and $\tilde\iota^*$ after we replace
  $\Phi = p_* j^!$ by $p^-_! (j^-)^{*}$ by \cite{Braden} and use the
  base change. As in \cref{prop:deformed} the fixed point components
  in $\underline{\widetilde\cM}(\lambda,\mu)^T$ correspond to
  decomposition $\mu = \sum \mu_i^{a_i}$. Moreover each factor
  corresponds to $h$, as $\nu^\CC_h$'s take different values on the
  open subset of disjoint configurations. Therefore
  $\prod_i \mathfrak S_{\bw_i}$-action permutes fixed point components
  according to permutations of $\mu_i^{a_i}$'s. Note also
  $\tilde\iota^*(\mathrm{IC}(\underline{\widetilde\cM}(\lambda,\mu)))[1-\ell]$
  and
  $\psi(\mathrm{IC}(\underline{\widetilde\cM}(\lambda,\mu))
  |_{\CC\nu^{\bullet,\CC}})$ are isomorphic as
  $\underline{\widetilde\cM}(\lambda,\mu))$ is a topologically trivial
  family. The same is true for those applied with $\Phi$. Therefore
  \eqref{eq:h:6} is given also by $\tilde\iota^*[1-\ell]$.
\end{proof}

\begin{NB}
  Old version:

  one parameter family parametrized by $\CC$ where
  the fiber at $0$ (resp.\ $1$) is
%
Recall that the stratification of the bow variety $\cM$ is induced
from that of two way part, which is a quiver variety (see the proof of
\cite[Th.~7.13]{2016arXiv160602002N}).
Also, $\pi$ for $\cM$ and for the quiver variety are compatible with
the hamiltonian reduction, hence we have a
$\prod_i \mathfrak S_{\bw_i}$ action on
$\pi_*(\mathrm{IC}(\cM^{\nu^{\circ,\RR}}(\lambda,\mu)))$. It induces
the action on its hyperbolic restriction. Taking the sum over $\mu$,
it is
\begin{equation*}
  \bigotimes_{i=0}^{n-1} \mathcal V(\Lambda_i)^{\otimes \bw_i}
\end{equation*}
by the factorization as in \cref{subsec:deformed-case}. Since the
Springer representation comes from the permutation of eigenvalues of
$B_\zeta$'s, $\prod_i \mathfrak S_{\bw_i}$ acts by
permutation of factors.
\end{NB}
\subsection{Weyl chambers on the space of one parameter subgroups}

Let us take a one-parameter subgroup
$\chi(t) = (t^{m_0},\dots, t^{m_{n-1}})\in T$ ($m_i\in\ZZ$). We
consider the corresponding fixed point
$\cM(\lambda,\mu)^{\chi}$. Since we can replace $t$ by
$t^N$ for $N\in\ZZ\setminus\{0\}$, we can also consider
$m_i\in\QQ$. We regard $\underline{m} = (m_0,\dots,m_{n-1})\in\QQ^n$
as an element of the Cartan subalgebra of the affine Lie algebra
$\widehat{\algsl}(n)$ (without the degree operator). Recall that the
roots of $\widehat{\algsl}(n)$ are $(k,k,\dots,k) = k\delta$
($k\neq 0$) and
$\pm(0,\dots,0,\underset{j}1,\dots,1,\underset{i}0,\dots,0) + k\delta$
($1\le j < i\le n$, $k\in\ZZ$).

\begin{Lemma}\label{lem:weyl-chambers}
  Suppose that $\langle\underline{m}, \alpha\rangle \neq 0$ for any
  root $\alpha$ of $\widehat{\algsl}(n)$. Then
  $\cM(\lambda,\mu)^{\chi} = \cM(\lambda,\mu)^T$.
\end{Lemma}

\begin{proof}
  Recall a stratification
  $\cM(\lambda,\mu) = \bigsqcup \cM^{\mathrm{s}}(\kappa,\mu)\times
  S^{\underline{k}}(\CC^2\setminus\{0\}/(\ZZ/\ell\ZZ))$ ($\ell > 1$),
  $\bigsqcup \cM^{\mathrm{s}}(\kappa,\mu)\times
  S^{\underline{k}}(\CC^2)$ ($\ell = 1$).
  The action on symmetric products is induced from the
  $\CC^\times$-action on $\CC^2$ given by
  $t\cdot (x,y) = (tx, t^{-1}y)$ where $t=s_0\cdots s_{n-1}$. We have
\begin{equation*}
  (\CC^2)^T = (\CC^2)^{\chi} \Longleftrightarrow
  0\neq \sum_{i=0}^{n-1} m_i = \langle \underline{m}, \delta\rangle.
\end{equation*}
Thus $T$-fixed points and $\chi(\CC^\times)$-fixed points in symmetric
products are the same if $\underline{m}$ is not in the imaginary root
hyperplane $\delta = 0$.

Next we consider $\cM^{\mathrm{s}}(\kappa,\mu)^{\chi}$ as
in \cref{subsec:torus-fixed-points}. If
\begin{equation*}
  (s_0\cdots s_{i-1}) (s_0\cdots s_{j-1})^{-1} \notin (s_0\cdots s_{n-1})^\ZZ
\end{equation*}
on the image of $\chi$, data for $\CC_{x_i}$ and $\CC_{x_j}$ live on
different weight spaces, hence we have a direct sum decomposition as
before. The above condition is
\begin{equation*}
  m_{i-1} + m_{i-2} + \dots + m_{j} \notin \ZZ \sum_{i=0}^{n-1} m_i,
\end{equation*}
if $i > j$ and similar for $i < j$. This means that $\underline{m}$ is
not in root hyperplanes
\[(0,\dots,0,\underset{j}1,\dots,1,\underset{i}0,\dots,0) + k\delta\]
for any $k$.
\end{proof}

Thus we get a geometric interpretation of Weyl chambers in terms of
one parameter subgroups.

\subsection{Fixed points with respect to smaller tori}\label{subsec:smaller}

Let us consider the `negative' chamber $m_i < 0$ for $i=0,\dots,n-1$
on the space of one parameter subgroups. It does not intersect with
root hyperplanes, hence the fixed point set with respect to $\chi(t) =
(t^{m_0},\dots, t^{m_{n-1}})$ coincides with the $T$-fixed point
set. Let us consider a one parameter subgroup $\chi_i(t)$, which lies
on the boundary of the negative chamber.

\begin{Theorem}\label{thm:smaller}
    Take a one parameter subgroup $\chi_i(t)$ with $m_i = 0$ for some
    $i$ and $m_j < 0$ otherwise.

    \textup{(1)} The $\chi_i$-fixed point set $\cM(\lambda,\mu)^{\chi_i}$
    is either empty or isomorphic to a Coulomb branch
    $\cM_{A_1}(\lambda',\mu')$ of an $A_1$ type framed quiver gauge
    theory with weights $\lambda'$, $\mu'$, where
    $\mu' = \langle\mu, h_i\rangle = \mu_i - \mu_{i+1}$
    \textup($i\neq 0$\textup), $\ell+\mu_n - \mu_1$
    \textup($i=0$\textup). Moreover the intersection of
    $\cM(\lambda,\mu)^{\chi_i}$ with a stratum is either empty or a
    stratum of $\cM_{A_1}(\lambda',\mu')$.

    \textup{(2)} The restrictions of $w_{j,r}$, $\mathsf y_{j,r}$ are
    zero for $j\neq i$, and are equal to ones for
    $\cM_{A_1}(\lambda',\mu')$ or zero for $j = i$.
    \begin{NB}
      The restriction of the $i$-th component of the factorization
      morphism $\varpi$ of $\cM(\lambda,\mu)$ to
      $\cM_{A_1}(\lambda',\mu')$ is equal to the factorization
      morphism of $\cM_{A_1}(\lambda',\mu')$ up to adding $0$.
    \end{NB}%
\end{Theorem}

The weight $\lambda'$ is determined from $\lambda$, $\mu$ in a
combinatorial way as we will see during the proof. On the other hand,
it is the largest highest weight with $\ge \mu'$ among those
corresponding $\algsl(2)_i$-modules appearing the restriction of the
integrable highest weight module $V(\lambda)$, once we establish
$\mathcal V(\lambda) = V(\lambda)$.
\begin{NB}
  Note that $\cM_{A_1}(\lambda',\mu')^T$ could be empty.
  From the argument in \cref{prop:naturaliso} we see that the fixed
  point set $\cM_{A_1}(\lambda',\mu')$ appears if and only if
  $\cM_{A_1}(\lambda',\lambda')$ appears in the fixed point in
  $\cM(\lambda,\mu+(\lambda'-\mu')\alpha_i/2)$. Since
  $\cM_{A_1}(\lambda',\lambda')^T\neq\emptyset$, this is equivalent to
  the existence of the corresponding $\algsl(2)_i$-modules.
\end{NB}%
But it is not clear to the author how to show that the combinatorial
expression of $\lambda'$ coincides with the representation theoretic
characterization directly.\footnote{This sentence was written in the
  first version of the manuscript. See \cref{rem:char2} for the update.}

\begin{proof}
Let us first suppose $i\neq 0,n-1$. The same argument as in
\cref{subsec:torus-fixed-points} shows that a point
$\cM^{\mathrm{s}}(\kappa,\mu)^{\chi_i}$ is represented by
$(A,B,C,D,a,b)$ which decomposes to data for $\CC_{\xl_j}$
($j\neq i,i+1$) and data for $\CC_{\xl_i}\oplus \CC_{\xl_{i+1}}$. 
We already know that the former data gives a single point by \cref{subsec:torus-fixed-points}.
We untwist the circle to the line as before, and the bow diagram for
the latter is
\begin{equation}\label{eq:1}
\begin{tikzpicture}[baseline=(current  bounding  box.center)]
  \node at (0,0) {$\vphantom{j^X}$};
  \node[label=below:$\xl_{i+1}$,label=above:$\mu_{i+1}$] at (0.5,0)
        {$\vphantom{j^X}\boldsymbol\times$};
  \node[label=below:$\xl_{i}$,label=above:$\mu_i$] at (-0.5,0)
  {$\vphantom{j^X}\boldsymbol\times$};
  \node[label=below:$h_0$] at (1.5,0)
        {$\vphantom{j^X}\boldsymbol\medcirc$};
  \node[label=below:$h_1$] at (-1.5,0)
        {$\vphantom{j^X}\boldsymbol\medcirc$};
  \node[label=below:$h_{-1}$] at (2.5,0)
        {$\vphantom{j^X}\boldsymbol\medcirc$};
  \node[label=below:$h_2$] at (-2.5,0)
        {$\vphantom{j^X}\boldsymbol\medcirc$};
  \node[label=below:$\cdots$] at (3.5,0)
        {$\vphantom{j^X}$};
  \node[label=below:$\cdots$] at (-3.5,0)
        {$\vphantom{j^X}$};
  \node[label=below:$h_{-n}$] at (4.5,0)
        {$\vphantom{j^X}\boldsymbol\medcirc$};
  \node[label=below:$h_m$] at (-4.5,0)
        {$\vphantom{j^X}\boldsymbol\medcirc$};
  \node[label=below:$\cdots$] at (5.5,0)
        {$\vphantom{j^X}\boldsymbol\medcirc$};
  \node[label=below:$\cdots$] at (-5.5,0)
        {$\vphantom{j^X}\boldsymbol\medcirc$};
  \draw[-] (-6,0) -- (6,0);
\end{tikzpicture}.
\end{equation}
When $i=n-1$, we have $s_0\cdots s_{n-2} = s_0\cdots s_{n-1}$. Then
the action on $\CC_{\xl_{n-1}}$ and that on $A$, $b$ at $\xl_0$ have
the same weight. The argument above yields the same diagram above, if we understand $(i,i+1) = (n-1, 0)$ and $\mu_0 = \mu_n$.

When $i=0$, we have $\ell$ $\boldsymbol\medcirc$'s between $\xl_0$ and
$\xl_1$. (See \eqref{eq:16}.) So we get
\begin{equation*}
\begin{tikzpicture}[baseline=(current  bounding  box.center)]
  \node at (0,0) {$\vphantom{j^X}$};
  \node[label=below:$\xl_{1}$,label=above:$\mu_{1}$] at (2,0)
        {$\vphantom{j^X}\boldsymbol\times$};
  \node[label=below:$\xl_{0}$,label=above:$\mu_n$] at (-2,0)
  {$\vphantom{j^X}\boldsymbol\times$};
  \node[label=below:$h_1$] at (1,0)
        {$\vphantom{j^X}\boldsymbol\medcirc$};
  \node[label=below:$\cdots$] at (0,0)
        {$\vphantom{j^X}$};
  \node[label=below:$h_{\ell}$] at (-1,0)
        {$\vphantom{j^X}\boldsymbol\medcirc$};
  \node[label=below:$h_0$] at (3,0)
        {$\vphantom{j^X}\boldsymbol\medcirc$};
  \node[label=below:$h_{\ell+1}$] at (-3,0)
        {$\vphantom{j^X}\boldsymbol\medcirc$};
  \node[label=below:$\cdots$] at (4,0)
        {$\vphantom{j^X}$};
  \node[label=below:$\cdots$] at (-4,0)
        {$\vphantom{j^X}$};
  \draw[-] (-5,0) -- (5,0);
\end{tikzpicture}.
\end{equation*}
Then we move $\ell$ $\boldsymbol\medcirc$'s to the left of
$\xl_0$ by Hanany-Witten transition to get the same diagram in
\eqref{eq:1} with $\mu_n$ replaced by $\mu_n + \ell$.
The argument below remains if we understand $\mu_{i=0}$ is
$\mu_n+\ell$.

We assume $\mu_i$, $\mu_{i+1} \ge 0$ as in the proof of \cref{prop:torus-fixed-points}.
We have $N(h_\sigma,h_{\sigma+1}) \ge 0$ for any $\sigma$ again as
before. On the other hand, $\sum N(h_\sigma,h_{\sigma+1}) = 2$ by
definition. Therefore $N(h_\sigma,h_{\sigma+1})\neq 0$ at most two
$\sigma$'s.

\begin{NB}
Let us first suppose $N(h_\tau,h_{\tau+1}) = 2$ for some $\tau$. Then
$N(h_\sigma,h_{\sigma+1}) = 0$ for $\sigma\neq\tau$. We move $\xl_i$,
$\xl_{i+1}$ to the interval between $h_\sigma$ and
$h_{\sigma+1}$. Since $N(h_\sigma,h_{\sigma+1})$ is preserved
(including $\sigma=\tau$), we have $N_{h_\sigma}$ is changed to $0$
for any $\sigma$. Therefore nonzero vector space appears only possibly
between $\xl_i$ and $\xl_{i+1}$. Let $\bv$ be its dimension. It is equal
to $-\mu_i+\tau = \mu_{i+1}-\tau$.
If $\bv > 0$, it is easy to see that the bow variety is nonempty, but
has no $\CC^\times$-fixed point. (It is isomorphic to the Coulomb
branch of the gauge theory $\GL(\bv)$ with $\bN=0$.)
If $\bv=0$, the bow variety is a single point, which is fixed by
$\CC^\times$.
If $\tau > 0$, the original data (before Hanany-Witten transition)
look like
\begin{equation*}
  \xymatrix@C=1.2em{
    \CC^2 \ar@<.5ex>[rr] && \CC^4 \ar@<.5ex>[ll] \ar@<.5ex>[rr]
    && \cdots \ar@<.5ex>[ll] \ar@<.5ex>[rr]
    && \CC^{2\tau} \ar@<.5ex>[ll] \ar[rr] \ar[dr]
    && \CC^{\tau+\bv} \ar[dr]
    &&&
  \\
    &&&&&&& \CC\ar[ur] && \CC}
\end{equation*}
If $\tau < 0$, the data is the mirror image of above with $\tau$
replaced by $-\tau$. Note that $\tau$ and $v$ are determined from
$\mu_i$, $\mu_{i+1}$ by
$\tau = (\mu_i + \mu_{i+1})/2$, $v = (-\mu_i + \mu_{i+1})/2$.

Next suppose
$N(h_{\tau_1},h_{\tau_1+1}) = 1 = N(h_{\tau_2}, h_{\tau_2+1})$ with
$\tau_1 > \tau_2$. Then $N(h_\sigma,h_{\sigma+1}) = 0$ for other
$\sigma$.
\end{NB}%

Let us suppose $N(h_\sigma,h_{\sigma+1})\neq 0$ for $\sigma = \tau_1$,
$\tau_2$ and $N(h_\sigma,h_{\sigma+1}) = 0$ otherwise. We assume
$\tau_1\ge\tau_2$. We have
$N(h_{\tau_1},h_{\tau_1+1}) = 1 = N(h_{\tau_2},h_{\tau_2+1})$ if
$\tau_1 > \tau_2$, $N(h_{\tau_1},h_{\tau_1+1}) = 2$ if
$\tau_1=\tau_2$.
We move $\xl_i$ (resp.\ $\xl_{i+1}$) between $h_{\tau_1}$ and
$h_{\tau_1+1}$ (resp.\ $h_{\tau_2}$ and $h_{\tau_2+1}$) by successive
applications of Hanany-Witten transition. Then $N_{h_\sigma}$ becomes
$0$ for any $\sigma$. Thus vector spaces appear between $\xl_i$ and
$\xl_{i+1}$ with the same dimension (let it be $\bv$), and all others
are $0$. See
\begin{equation*}
\begin{tikzpicture}[baseline=(current  bounding  box.center)]
  \node at (0,0) {$\vphantom{j^X}$};
  \node[label=below:$h_{\tau_2}$,label=above:$0$] at (-0.5,0)
  {$\vphantom{j^X}\boldsymbol\medcirc$};
  \node[label=above:$\bv$] at (-1.5,0)
        {$\vphantom{j^X}\boldsymbol\times$};
  \node[label=above:$0$,label=below:$h_{\tau_2\!+\! 1}$] at (-2.5,0)
        {$\vphantom{j^X}\boldsymbol\medcirc$};
  \node[label=above:$\cdots$] at (-3.5,0)
        {$\vphantom{j^X}$};
  \node[label=above:$0$,label=below:$h_{\tau_1}$] at (-4.5,0)
        {$\vphantom{j^X}\boldsymbol\medcirc$};
  \node[label=above:$-\bv$] at (-5.5,0)
        {$\vphantom{j^X}\boldsymbol\times$};
  \node[label=above:$0$,label=below:$h_{\tau_1\!+\! 1}$] at (-6.5,0)
        {$\vphantom{j^X}\boldsymbol\medcirc$};
  \draw[-] (-7,0) -- (0,0);
\end{tikzpicture}
\end{equation*}
The balanced condition is satisfied, hence it gives a Coulomb
branch. The gauge theory is of type $A_1$ with dimensions $\bv$,
$\bw$, where $\bw = \tau_1 - \tau_2$.
It is nonempty if and only if $\bv \ge 0$. 
\begin{NB}
    It has a $\CC^\times$-fixed point if $\bv\le\bw$.
\end{NB}%
More precisely, we consider the fixed point set in a stratum
$\cM^{\mathrm{s}}(\kappa,\mu)$, hence the data above must satisfy the
$0$-stability condition. Therefore the fixed point set is the open
stratum of the Coulomb branch.

Note that $\tau_1$, $\tau_2$ are determined by $\mu_i$, $\mu_{i+1}$
and $\bv$ as
\begin{equation*}
  \tau_1 =
  \begin{cases}
      \mu_i + \bv & (i\neq 0), \\
      \mu_n + \ell + \bv & (i = 0),
  \end{cases}
  \qquad \tau_2 = \mu_{i+1} - \bv.
  \begin{NB}
    \mu_i - \tau_1 = -\bv, \qquad \mu_{i+1} - \tau_2 = \bv.
  \end{NB}%
\end{equation*}
In particular, we have $\bw - 2\bv = \mu_i - \mu_{i+1}$ ($i\neq 0$),
$\ell + \mu_n - \mu_1$ ($i=0$).
\begin{NB}
    $\bw - 2\bv = \tau_1 - \tau_2 + \mu_i - \tau_1 - (\mu_{i+1} -
    \tau_2) = \mu_i - \mu_{i+1}$ ($i\neq 0$) and
    $=  \ell + \mu_n - \mu_1$ ($i=0$).
\end{NB}%
This is determined only by $\mu_i$, $\mu_{i+1}$, hence only by the bow
data for the original bow variety $\cM(\lambda,\mu)$.
\begin{NB}
  $\mu_i + \bv = \tau_1 \ge \tau_2 = \mu_{i+1} - \bv$. Thus
  $2\bv \ge \mu_{i+1} - \mu_i$. Or we have
  $2\bv = \bw - (\mu_i - \mu_{i+1}) \ge \mu_{i+1} - \mu_i$.
  In particular $\bv = 0$ is not possible if $\mu_{i+1} > \mu_i$.
\end{NB}%

Let us return back to the data \eqref{eq:1}, and then to \eqref{eq:16}
in order to see how strata of $\cM(\lambda,\mu)$ and the above are
related.

If $\tau_1 > \tau_2 > 0$, the data looks like
\begin{NB}
$0 = \cdots = N_{h_{\tau_1+2}} = N_{h_{\tau_1+1}}$,
$1 = N_{h_{\tau_1}} = \cdots = N_{h_{\tau_2+1}}$,
$2 = N_{h_{\tau_2}} = \cdots = N_{h_1}$,
$0 = N_{h_0} = N_{h_{-1}} = \cdots$.
\end{NB}%
\begin{equation}\label{eq:h:1}
  \begin{tikzcd}[column sep=1ex]
    \CC \arrow[shift left=1, rr] &&
    \CC^2 \arrow[shift left=1, rr] \arrow[shift left=1, ll] &&
    \cdots \arrow[shift left=1, ll] \arrow[shift left=1, rr] &&
    \CC^{\tau_1-\tau_2} \arrow[shift left=1, ll] \arrow[shift left=1, rr] &&
    \CC^{\tau_1-\tau_2+2} \arrow[shift left=1, ll] \arrow[shift left=1, rr] &&
    \cdots \arrow[shift left=1, ll] \arrow[shift left=1, rr] &&
    \CC^{\tau_1+\tau_2}
    \arrow[out=120,in=60,loop,looseness=3,"B_{\vout{\xl_i}}"] \arrow[shift left=1, ll]
    \arrow[rr] \arrow[rd,"b_{\xl_i}"'] &&
    \CC^{\tau_2+\bv}
    \arrow[out=120,in=60,loop,looseness=3,"B_{\vin{\xl_i}}"]
    \arrow[rd,"b_{\xl_{i+1}}"] &
    \\
    &&&&&&&&&&&&&\CC_{\xl_i} \arrow[ru,"a_{\xl_i}"'] && \CC_{\xl_{i+1}}.
  \end{tikzcd}
\end{equation}
From left to right, vector spaces increase dimension by $1$ until
$\tau_1 - \tau_2$. Then increase by $2$ until $\tau_1+\tau_2$.
If $\tau_1 = \tau_2$, the left edge starts as
\(
\begin{tikzcd}[column sep=2.4ex]
  \CC^2 \arrow[shift left=.5, r] &
  \CC^4 \arrow[shift left=.5, l] \arrow[shift left=.5, r] &
  \cdots \arrow[shift left=.5, l].
\end{tikzcd}
\)
\begin{NB}
    $\tau_1 = \tau_2$ if and only if $2\bv = \mu_{i+1} - \mu_i$. As we
    remarked above, $\bv$ cannot be smaller than $(\mu_{i+1} - \mu_i)/2$.
\end{NB}%

\begin{NB}
  We can apply the Hanany-Witten transition to move two
  $\boldsymbol\times$ to the left to arrive at
  \begin{equation*}
  \begin{tikzcd}[column sep=1.8ex]
    \CC \arrow[shift left=1, rr] &&
    \CC^2 \arrow[shift left=1, ll] \arrow[shift left=1, rr] &&
    \cdots \arrow[shift left=1, ll]
    \arrow[shift left=1, rr] &&
    \CC^{\bw-1} \arrow[shift left=1, ll]
    \arrow[shift left=1, rr] &&
    \CC^{\bw}
    \arrow[out=120,in=60,loop,looseness=3]
    \arrow[shift left=1, ll]
    \arrow[rr] \arrow[rd] &&
    \CC^{\bv}
    \arrow[out=120,in=60,loop,looseness=3] \arrow[rd] &
    \\
    &&&&&&&&&\CC \arrow[ru] && \CC.
  \end{tikzcd}
\end{equation*}
with $\bw = \tau_1 - \tau_2$.
\end{NB}%
Note $\tau_1+\tau_2 = \mu_i + \mu_{i+1}$, $\tau_2+\bv = \mu_{i+1}$,
$\tau_1 - \tau_2 =
\begin{NB}
  \mu_i + \bv - (\mu_{i+1} - \bv) =
\end{NB}%
2\bv + \mu_i - \mu_{i+1}$. In particular, $\tau_1+\tau_2$ and
$\tau_2+\bv$ are determined by $\mu$. Going to \eqref{eq:16}, we see
that $\tau_1 - \tau_2$ (and hence $\bv$) is determined from $\kappa$,
$\mu$. In fact, the sum of dimension of vector spaces in two way parts
is
\(
\frac14 (\tau_1+\tau_2)^2 + \frac12 (\tau_1+\tau_2)
+ \frac14 (\tau_1-\tau_2)^2.
\)
Hence $\tau_1 - \tau_2$ is determined by dimension vectors in the form
\eqref{eq:16}.
\begin{NB}
  Let $k$ be $\tau_1 - \tau_2$. From the left, we get
  $1 + \dots + k = \frac{k(k+1)}2$. From the remaining part, we get
  $(k+2) + (k+4) + \dots + (k + 2\tau_2) = k \tau_2 +
  \tau_2(\tau_2+1) = \tau_2 (k + \tau_2 + 1)$. Since
  $\tau_2 = (\tau_1+\tau_2 - (\tau_1-\tau_2))/2 = (\tau_1+\tau_2 -
  k)/2$, the later is
  \begin{equation*}
    \frac14(\tau_1+\tau_2 - k)(\tau_1+\tau_2 + k + 2)
    = \frac14 \left((\tau_1+\tau_2)^2 - k^2\right)
    + \frac12 (\tau_1+\tau_2 - k).
  \end{equation*}
  So in total, it is equal to
  \begin{equation*}
    \frac14 (\tau_1+\tau_2)^2 + \frac12 (\tau_1+\tau_2) + \frac14 k^2.
  \end{equation*}

  For example, $\tau_1+\tau_2 = 2$, $\tau_1 - \tau_2 = 0$, we have
  $2$, while the above gives $1 + 1 = 2$.
  For example, $\tau_1+\tau_2 = 3$, $\tau_1 - \tau_2 = 1$, we have
  $1+3=4$, while the above gives $9/4 + 3/2 + 1/4 = 4$.
  For example, $\tau_1+\tau_2 = 4$, $\tau_1 - \tau_2 = 2$, we have
  $1 + 2 + 4 = 7$, while the above gives
  $4 + 2 + 1 = 7$.
\end{NB}%

\begin{NB}
  Note also ${}^t\kappa_1$, \dots, ${}^t\kappa_\ell$ are obtained from
  $N_{h_\xp}$ for the above diagram, i.e.,
  $(\underbrace{1,\dots,1}_{\text{$\tau_1-\tau_2$ times}},
  \underbrace{2,\dots,2}_{\text{$\tau_2$ times}})$ by folding modulo
  $\ell$.
\end{NB}%

If we add $2$ to $\bw$ keeping $\bw-2\bv$ unchanged, we go to a larger
stratum in $\cM_{A_1}$.
\begin{NB}
  $\cM_{A_1}^{\mathrm{s}}(\bv,\bw)\subset \cM_{A_1}(\bv+1,\bw+2)$.
\end{NB}%
It corresponds to adding
\(
\begin{tikzcd}[column sep=2.4ex]
  \CC \arrow[shift left=.5, r] &
  \cdots \arrow[shift left=.5, l] \arrow[shift left=.5, r] &
  \CC \arrow[shift left=.5, l] & 
\end{tikzcd}
\) ($\tau_1 - \tau_2 + 1$ copies of $\CC$). It goes to a larger
stratum in $\cM(\lambda,\mu)$ also.

\begin{NB}
This case does not arise, as $\tau_1 = \mu_i + \bv \ge 0$.
  
If $0 \ge \tau_1 \ge \tau_2$, the data looks like
\begin{equation*}
  \begin{tikzcd}[column sep=1.8ex]
    & \CC^{\bv-\tau_1} \arrow[out=120,in=60,loop,looseness=3] 
    \arrow[rd] \arrow[rr] &&
    \CC^{-\tau_1-\tau_2} \arrow[out=120,in=60,loop,looseness=3] 
    \arrow[shift left=1, rr] && 
    \cdots \arrow[shift left=1, rr] \arrow[shift left=1, ll] && 
    \CC^{\tau_1-\tau_2+2} \arrow[shift left=1, rr] \arrow[shift left=1, ll] &&
    \CC^{\tau_1-\tau_2} \arrow[shift left=1, rr] \arrow[shift left=1, ll] &&
    \cdots \arrow[shift left=1, rr] \arrow[shift left=1, ll] &&
    \CC^2 \arrow[shift left=1, rr] \arrow[shift left=1, ll] &&
    \CC\rlap{ ,} \arrow[shift left=1, ll] \\
    \CC \arrow[ru] &&
    \CC \arrow[ru]
  \end{tikzcd}
\end{equation*}
where the case $\tau_1=\tau_2$ is understood as above.
\begin{NB2}
  $\bv - \tau_1 = -\mu_i$, $-(\tau_1+\tau_2) = -(\mu_i+\mu_{i+1})$.
\end{NB2}%
\end{NB}%

If $\tau_1 \ge 0 \ge \tau_2$, the data looks like
\begin{equation*}
  \begin{tikzcd}[column sep=1.8ex]
    \CC \arrow[shift left=1, rr]
    && \CC^2 \arrow[shift left=1, rr] \arrow[shift left=1, ll]&&
    \cdots \arrow[shift left=1, rr] \arrow[shift left=1, ll] && 
    \CC^{\tau_1} \arrow[out=120,in=60,loop,looseness=3] 
    \arrow[rr] \arrow[shift left=1, ll] \arrow[rd] &&
    \CC^{\bv} \arrow[out=120,in=60,loop,looseness=3]
    \arrow[rr] \arrow[rd] &&
    \CC^{-\tau_2} \arrow[out=120,in=60,loop,looseness=3]
    \arrow[shift left=1, rr] && 
    \cdots \arrow[shift left=1, rr] \arrow[shift left=1, ll] &&
    \CC^2 \arrow[shift left=1, rr] \arrow[shift left=1, ll] &&
    \CC\rlap{ .} \arrow[shift left=1, ll]
    \\
    &&&&&&& \CC \arrow[ru] && \CC \arrow[ru] &&&&&&&
  \end{tikzcd}
\end{equation*}

\begin{NB}
    If we move $\xl$ to the left, we arrive at
    \begin{equation*}
\vcenter{\vbox{      
  \xymatrix@C=1.2em{
    & \CC^{\bv} \ar[rr] \ar[dr] \ar@(ur,ul)^{\vphantom{j}}
    && \CC^{\bw} \ar@<.5ex>[r] \ar@(ur,ul)
    & \CC^{\bw-1} \ar@<.5ex>[l] \ar@<.5ex>[r]
    & \cdots \ar@<.5ex>[l] \ar@<.5ex>[r]
    & \CC^{2} \ar@<.5ex>[l] \ar@<.5ex>[r]
    & \CC \ar@<.5ex>[l]
  \\
    \CC \ar[ur] && \CC \ar[ur]}}}
\end{equation*}
with $\bw = \tau_1-\tau_2$.
\end{NB}%

The sum of dimension of vector spaces in two way parts is
\(
\frac14 (\tau_1+\tau_2)^2 + \frac12 (\tau_1+\tau_2)
+ \frac14 (\tau_1 - \tau_2)^2.
\)
\begin{NB}
  $\tau_1(\tau_1+1)/2 + (-\tau_2)(-\tau_2-1)/2 = (\tau_1^2 + \tau_2^2)/2
  + (\tau_1+\tau_2)/2 = \frac14 (\tau_1+\tau_2)^2 + \frac12 (\tau_1+\tau_2)
  + \frac14 (\tau_1 - \tau_2)^2$.
\end{NB}%
Hence $\tau_1 - \tau_2$ and $\bv$ are determined from $\kappa$, $\mu$
as in the first case. If we add $2$ to $\bw$ keeping $\bw-2\bv$
unchanged, we add $\CC$ to each entry of the top row, including
changing $0$ to $\CC$ at the leftmost and rightmost entries. This
process also change $\cM^{\mathrm{s}}(\kappa,\mu)$ to a larger
stratum.

For type $A_1$ Coulomb branches, the closure relation on strata is a
total order. We take $\cM_{A_1}(\lambda',\mu')$ as the closure of the
largest stratum. The intersection of $\cM(\lambda,\mu)^{\chi_i}$ with
$\cM^{\mathrm{s}}(\kappa,\mu)$ is either empty or
$\cM_{A_1}^{\mathrm{s}}(\tau_1-\tau_2, \tau_1+\tau_2)$ above, which is
a stratum of $\cM_{A_1}(\lambda',\mu')$.


Let us consider the assertion (2). The data for $\CC_{\xl_j}$ with
$j\neq i, i+1$ is of the form \eqref{eq:4}.
Note that we have other triangle parts $\xl_k$ ($k\neq j$) which are
suppressed as $a_{\xl_k} = 0 = b_{\xl_k}$ and $A_{\xl_k} = \id$.
The data contribute to $w_{k,r}$, $\mathsf y_{k,r}$ by zero including
the case $k=j$ since $B_{\xl_j}$ is
nilpotent and $a_{\xl_j} = 0$ holds.

Next consider the data for $\CC_{\xl_i}\oplus\CC_{\xl_{i+1}}$ as in
\eqref{eq:h:1}. The data contribute to $w_{i,r}$, $\mathsf y_{i,r}$ by
eigenvalues of $B_{\vin{\xl_i}}$ and
$b_{\xl_{i+1}} \prod_{s\neq r} (B_{\vin{\xl_i}}-w_{i,s})
a_{\xl_i}$. They are the same as $w_{i,r}$, $\mathsf y_{i,r}$ for
$\cM_{A_1}(\lambda',\mu')$. As for $w_{k,r}$, $\mathsf y_{k,r}$ with
$k\neq i$, the data contribute by zero since $B_{\vout{\xl_i}}$ is
nilpotent by the defining equation, and $a_{\xl_k} = 0$.
\end{proof}

\begin{Remark}
  The corresponding result for a finite type quiver $Q$ was proved by
  Krylov \cite[Lem.~5.5]{2017arXiv170900391K}. In fact, he considered
  more general one parameter subgroups corresponding to any Levi
  subalgebra. The fixed point set is, in general, not a Coulomb branch
  of a framed quiver gauge theory for the semisimple part of the Levi
  subalgebra. The above argument does not work since the dominance
  order, restricted to weights that differ from $\mu'$ by a linear
  combination of roots, is not a total order.
\end{Remark}

\begin{Remark}\label{rem:char2}
  During the proof of \cref{thm:smaller} we prove that a stratum
  $\cM_{A_1}^{\mathrm{s}}(\mu'+2\bv,\mu')$ of $\cM_{A_1}(\lambda',\mu')$ is the
  intersection $\cM^{\mathrm{s}}(\kappa,\mu)\cap \cM(\lambda,\mu)^{\chi_i}$.

  Let us compute $\kappa$ in terms of $\bv$, $\mu$.
  Note that $\kappa$ is determined by two way parts. Looking at
  \eqref{eq:h:1}, we find that dimensions of vector spaces in two way
  parts are the same as those for the torus fixed point appearing in
  the proof of \cref{prop:torus-fixed-points} after we replaced $\mu$
  by $\tilde\mu  = \mu + \bv\alpha_i$.
  \begin{NB}
    Let $\tau_1$, $\tau_2$ as in the proof.
    $\tilde\mu_i = \tau_1 = \mu_i + \bv$,
    $\tilde\mu_{i+1} = \tau_2 = \mu_{i+1} - \bv$,
    $\tilde\mu_j = \mu_j$ for $j\neq i,i+1$. Namely
    $\tilde\mu = \mu + \bv\alpha_i$.
  \end{NB}%
  The same is true for the case $\tau_1\ge 0\ge\tau_2$, once we
  understand that vector spaces and linear maps go to the right of
  $\xl_{i+1}$ when $\tilde\mu_{i+1} < 0$. Therefore by the proof of
  \cref{prop:torus-fixed-points}, $\kappa$ is the dominant weight
  $\tilde\mu^+$ in the Weyl group orbit of $\tilde\mu$. Namely we have
  \begin{equation*}
    \cM_{A_1}^{\mathrm{s}}(\mu'+2\bv,\mu')
    = \cM^{\mathrm{s}}(\tilde\mu^+,\mu)
    \cap \cM(\lambda,\mu)^{\chi_i}, \qquad
    \tilde\mu = \mu + \bv\alpha_i.
  \end{equation*}

  We further observe that
  \begin{equation*}
    \cM^{\mathrm{s}}(\tilde\mu^+,\mu) \subset \cM(\lambda,\mu)
    \Longleftrightarrow \lambda\ge \tilde\mu^+ \Longleftrightarrow
    \cM(\lambda,\tilde\mu)^T \neq \emptyset,
  \end{equation*}
  where the second $\Leftrightarrow$ is a consequence of
  \cref{cor:torus-fixed-points}.  Since we take $\kappa = \tilde\mu^+$
  with the largest $\bv$, we verify the characterization of
  $\lambda'$ in \cref{rem:char}.
\end{Remark}

\subsection{Another choice of the stability parameter}
\label{subsec:anoth-choice-stab}

Let us consider a parameter $\nu^\Box$ with $\nu_*^{\Box,\CC}$,
$\nu_*^{\Box,\RR}$ are nonzero, but all other $\nu^\Box_h$ are
zero.
We have $\underline{\cM}(\lambda,\mu)$,
$\cM^{\nu^{\Box,\CC}}(\lambda,\mu)$,
$\underline{\widetilde\cM}(\lambda,\mu)$,
$\cM^{\nu^{\Box,\RR}}(\lambda,\mu)\to\cM(\lambda,\mu)$ as in
\cref{subsec:deformed-case}.
Recall that we choose $\nu^{\bullet}_* = 0$ in \cref{subsec:deformed-case}.
Therefore this parameter $\nu^{\Box}$ is complementary to the choice
$\nu^\bullet$ there.

Let us consider the case $n=2$, $\ell = 1$, $\lambda = \Lambda_0$ for
notational simplicity. The data and equations are
\begin{equation*}
  \begin{tikzcd}[column sep=2.4ex]
    V_2 \arrow[rr, "A_1"] \arrow[rd, "b_1"']
    \arrow[shift right=1,bend left=45,rrrr,"D"' pos=.8]
    &&
    V_1 \arrow[rr,"A_0"] \arrow[rd, "b_0"']
    \arrow[out=120,in=60,loop,looseness=3, "B_1"]
    && V_0
    \arrow[shift right=1, bend right=45, llll, "C"' pos=.8] \\
    & \CC
    \arrow[ru, "a_1"']
    && \CC
    \arrow[ru, "a_0"'] &
  \end{tikzcd}
  \qquad
  \left\{
    \begin{aligned}[m]
      & (-DC + \nu_*^{\Box,\CC}) A_0 - A_0 B_1 + a_0 b_0 = 0,\\
      & B_1 A_1 + A_1 CD + a_1 b_1 = 0.
    \end{aligned}
    \right.
\end{equation*}
We assume the balanced condition, i.e., $\dim V_0 = \dim V_2$.

Let us analyze the fixed point set
$\cM^{\nu^{\Box,\CC}}(\lambda,\mu)^T$ as in
\cref{prop:torus-fixed-points,prop:deformed}.
The data for a fixed point decompose as sum of \eqref{eq:4} for
$\CC_{\xl_i}$ ($i=0,1$), and each summand looks like
\eqref{eq:19}. Recall that we shrink triangle parts, which do not
communicate with $\CC_{\xl_i}$ to get \eqref{eq:19}. In particular, at
a vector space $V_i$
\begin{NB}
on which the triangle $\xl_0$ is shrunk  
\end{NB}%
, the defining equation is $C_i D_i - D_{i-1} C_{i-1} + \nu_*^{\Box,\CC} = 0$.
\begin{NB}
  \begin{equation*}
    \begin{tikzcd}
      \arrow[shift left=1,r,"C_1"] & \arrow[shift left=1,l,"D_1"] V_1
      \arrow[out=120,in=60,loop,looseness=3, "B_1"]
      \arrow[r,"A"] & V_0 \arrow[out=120,in=60,loop,looseness=3,
      "B_0"] \arrow[shift left=1,r,"C_0"] & \arrow[shift
      left=1,l,"D_0"]
    \end{tikzcd}
    \quad
    C_1 D_1 + B_1 = 0, \quad D_0 C_0 + B_0 = 0,
    \quad (B_0+\nu_*^{\Box,\CC}) A - A B_1 = 0.
  \end{equation*}
  So we have $C_1 D_1 - D_0 C_0 = - B_1 + B_0 = -\nu_*^{\Box,\CC}$ after setting
  $A = \operatorname{id}$.
\end{NB}%
Therefore we see that eigenvalues of $-DC$, $B_1$,
$-CD$ are in $\ZZ\nu^{\Box,\CC}_*$. We decompose $V_0$, $V_1$, $V_2$
into generalized eigenspaces, and apply the factorization. Because of
the shift $\nu_*^{\Box,\CC}$, the resulted data look as
\begin{equation*}
  \begin{tikzcd}[column sep=1.8ex]
    \cdots\ \arrow[rr,"A_0"] &[-10pt]&[-10pt] V_0(m)
    \arrow[shift left=1, rr, "C"] &&
    V_2(m) \arrow[rr, "A_1"] \arrow[rd, "b_1"']
    \arrow[shift left=1, ll, "D"]
    &[-10pt] &[-10pt]
    V_1(m) \arrow[rr,"A_0"] \arrow[rd, "b_0"']
    \arrow[out=120,in=60,loop,looseness=3, "B_1"]
    &[-10pt]&[-10pt] V_0(m-1)
    \arrow[shift left=1,rr, "C"]
    &&
    V_2(m-1) \arrow[shift left=1,ll,"D"] \arrow[rd, "b_1"']
    \arrow[rr, "A_1"]
    &[-10pt]&[-10pt]     \ \cdots
    \\
    & \CC\arrow[ru, "a_0"']
    &&&& \CC \arrow[ru, "a_1"'] && \CC \arrow[ru, "a_0"'] & &&&
    \CC\rlap{ ,} &
  \end{tikzcd}
\end{equation*}
where $V_i(m)$ corresponds to the eigenvalue $m\nu_*^{\Box,\CC}$ for
the fixed point. (We are thinking of data in a neighborhood of the fixed
point. So the eigenvalue is not precisely $m\nu_*^{\Box,\CC}$. But it
is `close'.)
Since we may assume $CD$, $DC$ on $V_0(m)$, $V_2(m)$ have eigenvalues
different from $0$ if $m\neq 0$, we see that $C$, $D$ are
isomorphisms. Therefore we can identify $V_0(m)$ and $V_2(m)$ by $C$,
and then determine $D$ by the equation. Therefore we can collapse all
two way parts except one between $V_0(0)$ and $V_2(0)$.
Thus bow data is for type $A_\infty$, type $A$ diagram going to
infinity in both directions, with single $\boldsymbol\medcirc$
corresponding to the remaining two way part.
We can absorb $\nu^{\Box,\CC}_*$ to $B_i$'s once we do not have two
way parts except one. (See \cite[\S 3.1.4]{2016arXiv160602002N}.)
\begin{NB}
  We identify $V_0(-1)$ and $V_2(-1)$:
  \begin{equation*}
    \begin{gathered}
  \begin{tikzcd}[column sep=1.8ex]
    \cdots\ \arrow[rr,"A_0"] &[-10pt]&[-10pt] V_0(0)
    \arrow[shift left=1, rr, "C"] &&
    V_2(0) \arrow[rr, "A_1"] \arrow[rd, "b_1"']
    \arrow[shift left=1, ll, "D"]
    &[-10pt] &[-10pt]
    V_1(0) \arrow[rr,"A_0"] \arrow[rd, "b_0"']
    \arrow[out=120,in=60,loop,looseness=3, "B_1"]
    &[-10pt]&[-10pt] V_0(-1)
    \arrow[shift left=1,rr, "C"]
    &&
    V_2(-1) \arrow[shift left=1,ll,"D"] \arrow[rd, "b_1"']
    \arrow[rr, "A_1"]
    &[-10pt]&[-10pt]     \ \cdots
    \\
    & \CC\arrow[ru, "a_0"']
    &&&& \CC \arrow[ru, "a_1"'] && \CC \arrow[ru, "a_0"'] & &&&
    \CC\rlap{ ,} &
  \end{tikzcd}
\\      
  \begin{tikzcd}[column sep=1.8ex]
    \cdots\ \arrow[rr,"A_0"] &[-10pt]&[-10pt] V_0(0)
    \arrow[shift left=1, rr, "C"] &&
    V_2(0) \arrow[rr, "A_1"] \arrow[rd, "b_1"']
    \arrow[shift left=1, ll, "D"]
    &[-10pt] &[-10pt]
    V_1(0) \arrow[rr,"A_0"] \arrow[rd, "b_0"']
    \arrow[out=120,in=60,loop,looseness=3, "B_1"]
    &[-10pt]&[-10pt] V_0(-1)
    \arrow[out=120,in=60,loop,looseness=3, "B_0"]
    \arrow[rd, "b_1"']
    \arrow[rr, "A_1"]
    &[-10pt]&[-10pt]  \ \cdots
    \\
    & \CC\arrow[ru, "a_0"']
    &&&& \CC \arrow[ru, "a_1"'] && \CC \arrow[ru, "a_0"'] & &
    \CC\rlap{ ,} &
  \end{tikzcd}
    \end{gathered}
  \end{equation*}
  by setting $C=\id$, $D - \nu_*^{\Box,\CC} = -B_0$. Then we have
  $B_0 A_0 - A_0 B_1 + ab = 0$. In the next triangle right, the
  original equation was $B_1 A_1 + A_1 CD + ab = 0$. After the above
  substitution, this is changed to
  $B_1 A_1 - A_1 (B_0-\nu^{\Box,\CC}) + ab = 0$. We then replace $B_1$
  to $B_1-\nu^{\Box,\CC}$ to remove $\nu^{\Box,\CC}$. We continue.
\end{NB}%
Note also that the balanced condition $\dim V_0(0) = \dim V_2(0)$
remains to be true as we observed $\dim V_0(m) = \dim V_2(m)$ for
$m\neq 0$.
\begin{NB}
    The following sentence and subsequent paragraph are added on
    Feb.~13, 2020.
\end{NB}%
Therefore the $A_\infty$ type bow variety is
$\cM_{A_\infty}(\lambda',\mu')$ with $\lambda'=\Lambda_0$,
$\mu' = \Lambda_0 - \sum_{m,i}\bv_i(m)\alpha_{mn+i}$.

Moreover the factorization isomorphism
$\cM^{\nu^{\Box,\CC}}(\lambda,\mu) \approx
\cM_{A_\infty}(\lambda',\mu')$
is $T$-equivariant, if we let $T$ act on the right hand side as
follows: for $\CC$ in a triangle between $V_2(m)$ and $V_1(m)$, it
acts by $s_0(s_0 s_1)^{-m}$. For $\CC$ in a triangle between
$V_1(m+1)$ and $V_0(m)$, it acts by $(s_0 s_1)^{-m}$. The action of
$s_0 s_1$ on the triangle between $V_1$ and $V_0$ for the left hand
side is absorbed into the shift $(s_0 s_1)^{-m}$ in the right hand
side. By the same argument as in the proof of
\cref{lem:weyl-chambers}, the $T$-fixed point set in
$\cM_{A_\infty}(\lambda',\mu')$ is \emph{not} larger than the fixed
point set with respect to the torus $T_\infty$ acting by linearly
independent weights on $\CC$ for triangles.

\begin{Proposition}\label{prop:unwind}
  Let $\lambda = \Lambda_0$ as above and write
  $\lambda - \mu = \sum_{i=0}^{n-1} \bv_i \alpha_i$.
  The $T$-fixed points $\cM^{\nu^{\Box,\CC}}(\lambda,\mu)^T$ are
  finite, and correspond to a decomposition
  $\bv_i = \sum_{m\in\ZZ} \bv_i(m)$ such that the corresponding
  $A_\infty$ bow variety $\cM_{A_\infty}(\lambda',\mu')$ as above has
  a $T$-fixed point, where $\lambda' = \Lambda_0$,
  $\mu' = \Lambda_0 - \sum_{m,i} \bv_i(m) \alpha_{mn+i}$.
  The image of
  $\cM^{\nu^{\Box,\CC}}(\lambda,\mu)^T$ under the factorization
  morphism $\varpi$ is supported in $\ZZ\nu^{\Box,\CC}_*$, and the
  multiplicities of $m\nu^{\Box,\CC}$ are given by $\bv_0(m)$, $\bv_1(m)$.
  Around the fixed point, the bow variety
  $\cM^{\nu^{\Box,\CC}}(\lambda,\mu)$ is isomorphic to a
  neighborhood of the unique $T$-fixed point in
  $\cM_{A_\infty}(\lambda',\mu')$.
\end{Proposition}

It is also easy to describe $\cM_{A_\infty}(\lambda',\mu')^T$. See
\cref{sec:Maya}.

Next consider the morphism
$\pi\colon\cM^{\nu^{\Box,\RR}}(\lambda,\mu)\to\cM(\lambda,\mu)$. Since
$\ell=1$, the stratification (see \cref{thm:stratification}) is
$\cM(\lambda,\mu) = \bigsqcup
\cM^{\mathrm{s}}(\lambda-|\underline{k}|\delta,\mu) \times
S^{\underline{k}}(\CC^2)$.
The fiber of $\pi$ over the stratum
$\cM^{\mathrm{s}}(\lambda-|\underline{k}|\delta,\mu) \times
S^{\underline{k}}(\CC^2)$ is isomorphic to a product of the punctual
Hilbert scheme of $k_1$ points, $k_2$ points, \dots on $\CC^2$ by
\cite[\S4.3]{2016arXiv160602002N}. In particular, it is
irreducible. Hence
\begin{equation*}\label{eq:20}
  \pi_*(\CC_{\cM^{\nu^{\Box,\RR}}(\lambda,\mu)}[\dim])
  \cong \bigoplus \mathrm{IC}(\cM(\lambda-|\underline{k}|\delta,\mu))
  \boxtimes \CC_{\overline{S^{\underline{k}}(\CC^2)}}[\dim],
\end{equation*}
where $\CC_{X}[\dim]$ denote the shift of the constant sheaf on $X$ by
$\dim X$ and $\overline{S^{\underline{k}}(\CC^2)}$ is the closure of
$S^{\underline{k}}(\CC^2)$ in $S^{|\underline{k}|}(\CC^2)$. Therefore
\begin{equation}\label{eq:21}
  \bigoplus_\mu \Phi\circ \pi_*(\CC_{\cM^{\Box,\RR}(\lambda,\mu)}[\dim])
  \cong \left(\bigoplus_\mu \Phi(\mathrm{IC}(\cM(\lambda,\mu)))\right)
  \otimes\left(\bigoplus_{\underline{k}} \CC[\overline{S^{\underline{k}}(\CC)}]
  \right),
\end{equation}
where $S^{\underline{k}}(\CC)$ is the stratum of the symmetric product
of the line $\CC$ corresponding to a partition $\underline{k}$, and
its closure is the attracting set in
$\overline{S^{\underline{k}}(\CC^2)}$.

\begin{Remark}\label{rem:Hilb}
  Suppose $n=1$, $\ell = 1$. Let us denote the corresponding balanced
  bow variety by $\cM(\bv,1)$ by using dimension vectors as in
  \cref{subsec:coulomb-branch}. ($\bw = 1$ as $\ell = 1$.) We have
  also $\cM^{\nu^{\Box,\RR}}(\bv,1)$, etc. In this case $\cM(\bv,1)$,
  $\cM^{\nu^{\Box,\RR}}(\bv,1)$ are isomorphic to quiver varieties of
  Jordan type as the cobalanced condition is satisfied. In particular,
  they are isomorphic to the symmetric product $S^{\bv}(\CC^2)$ and
  Hilbert scheme $\mathrm{Hilb}^{\bv}(\CC^2)$ of $\bv$ points on
  $\CC^2$. The proof of \cref{prop:unwind} works in this case, and we
  recover a well-know fact that a fixed point in
  $\mathrm{Hilb}^{\bv}(\CC^2)$ corresponds to a partition of
  $\bv$. (See e.g., \cite[Ch.~5]{Lecture}.)

  Since
  $\cM^{\mathrm{s}}(\lambda-|\underline{k}|\delta,\mu) = \emptyset$
  unless $\mu = \lambda-|\underline{k}|\delta$ in this case, we have
  \( \bigoplus_\mu \Phi\circ
  \pi_*(\CC_{\cM^{\nu^{\Box,\RR}}(\lambda,\mu)}[\dim]) =
  \bigoplus_{\underline{k}} \CC[\overline{S^{\underline{k}}(\CC)}] \)
  instead of \eqref{eq:21}. This result is also known, see e.g.,
  \cite[Ch.~7]{Lecture}.
\end{Remark}

\subsection{Hyperbolic restriction in two steps}\label{subsec:twosteps}

Let us take one parameter subgroup as in \cref{thm:smaller}. Since it
depends on $i$, let us denote it by $\chi_i$. We also take a one
parameter subgroup $\chi(t) = (t^{m_0},\dots, t^{m_{n-1}})$ with
$m_j < 0$ for all $j$. We have a chamber structure on the space of one
parameter subgroups, and $\chi_i$ lives in the boundary of the chamber
containing $\chi$.
The result in this subsection remains true only under this assumption,
but we keep notation for brevity.
Let us denote the attracting sets with respect to
$\chi$ and $\chi_i$ by $\fA$ and $\fA_{i}$ respectively. We have
\(
    \cM(\lambda,\mu)^{\chi_i} \xleftarrow{p_i}\fA_{i}
    \xrightarrow{j_i} \cM(\lambda,\mu).
\)
Note also that $\chi$ acts nontrivially on the fixed point set
$\cM(\lambda,\mu)^{\chi_i}$, and we have the corresponding
attracting set, which will be studied in \cref{subsec:attrA1}. Let us
denote it by $\fA^i$. We have
\(
    \mathrm{pt} = \cM(\lambda,\mu)^{\chi} \xleftarrow{p^i}\fA^i
    \xrightarrow{j^i} \cM(\lambda,\mu)^{\chi_i}
\)
as above.
Then we form the diagram
\begin{equation*}
  \begin{CD}
    \fA @>j''>> \fA_{i} @>j_i>> \cM(\lambda,\mu) \\
    @V{p''}VV @VV{p_i}V @. \\
    \fA^i @>j^i>> \cM(\lambda,\mu)^{\chi_i} @. \\
    @V{p^i}VV @. @.
    \\
    \cM(\lambda,\mu)^{\chi} @. @.
  \end{CD}
\end{equation*}
where $\fA$ is the fiber product of $\fA_{i}$ and $\fA^i$ over
$\cM(\lambda,\mu)^{\chi_i}$. Then as in
\cite[\S4.5]{2014arXiv1406.2381B} $\fA$ is the attracting set in
$\cM(\lambda,\mu)$ with respect to $\chi$ such that
compositions $j_i\circ j''$, $p^i\circ p''$ give the diagram
\(
    \mathrm{pt} = \cM(\lambda,\mu)^{\chi} \xleftarrow{p^i\circ p''}\fA
    \xrightarrow{j_i\circ j''} \cM(\lambda,\mu)
\)
is the one given by $\chi$.

\begin{Proposition}\label{prop:factor}
  There is a natural transformation from the hyperbolic restriction
  functor $\Phi = (p^i\circ p'')_* (j_i\circ j'')^!$ to
  $\Phi^i\circ \Phi_i$, the composition of two hyperbolic restriction
  functors by base change. Here $\Phi^i = p^i_* (j^i)^!$,
  $\Phi_i = (p_i)_* j_i^!$.
\end{Proposition}

\section{Construction}\label{sec:construction}

After preparation in the previous sections, we are ready to define the
action of generators $e_i$, $f_i$, $h_i$ on $\mathcal V(\lambda)$. The
operator $h_i$ is defined so that $\mathcal V_\mu(\lambda)$ is the
weight space with weight $\mu$.

\subsection{Type \texorpdfstring{$A_1$}{A1}}\label{subsec:attrA1}

Let us consider the bow variety of type $A_1$ with the balanced
condition. We suppose dimension vectors are $\bv$,
$\bw\in\ZZ_{\ge 0}$. By \cref{cor:torus-fixed-points} we assume
$\bv\le\bw$, as there is no fixed point otherwise. We take a one
parameter subgroup $\chi(t) = t^m$ ($m < 0$) as in the previous
subsection.

The following was observed in \cite[several paragraphs after
Th.~3.1]{2017arXiv170900391K}, but let us give a proof in terms of bow
varieties for completeness.

\begin{Theorem}
  The attracting set $\fA$ for the type $A_1$ balanced bow variety
  associated with $\bv$, $\bw$ is isomorphic to $\CC^\bv$.
\end{Theorem}

\begin{proof}
  We start with the balanced diagram, two triangles at the left and
  right ends.  We have $\bw$ two way parts in between. The leftmost
  and rightmost vector spaces are $0$ and others are
  $\bv$-dimensional. 
  By Hanany-Witten transition, we move the leftmost triangle next to
  the rightmost one. The result is
\begin{equation*}
  \begin{tikzcd}[column sep=1.8ex]
    \CC \arrow[shift left=1, rr, "C_{\bw}"] &&
    \CC^2 \arrow[shift left=1, ll, "D_{\bw}"] 
    \arrow[shift left=1, rr, "C_{\bw-1}"] &&
    \cdots \arrow[shift left=1, ll, "D_{\bw-1}"]
    \arrow[shift left=1, rr, "C_{3}"] &&
    \CC^{\bw-1} \arrow[shift left=1, ll,"D_{3}"]
    \arrow[shift left=1, rr, "C_{2}"] &&
    \CC^{\bw}
    \arrow[out=120,in=60,loop,looseness=3,"B_+"]
    \arrow[shift left=1, ll, "D_{2}"]
    \arrow[rr,"A"] \arrow[rd,"b_+"'] &&
    \CC^{\bv}
    \arrow[out=120,in=60,loop,looseness=3,"B_-"] \arrow[rd,"b_-"'] &
    \\
    &&&&&&&&&\CC_+ \arrow[ru,"a_+"'] && \CC_-.
  \end{tikzcd}
\end{equation*}

Let us study the attracting set. The action is induced from one given
by $b_{-}\mapsto t^{m} b_{-}$ and other maps unchanged.
\begin{NB}
    If $m > 0$ instead of $m < 0$, we repeat the following
    argument after applying Hanany-Witten transition to move to
\begin{equation*}
  \xymatrix@C=1.2em{ & V_0 \ar@(ur,ul)_{B_0} \ar[rr]_{A} \ar[dr]_{b_+} 
  && V_+ \ar@(ur,ul)_{B_+}
  \ar@<.5ex>[rr]^{C_1} && V_1 \ar@<.5ex>[ll]^{D_1} \ar@<.5ex>[rr]^{C_2} &&
  \ar@<.5ex>[ll]^{D_2} \cdots \ar@<.5ex>[rr]^{C_n} && V_{N-1} \ar@<.5ex>[ll]^{D_n}
  \\
  \CC\ar[ur]_{a_-} && \CC \ar[ur]_{a_+}},
\end{equation*}
\end{NB}%

We have $b_-B_-^k a_+ = 0$ for any $k$ as it has weight
$m < 0$.  But $b_- B_-^k$ ($k=0,1,\dots,\bv-1$) is a base of the
dual of $\CC^\bv$ by the condition (S1). Hence $a_+ = 0$. The defining
equation now becomes $B_- A = A B_+$. Therefore $A$ is surjective by
(S2). (This gives a direct proof that there is no fixed point unless
$\bv\le\bw$.)

By the equation $C_{i+1} D_{i+1} - C_i D_i  = 0$, we see that
$B_+ = - C_{1} D_{1}$ is nilpotent by induction. Hence $B_-$ is also
nilpotent. By (S1) $b_- B_-^k$
($k=0,1,\dots,\bv-1$) is a base of the dual of $\CC^\bv$, hence we write $B_-$, $b_-$ as
\begin{equation*}
  B_- = {}^t\! J_\bv
  \quad
  b_- = {}^t e_\bv.
\end{equation*}
This normalization kills the action of $\GL(\bv)$.
Here ${}^t e_\bv =
\begin{bmatrix}
  0 & \cdots & 0 & 1
\end{bmatrix}.
$

Note that $\Ker A$ is $B_+$-invariant, and
$b_+|_{\Ker A}\in (\Ker A)^*$ is cocyclic with respect to
$B_+|_{\Ker A}$ by (S1). Therefore
\begin{gather*}
    b_- B_-^{\bv-1}A = b_- A B_+^{\bv-1}, \quad
    b_- B_-^{\bv-2}A = b_- A B_+^{\bv-2}, \quad \dots, \quad b_- A,\\
    b_+ B_+^{\bw-\bv-1}, \quad b_+ B_+^{\bw-\bv-2}, \quad\dots,\quad b_+ B_+, \quad
    b_+
\end{gather*}
is a base of the dual of $\CC^\bw$.
If $\bw > \bv$, we have
\(
   A = \left[
   \begin{smallmatrix}
     \id_\bv & 0
   \end{smallmatrix}
 \right], \) \( b_+ = {}^t e_\bw \) and
\begin{equation*}
  B_+ = \prescript{t}{}{\left[
    \begin{array}{c|c|c|c}
      J_\bv & \vec{c} & \multicolumn{2}{c}{0} \\
      \hline
      0 & \multicolumn{3}{c}{J_{\bw-\bv}}
    \end{array}
  \right]}, \quad
  \vec{c} =
  \begin{bmatrix}
    c_1 \\ \vdots \\ c_\bv
  \end{bmatrix}
  \text{with }
  b_+ B_+^{\bw-\bv} = c_1 b_- A B_+^{\bv-1} + \dots + c_\bv b_- A.
\end{equation*}
\begin{NB}
  Note that $b_+ B_+^{\bw-\bv}$ is zero on $\Ker A$, as
  $\dim \Ker A = \bw - \bv$.
\end{NB}%
If $\bw = \bv$, we have $A = \id$, $B_+ = \prescript{t}{}{J_\bv}$, and
$b_+$ is arbitrary.
Once the action of $\GL(\bw)$ is killed, the remaining data $C_i$,
$D_i$ are regarded as a point of a quiver variety of type $A_{\bw-1}$,
which is the nilpotent cone of $\algsl(\bw)$. See
\cite[\S7]{Na-quiver}. Therefore they are normalized by the remaining
action of $\GL(\bw-1)\times\dots\times\GL(1)$. Hence the attracting
set is $\CC^\bv$ parametrizing $c_1$, \dots, $c_\bv$ ($\bw > \bv$), or
$b_+$ ($\bw = \bv$).
\end{proof}

\begin{Remark}
  Suppose $\bw > \bv$. Then $\cM$ is isomorphic to the intersection of
  the nilpotent cone of $\algsl(\bw)$ and the space of $\bw\times\bw$
  matrices of the following form:
  \begin{equation*}
    \begin{array}{rrcccc|ccccl}
      && \multicolumn{4}{l}{\overbrace{\hspace{6em}}^{\mbox{$\bv$ columns}}}
      &  \multicolumn{4}{c}{\overbrace{\hspace{6em}}^{\mbox{$\bw-\bv$ columns}}}
      &
      \\
      \ldelim\{{4}{*}[$\bv$ rows] & \ldelim[{9}{*} &
        0 & \dots & 0 & * & 0 & \dots & 0 & *
      & \rdelim]{9}{*} \\
            && 1 &  & 0 & * & \multirow{2}{*}{\vdots} &&
               \multirow{2}{*}{\vdots} & \multirow{2}{*}{\vdots}& \\
            && 0 & \ddots & 0 & * &&&&& \\
            && 0 & \dots & 1 & * & 0 & \dots & 0 & *& \\
      \cline{3-10}
      \ldelim\{{4}{*}[$\bw-\bv$ rows]
      && * & \multicolumn{2}{c}{\dots}     & * & 0 & \dots & 0 & * &\\
            && 0 & \multicolumn{2}{c}{\dots} & 0 & 1 && 0 & * & \\
            && 0 & \multicolumn{2}{c}{\dots} & 0 & 0 & \ddots & 0 & * &\\
            && 0 & \multicolumn{2}{c}{\dots} & 0 & 0 & \dots & 1 & * &
    \end{array},
  \end{equation*}
  where $*$ indicates an arbitrary complex number. We use
  \cite[Prop.~3.2]{2016arXiv160602002N} twice, first to determine the
  upper left block, and second to give the remaining part. The space
  of matrices of the above form is a slice to the nilpotent matrix
  $(\bv,\bw-\bv)$ considered in \cite{MR1968260} when
  $\bv \le \bw-\bv$. When $\bv > \bw - \bv$, the space has larger
  dimension than the slice.

  If $\bw = \bv$, the triangle parts (quotient-ed by $\GL(\bv)$) give
  the product of $\GL(\bw)\times \CC^{2\bw}$ and the space of the
  matrix of the above form (with $\bw - \bv = 0$). Then $\cM$ is the
  hamiltonian reduction of the product of the nilpotent cone for
  $\algsl(\bw)$ and this space by the diagonal action of $\GL(\bw)$.
\end{Remark}

Since the hyperbolic restriction functor $\Phi$ is hyperbolic
semismall, $\Phi(\mathrm{IC}(\cM_{A_1}(\lambda, \mu)))$
($\lambda = \bw$, $\mu = \bw - 2\bv$) has a base parametrized by
irreducible components of the attracting set
$\fA=\fA_{A_1}(\lambda,\mu)$. Therefore

\begin{Corollary}
  The attracting set $\fA_{A_1}(\lambda,\mu)$ is irreducible. Hence
  $\Phi(\mathrm{IC}(\cM_{A_1}(\lambda, \mu))) \cong
  \CC[\fA_{A_1}(\lambda,\mu)]$.
\end{Corollary}

\begin{Theorem}\label{thm:A1}
  \textup{(1)} The direct sum
  $\bigoplus_\mu \mathcal V_\mu(\lambda) = \bigoplus_\mu
  \Phi(\mathrm{IC}(\cM_{A_1}(\lambda, \mu)))$ has an
  $\algsl(2)$-module structure, which is irreducible with dimension
  $\lambda+1 = \bw + 1$. Moreover homomorphisms in \eqref{eq:9}
  intertwine $\algsl(2)$-module structures when we endow an
  $\algsl(2)$-module structure on the right hand side as the tensor
  product via \cref{cor:tensor}.

  \textup{(2)}
  $[\fA_{A_1}(\lambda,\mu)] = \frac{f^n}{n!}
  [\fA_{A_1}(\lambda,\lambda)]$ for $\lambda - \mu = 2n$.
\end{Theorem}

The construction is explicit and will be given during the proof.

\begin{NB}
  We have $f^2 = 0$ in each tensor factor. Then
  \begin{multline*}
    (f\otimes 1\otimes\cdots\otimes 1
    + 1\otimes f\otimes 1\otimes\cdots\otimes 1 + \cdots
    + 1 \otimes \cdots \otimes 1\otimes f)^n\\
    = {n!}\underbrace{f\otimes f\otimes\cdots\otimes f}_{\text{$k$ times}}
      \otimes \underbrace{1\otimes\cdots\otimes 1}_{\text{$n-k$ times}}
      + \dots
  \end{multline*}
\end{NB}%

\begin{proof}
The operator $h$ is given by $\mu\operatorname{id}$ on the summand
$\mathcal V_\mu(\lambda)$.

If $\lambda = 0$, it is the trivial representation. We have nothing to
do. Next consider the case $\lambda = 1$. We need to study
$\cM_{A_1}(1,\pm 1)$. In the $+$ case, we get a special case of the
bow diagram studied in the proof of \cref{prop:torus-fixed-points}. In
particular, it is a point. 
We have $\Phi(\mathrm{IC}(\cM_{A_1}(1,1))) \cong \CC[\mathrm{point}]$.
In the $-$ case we normalize $A$, $b_-$ to $1$ to kill the action, and
determine $B_-$ from the equation $B_- + a_+ b_+ = 0$. The remaining
data are $a_+$, $b_+$, hence we have $\cM_{A_1}(1, -1) \cong\CC^2$.
The attracting set $\fA$ is given by $a_+ = 0$ as above, therefore
$\fA\cong\CC$. We have
$\Phi(\mathrm{IC}(\cM_{A_1}(1, -1)))\cong \CC[\fA]$. We then define
$e$, $f$ on
$\Phi(\mathrm{IC}(\cM_{A_1}(1, 1)))\oplus\Phi(\mathrm{IC}(\cM_{A_1}(1,
-1)))$ by
\begin{equation*}
    e[\mathrm{point}] = 0, \quad f[\mathrm{point}] = [\fA], \quad
    f[\fA] = 0, \quad e[\fA] = [\mathrm{point}].
\end{equation*}
This gives the two dimensional standard representation of $\algsl(2)$.
The formula in (2) holds by definition.

Let us consider $\lambda > 1$. We take a real parameter $\nu^\RR$
so that $\nu_1^\RR < \nu_2^\RR < \cdots < \nu_{\lambda-1}^\RR$.
There is no $\nu^\RR_*$ as we are considering type $A_1$ bow varieties.
The condition ($\boldsymbol\nu\bf 2$) is automatically satisfied, and
($\boldsymbol\nu\bf 1$) says that a graded subspace $S$ as in
($\boldsymbol\nu\bf 1$) must be zero. In particular,
$\nu^\RR$-semistability and $\nu^\RR$-stability are equivalent, and
$\cM^\nu$ is smooth.

Recall that we constructed an isomorphism between the hyperbolic
restriction $\Phi$ of
$\pi_*(\mathrm{IC}(\cM^{\nu^\RR}_{A_1}(\lambda,\mu)))$ and
\begin{equation*}
  \bigoplus_{\substack{\mu_i = \pm 1\\ \mu_1+\dots+\mu_\bw=\mu}}
  \bigotimes_{i=1}^\lambda \mathcal V_{\mu_i}(1)
\end{equation*}
in \cref{cor:tensor}. Therefore for the direct sum
\begin{equation}\label{eq:10}
   \bigoplus_\mu \Phi(\pi_*(\mathrm{IC}(\cM^{\nu^\RR}_{A_1}(\lambda,\mu))))
   \cong
    \bigotimes_{i=1}^\lambda \left( 
      \mathcal V_{1}(1) \oplus \mathcal V_{-1}(1)\right)
    = \underbrace{\CC^2\otimes\cdots\otimes \CC^2}_{\text{$\lambda$ times}}.
\end{equation}
We endow an $\algsl(2)$-module structure as the tensor product of the
above construction for $\lambda = 1$.

Now we consider the hamiltonian reduction in the definition of the bow
variety in two steps as in \cref{subsec:Weyl}. In this case, the first
reduction gives the product of a quiver variety of type
$A_{\lambda-1}$ with dimension vectors $(1,2,\dots,\lambda-1)$,
$(0,\dots,0,\lambda)$ and another variety given by triangles. The
first quiver variety is the cotangent bundle of the flag variety for
$\SL(\lambda)$.
\begin{NB}
Let us denote it by $\mathfrak M^{\nu^\RR}$. On the
other hands, if we do not put the real parameter $\nu^\RR$, we get the
nilpotent cone. Let us denoted it by $\mathfrak M^0$.  The projection
$\pi\colon \mathfrak M^{\nu^\RR}\to\mathfrak M^0$ is the Springer
resolution, and $\End(\pi_*(\mathrm{IC}(\mathfrak M^{\nu^\RR})))$ is
isomorphic to the group ring of the symmetric group
$\mathfrak S_{\lambda}$ of $\lambda$ letters as the Springer
representation.
Recall that the stratification of $\cM$ is induced from that of
$\mathfrak M^0$ (see the proof of
\cite[Th.~7.13]{2016arXiv160602002N}).
Also $\pi$ for $\mathfrak M$ and for $\cM$ are compatible with the
hamiltonian reduction, hence
$\pi_*(\bigoplus_\mu \mathrm{IC}(\cM^{\nu^\RR}_{A_1}(\lambda,\mu)))$
is given by $\pi_*(\mathrm{IC}(\mathfrak M^{\nu^\RR}))$ by the
hamiltonian reduction. In particular, we have the
$\mathfrak S_{\lambda}$ action on
$\pi_*(\bigoplus_\mu \mathrm{IC}(\cM^{\nu^\RR}_{A_1}(\lambda,\mu)))$
and its hyperbolic restriction.
Since the Springer representation comes from the permutation of
eigenvalues of $B_+$, it is compatible with our analysis in
\cref{subsec:deformed-case}. Therefore $\mathfrak S_\lambda$ acts on
\eqref{eq:10} by permutation of factors.
\end{NB}%
Thus we have $\mathfrak S_\lambda$ action on
$\bigoplus_\mu\pi_*(\mathrm{IC}(\cM^{\nu^\RR}_{A_1}(\lambda,\mu)))$
and its hyperbolic restriction \eqref{eq:10}. The action on the latter
is given by permutation of factors.

Now
$\bigoplus_\mu\mathcal V_\mu(\lambda) = \bigoplus_\mu
\Phi(\mathrm{IC}(\cM_{A_1}(\lambda, \mu)))$ is the direct summand of
\eqref{eq:10} consisting of $\mathfrak S_\lambda$-fixed
vectors. Therefore it is isomorphic to the symmetric power
$S^\lambda(\CC^2)$.
In particular, it inherits an $\algsl(2)$-module structure, which is
irreducible with dimension $\lambda + 1$, as we promised.
The assertion on tensor products is clear from the construction.

To check the formula (2) we need to compute $[\fA_{A_1}(\lambda,\mu)]$
in the tensor product \eqref{eq:10}. The isomorphism in \eqref{eq:10}
came from the factorization, and we use the base change
from $\BA^n/\mathfrak S_n$ to $\BA^n$. Therefore we get a factor $n!$.
\end{proof}

\begin{NB}

**************************************************************

\begin{equation*}
\begin{tikzpicture}[baseline=(current  bounding  box.center)]
  \node[label=above:$v$] at (0,0) {$\vphantom{j^X}$};
  \node[label=below:$\xl_{i+1}$] at (0.5,0)
        {$\vphantom{j^X}\boldsymbol\times$};
  \node[label=below:$\xl_{i}$] at (-0.5,0)
  {$\vphantom{j^X}\boldsymbol\times$};
  \node[label=above:$n$] at (1,0)
  {$\vphantom{j^X}$};
  \node[label=above:$m$] at (-1,0)
  {$\vphantom{j^X}$};
  \node[label=above:$-1$] at (2,0)
        {$\vphantom{j^X}\boldsymbol\medcirc$};
  \node[label=above:$1$] at (-2,0)
        {$\vphantom{j^X}\boldsymbol\medcirc$};
  \node[label=above:$-1$] at (3,0)
        {$\vphantom{j^X}\boldsymbol\medcirc$};
  \node[label=above:$1$] at (-3,0)
        {$\vphantom{j^X}\boldsymbol\medcirc$};
  \node[label=above:$\cdots$] at (4,0)
        {$\vphantom{j^X}$};
  \node[label=above:$\cdots$] at (-4,0)
        {$\vphantom{j^X}$};
  \node[label=above:$-1$] at (5,0)
        {$\vphantom{j^X}\boldsymbol\medcirc$};
  \node[label=above:$1$] at (-5,0)
        {$\vphantom{j^X}\boldsymbol\medcirc$};
  \node[label=above:$-1$] at (6,0)
        {$\vphantom{j^X}\boldsymbol\medcirc$};
  \node[label=above:$1$] at (-6,0)
        {$\vphantom{j^X}\boldsymbol\medcirc$};
  \node[label=above:$0$] at (6.7,0) {$\vphantom{j^X}$};
  \node[label=above:$0$] at (-6.7,0) {$\vphantom{j^X}$};
  \draw[-] (-7.5,0) -- (7.5,0);
\end{tikzpicture}
\end{equation*}

\begin{equation*}
  \xymatrix@C=1.2em{
    V_{1-m}
    \ar@<.5ex>[rr]^{C_{1-m}}
    && \cdots \ar@<.5ex>[rr]^{C_{-2}} \ar@<.5ex>[ll]^{D_{1-m}}
    && V_{-1} \ar@<.5ex>[rr]^{C_{-1}} \ar@<.5ex>[ll]^{D_{-2}}
    && V_- \ar[rr]_{A_-} \ar@(ur,ul)_{B_-} \ar[dr]_{b_-} \ar@<.5ex>[ll]^{D_{-1}}
    && V_0 \ar@(ur,ul)_{B_0} \ar[rr]_{A_+} \ar[dr]_{b_+} 
  && V_+ \ar@(ur,ul)_{B_+}
  \ar@<.5ex>[rr]^{C_1} && V_1 \ar@<.5ex>[ll]^{D_1} \ar@<.5ex>[rr]^{C_2} &&
  \ar@<.5ex>[ll]^{D_2} \cdots \ar@<.5ex>[rr]^{C_{n-1}}
  && V_{n-1} \ar@<.5ex>[ll]^{D_{n-1}}
  \\
  &&&&&&& \CC\ar[ur]_{a_-} && \CC \ar[ur]_{a_+}},
\end{equation*}
where $\dim V_{-k} = m-k$ ($1\le k\le m-1$), $\dim V_- = m$,
$\dim V_0 = v$, $\dim V_+=n$, $\dim V_{k} = n-k$ ($1\le k\le n-1$).

Let us apply Hanany-Witten transition to move to the following situation.
\begin{equation*}
  \xymatrix@C=1.2em{ & V_0 \ar@(ur,ul)_{B_0} \ar[rr]_{A} \ar[dr]_{b_+} 
  && V_+ \ar@(ur,ul)_{B_+}
  \ar@<.5ex>[rr]^{C_1} && V_1 \ar@<.5ex>[ll]^{D_1} \ar@<.5ex>[rr]^{C_2} &&
  \ar@<.5ex>[ll]^{D_2} \cdots \ar@<.5ex>[rr]^{C_n} && V_{N-1} \ar@<.5ex>[ll]^{D_n}
  \\
  \CC\ar[ur]_{a_-} && \CC \ar[ur]_{a_+}},
\end{equation*}
where new data has $N=m+n$, $\dim V_0 = v$, $\dim V_+ = N$,
$\dim V_1 = N-1$, \dots, $\dim V_{N-1} = 1$.
\begin{NB2}
    Note that the dimension of this bow variety is
    $2v$.
\end{NB2}%

Let us study the attracting set. We have $b_+ B_0^k a_- = 0$. But
vectors $B_0^k a_-$ ($0\le k\le v-1$) span $V_0$ by (S2). Hence
$b_+ = 0$. Then we have $B_+ A = A B_0$, therefore (S1) implies that
$A$ is injective.

By the equation $C_i D_i - D_{i+1} C_{i+1} = 0$, we see that
$B_+ = - D_1 C_1$ is nilpotent by induction. Hence $B_0$ is also
nilpotent. By (S2) $a_-$, $B_0 a_-$,\dots,
$B_0^{v-1}a_-$ is a base of $V_0$. Hence we write $B_0$, $a_-$ as
\begin{equation*}
  B_0 = J_v
  \quad
  a_- = e_v.
\end{equation*}
This normalization kills the action of $\GL(V_0)$.

Note that $\Ima A$ is $B_+$-invariant, therefore $B_+$ induces an
endomorphism of $V_+/\Ima A$. Then $(a_+\bmod\Ima A)\in V_+/\Ima A$ is
cyclic with respect to $B_+$ by (S2). Therefore
\begin{gather*}
    A B_0^{v-1}a_- = B_+^{v-1} A a_-, \quad
    A B_0^{v-2} a_- = B_+^{v-2} A a_-, \quad \dots, \quad A a_-,\\
    B_+^{N-v-1} a_+, \quad B_+^{N-v-2} a_+, \quad\dots,\quad B_+ a_+, \quad
    a_+
\end{gather*}
is a base of $V_+$. We have
\(
   A = \left[
   \begin{smallmatrix}
     \id_v \\ 0
   \end{smallmatrix}
 \right], \) \( a_+ = e_N \) and
\begin{equation*}
  B_+ = \left[
    \begin{array}{c|c|c|c}
      J_v & \vec{c} & \multicolumn{2}{c}{0} \\
      \hline
      0 & \multicolumn{3}{c}{J_{N-v}}
    \end{array}
  \right], \quad
  \vec{c} =
  \begin{bmatrix}
    c_1 \\ \vdots \\ c_v
  \end{bmatrix}
  \text{with }
  B_+^{N-v}a_+ = c_1 A B_0^{v-1} a_- + \dots + c_v A a_-.
\end{equation*}
Therefore the attracting set is $\CC^v$ parametrizing $c_1$, \dots,
$c_v$.

More generally we need to consider the case $0\ge\tau_1\ge\tau_2$. So
we assume $\dim V_+ = N$, $\dim V_{1} = N-2$, \dots,
$\dim V_{k} = N-2k$, $\dim V_{k+1} = N-2k-1$, $\dim V_{k+2} = N-2k-2$,
\dots, $\dim V_{N-k-1} = 1$,
$\dim V_{N-k} = \dim V_{N-k+1} = \cdots = 0$, where $N - 2k\ge 0$.
\begin{NB2}
    We have $N=-(\tau_1+\tau_2)$, $\tau_1 - \tau_2 = N-2k$. Therefore
    $\bv = v + \tau_1 = v - k$.
\end{NB2}%
In this case, $B_+$ is not only nilpotent, it must be in the closure
of the orbit corresponding to the partition $[N-k,k]$.
\begin{NB2}
    We have $N-k\ge k$ as $N-2k\ge 0$.
\end{NB2}%
Then we must have $N-k\ge \max(v, N-v)$.
Computing the Jordan normal form of $B_+$, we find that this condition
is equivalent to $0 = c_v = c_{v-1} = \cdots = c_{v-k+1}$.
\begin{NB2}
    E.g., $c_v = 0$ if $k=1$.
\end{NB2}%
Therefore the attracting set is $\CC^{v-k}$.
\begin{NB2}
    A rescale of $a_+$ changes $\vec{c}$ also by rescale. Therefore
    the closure of $B_+$ contains $J_v\oplus J_{N-v}$. Therefore the
    Jordan normal form has at most only two blocks. 
%
    If $0 = c_v = c_{v-1} = \cdots = c_{v-l+1}$, $c_{v-l}\neq 0$, we
    have $B_+^{N-l-1} a_+ = c_{v-l} B_+^{v-1} A a_- \neq 0$.
    \begin{NB3}
    \begin{equation*}
        \begin{split}
        B_+^{N-v}a_+ &= c_1 B_+^{v-1}A a_- + \cdots + c_{v-l} B_+^l A a_-, \\
        B_+^{N-v+1}a_+ &= c_2 B_+^{v-1} A a_- + \cdots c_{v-l} B_+^{l+1} A a_-, \\
        & \qquad \vdots \\
        B_+^{N-l-1} a_+ &= c_{v-l} B_+^{v-1} A a_- \neq 0, \\
        B_+^{N-l} a_+ &= 0.
        \end{split}
    \end{equation*}
    \end{NB3}%
    Therefore one of Jordan blocks has a size $\ge N-l$. 
    If we further suppose $N-v\le l$, then we correct $a_+$ as
    $a'_+ = a_+ - B_+^{l-(N-v)}(c_1 B_+^{v-l-1} A a_- + \cdots +
    c_{v-l} A a_-)$. Then
    \begin{equation*}
        \begin{split}
        & B_+^{N-v} a'_+ = B_+^{N-v}a_+ - B_+^l(c_1 B_+^{v-l-1} A a_- + \cdots +
        c_{v-l} A a_-) = 0, \\
        & B_+^{N-v-1} a'_+ = B_+^{N-v-1}a_+ - c_1 B_+^{v-2} a_- - \cdots -
        c_{v-l} B_+^{l-1} A a_- \neq 0.
        \end{split}
    \end{equation*}
    Together with $B_+^{v-1}A a_-$, \dots, $A a_-$, we see that the
    Jordan type is $[N-v,v]$.

    Next suppose $l\le N-v$. We correct $a_-$ as
    $a'_- = c_1 B_+^{v-l-1} A a_- + \cdots + c_{v-l-1} B_+ A a_- + c_{v-l} A a_- 
    - B_+^{N-v-l} a_+$. Then
    \begin{equation*}
        \begin{split}
            & B_+^l a'_- = c_1 B_+^{v-1} Aa_- + \cdots + c_{v-l} B_+^l
             A a_-
            - B_+^{N-v}a_+ = 0,\\
            & B_+^{l-1} a'_- = c_1 B_+^{v-2} Aa_- + \cdots + c_{v-l}
            B_+^{l-1}A a_-
            - B_+^{N-v-1}a_+ \neq 0.
        \end{split}
    \end{equation*}
    Together with $a_+$, $B_+ a_+$, \dots, $B_+^{N-l-1} a_+$, we see
    that the Jordan type is $[N-l,l]$.

    If $B_+$ is in the closure of $\mathcal O_{[N-k,k]}$, we must have
    $N-k\ge N-l$. Hence $l\ge k$. Thus $0 = c_v = \cdots = c_{v-k+1}$
    holds.

    Conversely suppose $0 = c_v = c_{v-1} = \cdots = c_{v-k+1}$. Then
    the above $l\ge k$. If the Jordan type is $[N-v,v]$, it is in
    $\mathcal O_{[N-k,k]}$ by the assumption $N-k\ge\max(v,N-v)$. If
    the Jordan type is $[N-l,l]$, it is still OK as $N-l\le N-k$,
    $l\le N-v\le N-k$.
\end{NB2}%

\end{NB}%

\subsection{Definition of operators \texorpdfstring{$e_i$}{ei}, \texorpdfstring{$f_i$}{fi}}\label{subsec:definition of operators}
Let us take $\chi$, $\chi_i$ as in \cref{subsec:twosteps}.

\begin{Proposition}\label{lem:nontrivial}
  Suppose $\cM(\lambda,\mu)^{\chi_i}$ is
  $\cM_{A_1}(\lambda',\mu')$ as in \cref{thm:smaller}.
  The hyperbolic restriction $\Phi_i 
  $ sends $\mathrm{IC}(\cM(\lambda,\mu))$ to a direct sum of
  $\mathrm{IC}(\cM_{A_1}(\kappa',\mu'))$ with various $\kappa'$ with
  $\mu'\le\kappa'\le\lambda'$.
\end{Proposition} 

\begin{proof}
  Since $\Phi_i = (p_i)_* j_i^!$ is hyperbolic semismall with respect to the
  natural stratification of $\cM(\lambda,\mu)$,
  $\cM_{A_1}(\lambda',\mu')$, the assertion means that there is no
  direct summand for an $\mathrm{IC}$ complex for a \emph{non trivial}
  local system.
  Using a deformation by $\nu^\CC$ as in \cref{subsec:deformed-case}, it
  is enough to show that $\psi\circ \Phi_i 
  (\mathrm{IC}(\underline{\cM}(\lambda,\mu)))$
  does not contain such a summand. We take $\nu^\CC$ generic so that
  fibers of $\underline{\cM}(\lambda,\mu)\to\CC$ are smooth outside
  $0$, unlike a degenerate case used in \cref{subsec:deformed-case}.
  Again as in \cref{subsec:deformed-case}, we introduce the
  corresponding real parameter to construct a topologically trivial
  family
  $\underline{\widetilde\cM}(\lambda,\mu)$. Then we
  have
  \begin{equation*}
    \psi\circ \Phi_i 
    (\mathrm{IC}(\underline{\cM}(\lambda,\mu)))
    = \pi_*(\CC_{\cM^{\nu^\RR}(\lambda,\mu)^{\chi_i}}[\dim]),
  \end{equation*}
  where the right hand side is the direct sum of constant sheaves on
  connected components of the fixed point set
  $\cM^{\nu^\RR}(\lambda,\mu)^{\chi_i}$ shifted by their
  dimensions. Note that $\cM^{\nu^\RR}(\lambda,\mu)$ and
  $\cM^{\nu^\RR}(\lambda,\mu)^{\chi_i}$ are smooth.

  We analyze the fixed point set as in the proof of
  \cref{thm:smaller}. After Hanany-Witten transitions, we arrive at
  \begin{equation}\label{eq:12}
    \begin{tikzcd}[column sep=small]
        V_m \arrow[shift left=1,r,"C_m"] & 
        V_{m-1} \arrow[shift left=1,l,"D_m"] \arrow[shift left=1]{r} &
        \cdots \arrow[shift left=1]{l} \arrow[shift left=1,r] &
        V_2 \arrow[shift left=1,l] \arrow[shift left=1,r,"C_2"] 
        & 
        V_+ \arrow[shift left=1,l,"D_2"]
        \arrow[out=120,in=60,loop,looseness=3, "B_+"]
        \arrow[dr, "b"'] \arrow[rr,"A"]
        && V_- \arrow[out=120,in=60,loop,looseness=3, "B_-"]
        \arrow[dr,"b"'] &
        \\
        &&&&& \CC \arrow[ur, "a"'] && \CC \rlap{ ,}
    \end{tikzcd}
  \end{equation}
  where dimensions of $V_m$, \dots, $V_+$, $V_-$ may differ depending
  on components. As in the proof of \cref{thm:A1} this is a
  hamiltonian reduction of the product of a quiver variety of type
  $A$, which is the cotangent bundle of a partial flag variety, and a
  variety given by the triangles by the action of $\GL(V_+)$. Since
  the hamiltonian reduction is compatible with the decomposition of
  the pushforward
  $\pi_*(\CC_{\cM^{\nu^\RR}(\lambda,\mu)^{\chi_i}}[\dim])$
  as before, the assertion follows from the corresponding result for
  nilpotent orbits of type $A$ (the connectedness of stabilizers), or
  quiver varieties \cite[Prop.~15.3.2]{Na-qaff}.
\end{proof}

\begin{NB}
  This is not precise.
\begin{Remark}
  There is no reason that the vector
  $(\dim V_m, \dots, \dim V_2, \dim V_+)$ appearing in the above proof
  is \emph{dominant}. However, it can be achieved by reflection
  functors \cite{Na-reflect} at the cost of change of the parameter
  $\nu$. Once it becomes dominant, the bow variety can be transformed
  to a balanced one by successive applications of Hanany-Witten
  transitions \cite[Prop.~7.20]{2016arXiv160602002N}. Therefore
  $\cM^{\nu^\RR}(\lambda,\mu)^{\chi_i}$ is isomorphic to a
  union of Coulomb branches of the quiver gauge theory of type $A_1$, with
  possibly different real parameter from the original $\nu^\RR$.
\end{Remark}

In order to go to a balanced bow variety, the dominance condition is
\emph{not} enough.
\end{NB}%

For a later purpose, we study the fixed point set
$\cM^{\nu^\RR}(\lambda,\mu)^{\chi_i}$. Here the parameter $\nu^\RR$ is
arbitrary, not necessarily generic in the proof of
\cref{lem:nontrivial}.

\begin{Lemma}\label{lem:perm}
  The fixed point set $\cM^{\nu^\RR}(\lambda,\mu)^{\chi_i}$ is a union
  of $A_1$ type balanced bow varieties with real parameters induced
  from $\nu^\RR$.
  \begin{NB}
    with parameter ${}'\nu^\RR$, which is a permutation of $\nu^\RR$.
  \end{NB}%
\end{Lemma}

Here the induced parameters mean the following: recall $\nu_h^\RR$ is
assigned for each $h=h_\sigma$ ($1\le\sigma\le \ell$). We take the
universal covering of the bow diagram so that $\sigma$ runs over
$\ZZ$. Then $\{ h_\sigma\}$ for $A_1$ is a subset of
$\{ h_\sigma \mid \sigma\in\ZZ\}$, and the parameters $\nu_h^\RR$ are
given by the restriction.

\begin{proof}
  Let us change the fixed point component to the form of \eqref{eq:12}
  as in the proof of \cref{lem:nontrivial}.

  We will show that the balanced condition is achieved by first
  applying reflection functors in \cite{Na-reflect}, then next
  applying Hanany-Witten transitions.
  We consider the deformation $\underline{\cM}$ as in the proof of
  \cref{lem:nontrivial}. Since reflection functors are hyperK\"ahler
  isometry, $\cM^{\nu^\RR}$ and $\cM^{\nu^\CC}$ are changed in
  the same way. Also Hanany-Witten transitions respects complex and
  real parameters. Therefore it is enough to show the statement for
  $\cM^{\nu^\CC}$.

  Let $N_{h_2} = \dim V_+ - \dim V_2$, \dots, $N_{h_{m+1}} = \dim V_m$
  as before. By applying reflection functors in \cite{Na-reflect} at
  the cost of change of the parameter $\nu$, we achieve the dominance
  condition $N_{h_2} \ge N_{h_3} \ge \dots \ge N_{h_{m+1}}$. By an
  argument in \cite[Prop.~7.5]{2016arXiv160602002N} we can transform
  the bow diagram to a balanced one by Hanany-Witten transition, if
  $N_{h_2}\le 2$, the number of triangles in \eqref{eq:12}.

  As in the proof of \cref{prop:deformed}, the data \eqref{eq:12}
  factorizes according to eigenvalues of $B_+$, which are entries of
  $\nu^\CC$. Moreover we can normalize $C_h$ to the identity on a
  component for an eigenvalue $\neq \nu^\CC_h$. Then each factor is an
  $A_1$ type bow variety with the parameter $0$, which was studied
  during the proof of \cref{thm:smaller}. In particular, we have
  $N_{h_\xp} = 0$, $1$ or $2$ in each factor. Since other factors do
  not contribute to $N_{h_\xp}$, we have $N_{h_{\xp}}\le 2$, in
  particular, $N_{h_2}\le 2$.

  Once we achieve the balanced condition, the ordering on parameters
  $\nu_h$ is irrelevant by \cref{rem:reflect}.
\end{proof}

\begin{NB}
In the above argument, we only know that ${}'\nu^\RR$ is a permutation
of $\nu^\RR$ by the symmetric group for all $\boldsymbol\medcirc$'s,
as we have used Hanany-Witten transition. We do not know whether this
would be a permutation among $\bw_i$ $\boldsymbol\medcirc$'s between
$\xl_i\to\xl_{i+1}$.
\end{NB}%

We can apply \cite[\S4]{2018arXiv180511826B} after changing the
stability parameter in the decreasing order so that
\cite[(4.2)]{2018arXiv180511826B} is satisfied.
Hence $\cM^{\nu^\RR}(\lambda,\mu)^{\chi_i}$ is isomorphic to a union
of Coulomb branches of the quiver gauge theory of type $A_1$ with
parameter induced from the original $\nu^\RR$.
\begin{NB}
  Once we arrive at the Coulomb branch, the ordering does not matter.
\end{NB}%

\begin{Remark}
  Let us give an alternative argument for \cref{lem:perm} for finite
  type $A$. It is informed us by Finkelberg.
  \begin{NB}
    E-mail on Oct.~10, 2018.
  \end{NB}%

  Recall that $\cM^{\nu^\RR}(\lambda,\mu)$ is realized as an iterated
  convolution diagram
  $\widetilde{\mathcal W}^{\underline{\lambda}}_\mu$ in
  \cite[\S5]{2018arXiv180511826B}. For type $A_n$, it is the
  moduli space of flags of lattices $L_0\supset L_1\supset\dots\supset
  L_N$ such that $L_{s-1}\supset L_s\supset zL_{s-1}$, and $\dim
  L_{s-1}/L_s=i_s$, $1\leq i_s\leq n$, where $L_0=V[[z]]$, $V=\CC
  v_0\oplus\dots\oplus\CC v_n$.
  Here the sequence $i_1,i_2,\dots, i_N$ is as follows. We consider
  coordinates of cocharacter of $G_F$ corresponding to $\nu^\RR$. We
  choose $i_1$ so that the maximum $k_1$ of coordinates is achieved at the
  vertex $i_1$ of the Dynkin diagram. Then we choose $i_2$ so that the
  next one $k_2$ is achieved at $i_2$, and so on.
  \begin{NB}
  You may view $L_s$ as the sections of a vector bundle
  $\mathcal{V}_s$ on $\mathbb{P}^1$ restricted to the formal
  neighborhood of $0\in\mathbb{P}^1$, and you may add a flag in
  $\mathcal{V}_N$, but it will affect nothing.
  \end{NB}%
  The determinant bundle is $\mathcal{D}_s=\det(L_{s-1}/L_s)$ (or its
  dual?), and the ample line bundle for the stability parameter
  $\nu^\RR$ is $\bigotimes\mathcal D_s^{\otimes k_s}$. Let our Levi
  subgroup be $\GL(V':=\mathbb{C} v_0\oplus\mathbb{C} v_1)$ times the
  remaining torus $T''$. Then the corresponding torus fixed points in
  the above convolution diagram will have many connected components. A
  typical connected component will be the moduli space of flags
  $(L'_0\supset L'_1\supset\dots\supset L'_N)\oplus(L''_0\supset
  L''_1\supset\dots\supset L''_N)$ where the first summand is a flag
  in $L'_0=V'[[z]]$, while the second summand is a $T''$-invariant
  flag in $L''_0=V''[[z]]$, where
  $V''=\CC v_2\oplus\dots\oplus\mathbb{C} v_n$.
  \begin{NB}
    We must have $0\leq\dim L'_{s-1}/L'_s\leq i_s$ evidently;
    otherwise no restrictions (and there are finitely many
    possibilities for the choice of $L''_s$).
  \end{NB}%
  There are finitely many possibilities for choices of $L''_s$.
  Clearly,
  $\det(L_{s-1}/L_s)=\det(L'_{s-1}/L'_s)\otimes
  \det(L''_{s-1}/L''_s)$, and the second factor is ``constant" from
  the point of view of $\algsl(2)$-Grassmannian. Thus the ample line
  bundle for the fixed point component is
  $\bigotimes\det(L'_{s-1}/L'_s)^{\otimes k_s}$.
  \begin{NB}
    We omit the $s$-th factor if $L'_{s-1} = L'_s$.
  \end{NB}%
\end{Remark}

Thanks to \cref{lem:nontrivial}, we write the hyperbolic restriction
$\Phi_i 
(\mathrm{IC}(\cM(\lambda,\mu)))$ as in \eqref{eq:11}.
The multiplicity space $M^{\lambda,\mu}_{\kappa',\mu'}$ is top degree
Borel-Moore homology group of a subvariety $p_i^{-1}(x)$, where $x$ is
a point in the smooth locus of a stratum
$\cM^{\mathrm{s}}_{A_1}(\kappa',\mu')$ of
$\cM_{A_1}(\lambda',\mu') = \cM(\lambda,\mu)^{\chi_i}$. See
\cref{subsec:twosteps}. By \cref{lem:nontrivial} the top degree
Borel-Moore homology group form a trivial local system over
$\cM^{\mathrm{s}}_{A_1}(\kappa',\mu')$. Therefore its fibers are
canonically identified, i.e., independent of the choice of $x$.

\begin{Proposition}\label{prop:naturaliso}
  The factorization gives us an isomorphism
  \begin{equation*}
    M^{\lambda,\mu}_{\kappa',\mu'} \cong M^{\lambda,\mu-\alpha_i}_{\kappa',\mu'-2}.
  \end{equation*}
\end{Proposition}

The construction will be explained during the proof.

\begin{NB}
Note that $\cM_{A_1}(\kappa',\mu'+2)$ is empty if $\kappa' < \mu'+2$.
\end{NB}%

\begin{proof}
  Recall that the $i$-th component of the factorization morphism
  $\Phi$ of $\cM(\lambda,\mu)$ is restricted to the factorization
  morphism of $\cM_{A_1}(\lambda',\mu')$ (for which
  $\cM_{A_1}(\kappa',\mu')$ is a stratum).

  We consider the Coulomb branch of another quiver gauge theory
  obtained by increasing the $i$-th entry of the vector
  $\underline{\bv}$ by $1$, i.e., $\cM(\lambda,\mu-\alpha_i)$. Let
  $\underline{\delta_i}$ be the dimension vector whose $i$-entry is
  $1$ and other entries are $0$. We take the open subset
  $(\BA^{\underline{\bv}}\times \mathbb
  G_m^{\underline{\delta_i}})_{\mathrm{disj}}$ of
  $\BA^{\underline{\bv}}\times\BA^{\underline{\delta_i}}$ consisting
  of pairs of colored configurations whose supports are disjoint as in
  \cref{subsec:factor}, and also the support of the second
  configuration is disjoint from $0$. Then we have the factorization
  isomorphism
  \begin{multline}\label{eq:h:3}
    \cM(\lambda,\mu-\alpha_i)
    \times_{\BA^{\underline{\bv}+\underline{\delta_i}}}
    (\BA^{\underline{\bv}}\times
    \mathbb G_m^{\underline{\delta_i}})_{\mathrm{disj}}\\
    \cong 
    (
    \cM(\lambda,\mu)\times \BA\times\mathbb G_m
    )\times_{
      \BA^{\underline{\bv}}\times\BA^{\underline{\delta_i}}} 
    (\BA^{\underline{\bv}}\times
    \mathbb G_m^{\underline{\delta_i}})_{\mathrm{disj}}.
  \end{multline}
  The second factor is $\cM(\lambda,\lambda-{\alpha_i})$ in the
  statement of \cref{thm:fact}. It can be replaced by the Coulomb
  branch of type $A_1$ theory with $\bw = {\bw}_i$, $\bv = 1$ as
  entries of $\underline{\delta_i}$ are $0$ except at $i$. Moreover we
  restrict it to the open subset where the factorization morphism is
  nonzero, hence we can further replace it by the Coulomb branch of
  the quiver gauge theory of type $A_1$ with $\bw = 0$, $\bv = 1$,
  which is just $\BA\times\mathbb G_m$ as above.

  This factorization is compatible with the one parameter subgroup
  $\chi_i$, as $\chi_i$ does not change the $i$-th component of
  $\Phi$. Note $\chi_i$ acts trivially on the factor
  $\BA\times\mathbb G_m$. We have the factorization
  $\cM_{A_1}^{\mathrm{s}}(\kappa',\mu'-2) \approx
  \cM_{A_1}^{\mathrm{s}}(\kappa',\mu')\times (\BA\times\mathbb G_m)$
  for strata of $\chi_i$-fixed point sets, compatible with
  $\cM(\lambda,\mu-\alpha_i)\approx
  \cM(\lambda,\mu)\times(\BA\times\mathbb G_m)$. Thus the desired
  isomorphism is given by the factorization.
\end{proof}

By \cref{prop:factor} and \eqref{eq:11} we define operators $e_i$,
$f_i$ on $\bigoplus_\mu \mathcal V_\mu(\lambda)$ induced from $e$, $f$
given by \cref{thm:A1} for the $A_1$ case, as explained in
Introduction.

\subsection{Tensor product}\label{subsec:tensor-product}

Let us slightly generalize the above construction of $e_i$, $f_i$. We
replace $\mathrm{IC}(\cM(\lambda,\mu))$ by
$\pi_*(\mathrm{IC}(\cM^{\nu^{\bullet,\RR}}(\lambda,\mu)))$. Here
$\nu^{\bullet,\RR}$ is as in \cref{subsec:deformed-case}. Then
$\Phi_i\circ\pi_*(\mathrm{IC}(\cM^{\nu^{\bullet,\RR}}(\lambda,\mu)))$
decomposes as
$\bigoplus_{\kappa'} \widetilde{M}^{\lambda,\mu}_{\kappa',\mu'}\otimes
\mathrm{IC}(\cM_{A_1}(\kappa',\mu'))$ as in \eqref{eq:11}.  Then we
construct an isomorphism
$\widetilde{M}^{\lambda,\mu}_{\kappa',\mu'} \cong
\widetilde{M}^{\lambda,\mu-\alpha_i}_{\kappa',\mu'-2}$ in the same way
as in the proof of \cref{prop:naturaliso} by the factorization.
Then we define $e_i$, $f_i$ as
$(\text{this isomorphism})\otimes (e, f \text{ for $A_1$ case})$.

This construction is compatible with the original one: Since
$\mathrm{IC}(\cM(\lambda,\mu))$ is a direct summand of
$\pi_*(\mathrm{IC}(\cM^{\nu^{\bullet,\RR}}(\lambda,\mu)))$, we
have the induced inclusion and projection 
\(
   M^{\lambda,\mu}_{\kappa',\mu'}
   \leftrightarrows\widetilde{M}^{\lambda,\mu}_{\kappa',\mu'}.
\)
They commute with isomorphisms
$M^{\lambda,\mu}_{\kappa',\mu'} \cong
M^{\lambda,\mu-\alpha_i}_{\kappa',\mu'-2}$,
$\widetilde{M}^{\lambda,\mu}_{\kappa',\mu'} \cong
\widetilde{M}^{\lambda,\mu-\alpha_i}_{\kappa',\mu'-2}$ by the
construction. Therefore maps in \eqref{eq:9} intertwine $e_i$, $f_i$.

Now recall
$\Phi\circ\pi_*(\mathrm{IC}(\cM^{\nu^{\bullet,\RR}}(\lambda,\mu)))$
decomposes into a sum of tensor product by \cref{cor:tensor}. Hence we
have
\begin{equation}\label{eq:13}
    \bigoplus_\mu\Phi\circ\pi_*(\mathrm{IC}(\cM^{\nu^{\bullet,\RR}}(\lambda,\mu)))
    \cong \mathcal V(\lambda^1)\otimes\mathcal V(\lambda^2).
\end{equation}

\begin{Proposition}\label{prop:tensor}
    \textup{(1)} Homomorphisms in \eqref{eq:9} intertwine
    operators $e_i$, $f_i$.

    \textup{(2)} The operators $e_i$, $f_i$ defined on the left hand
    side of \eqref{eq:13} just above is equal to the tensor product in
    the right hand side, i.e., $e_i$ is given by
    $1\otimes e_i + e_i \otimes 1$, etc.
\end{Proposition}

\begin{proof}
    The statement (1) is already proved above.

    In order to understand
    $\Phi\circ\pi_*(\mathrm{IC}(\cM^{\nu^{\bullet,\RR}}(\lambda,\mu)))$
    in \eqref{eq:13}, we consider
    $\psi\circ\Phi(\mathrm{IC}(\underline{\cM}(\lambda,\mu)))$ as in
    \cref{cor:tensor} and change it to
    $\Phi^i\circ\psi\circ\Phi_i(\mathrm{IC}(\underline{\cM}(\lambda,\mu)))$.
    Then the triviality of the family
    $\underline{\widetilde\cM}(\lambda,\mu)\to \CC$ and the
    commutativity of the nearby cycle and hyperbolic restriction
    functors give an isomorphism
    \begin{equation*}
      \Phi_i\circ\pi_*(\mathrm{IC}(\cM^{\nu^{\bullet,\RR}}(\lambda,\mu)))
      \cong
      \psi\circ\Phi_i(\mathrm{IC}(\underline{\cM}(\lambda,\mu))).
    \end{equation*}
    By \cref{lem:perm} (or more precisely its version for the complex
    parameter ${\nu^\CC}$),
    $\Phi_i(\mathrm{IC}(\underline{\cM}(\lambda,\mu)))$ is a direct
    sum
    $\bigoplus \overline{M}^{\lambda,\mu}_{\kappa',\mu'}\otimes
    \mathrm{IC}(\underline{\cM}_{A_1}(\kappa',\mu'))$ over the inverse
    image of $\CC\setminus\{0\}$ under
    $\underline{\cM}_{A_1}(\kappa',\mu')\to \CC$.
    \begin{NB}
      It is probably better to remember that
      $\overline{M}^{\lambda,\mu}_{\kappa',\mu'}$ is independent of
      $\mu$, $\mu'$. By applying $\Phi^i\circ\psi$, we get
      $\bigoplus \overline{M}^{\lambda,\mu}_{\kappa',\mu'}\otimes
      \mathcal V_{\mathfrak l_i,\mu^{1\prime}}(\kappa^{1\prime})
      \otimes\mathcal V_{\mathfrak
        l_i,\mu^{2\prime}}(\kappa^{2\prime})$. Then
      $\overline{M}^{\lambda,\mu}_{\kappa',\mu'}$ is a multiplicity of
      $\mathcal V_{\mathfrak l_i,\mu^{1\prime}}(\kappa^{1\prime})
      \otimes\mathcal V_{\mathfrak
        l_i,\mu^{2\prime}}(\kappa^{2\prime})$ in
      $\mathcal V(\lambda^1)\otimes \mathcal V(\lambda^2)$.
    \end{NB}%
    Moreover $\underline{\cM}(\lambda,\mu)$ factors as
    \begin{equation*}
      \underline{\cM}(\lambda^1,\mu^1)\times
      \underline{\cM}(\lambda^2,\mu^2)
    \end{equation*}
    around a $T$-fixed point by \cref{prop:deformed}. Therefore
    $\overline{M}^{\lambda,\mu}_{\kappa',\mu'}$ is the tensor product
    $M^{\lambda^1,\mu^1}_{\kappa^{1\prime},\mu^{1\prime}}
    \otimes M^{\lambda^2,\mu^2}_{\kappa^{2\prime},\mu^{2\prime}}$.

    The factorization gives an isomorphism
    $\overline{M}^{\lambda,\mu}_{\kappa',\mu'}\cong
    \overline{M}^{\lambda,\mu-\alpha_i}_{\kappa',\mu'-2}$ as in
    \cref{prop:naturaliso}. We can apply two factorization
    simultaneously.
    Therefore the isomorphism is induced from isomorphisms for factors
    $M^{\lambda^1,\mu^1}_{\kappa^{1\prime},\mu^{1\prime}}$,
    $M^{\lambda^2,\mu^2}_{\kappa^{2\prime},\mu^{2\prime}}$, and it is
    also independent of the choice of the decomposition $\mu=\mu^1+\mu^2$.
    
    Thus it is enough to check the assertion
    $e_i = e_i\otimes 1 + 1\otimes e_i$ for
    \begin{equation*}
      \psi\circ\Phi^i(\mathrm{IC}(\underline{\cM}_{A_1}(\kappa',\mu')))
      \cong \bigoplus_{\mu^{1\prime}+\mu^{2\prime} = \mu'}
      \Phi^i(\mathrm{IC}(\cM_{A_1}(\kappa^{1\prime},\mu^{1\prime})))
      \otimes
      \Phi^i(\mathrm{IC}(\cM_{A_1}(\kappa^{2\prime},\mu^{2\prime}))),
    \end{equation*}
    the isomorphism given by the factorization as a special case of
    \cref{cor:tensor} for type $A_1$. But this is clear from the
    definition as explained in the proof of \cref{thm:A1}.
    \begin{NB}
        Earlier attempt.....

    Recall \eqref{eq:9} comes from projection and inclusion
    $\mathrm{IC}(\cM(\lambda,\mu))\leftrightarrows
    \pi_*(\mathrm{IC}(\cM^{\nu^\RR}(\lambda,\mu)))
    $,
    applied with $\Phi$.
    Let us use the factorization $\Phi
    = \Phi^i\circ\Phi_i$ and $\Phi_i(\mathrm{IC}(\cM(\lambda,\mu))) =
    \bigoplus M^{\lambda,\mu}_{\kappa',\mu'}\otimes
    \mathrm{IC}(\cM_{A_1}(\kappa',\mu'))$.
    \begin{equation*}
        \mathrm{IC}(\cM_{A_1}(\kappa',\mu'))
        \leftrightarrows \pi_*(\mathrm{IC}(\cM_{A_1}^{{}'\nu^\RR}(\kappa',\mu')))
    \end{equation*}
    \end{NB}%
\end{proof}

\begin{NB}
  \begin{equation*}
    \Phi(\pi_*\mathrm{IC}(\cM^{\nu^\RR}(\lambda,\mu))) =
    \Phi\circ\psi(\mathrm{IC}(\underline{\cM}^{\CC\nu^\bullet}(\lambda,\mu)))
    = \psi\circ\Phi(\mathrm{IC}(\underline{\cM}^{\CC\nu^\bullet}(\lambda,\mu)))
  \end{equation*}

  $\psi\circ \Phi = \psi \circ \Phi^i\circ \Phi_i = \Phi^i\circ \psi \circ\Phi_i$.
\end{NB}%

\subsection{Type \texorpdfstring{$A_2$}{A2}}\label{subsec:a_2-case}

We next show the relation $[e_i, f_j] = 0$ if $i\neq j$ and the Serre
relation. This is reduced to the rank $2$ case. If $i$ and $j$ are not
connected in the Dynkin diagram, the bow variety decomposes into a
product. The assertion is trivial. Next we study the $A_2$
case. Thanks to \cref{prop:tensor}, we may assume $\lambda$ is a
fundamental weight.  We may further assume $\lambda = \Lambda_1$, the
first fundamental weight, by a diagram automorphism. The bow variety
$\cM_{A_2}(\lambda,\mu)$ has a fixed point if and only if
$\mu = \Lambda_1$, $\Lambda_1 - \alpha_1$,
$\Lambda_1 - \alpha_1 - \alpha_2$. We apply Hanany-Witten transition
to go to a bow diagram
\begin{equation*}
  \begin{tikzpicture}[baseline=(current  bounding  box.center)]
  \node at (-1,0)
        {$\vphantom{j^X}\boldsymbol\medcirc$};
  \node at (0,0)
        {$\vphantom{j^X}\boldsymbol\times$};
  \node at (1,0)
  {$\vphantom{j^X}\boldsymbol\times$};
  \node at (2,0)
  {$\vphantom{j^X}\boldsymbol\times$};
  \node[label=above:$0$] at (2.5,0)
  {};
  \node[label=above:$\bv_2$] at (1.5,0)
  {};
  \node[label=above:$\bv_1$] at (0.5,0)
  {};
  \node[label=above:$1$] at (-0.5,0)
  {};
  \node[label=above:$0$] at (-1.5,0)
  {};
  \draw[-] (3,0) -- (-2,0);
\end{tikzpicture}
\end{equation*}
with $(\bv_1,\bv_2) = (0,0)$, $(1,0)$, $(1,1)$.

In the case $\mu = \Lambda_1$, the bow variety is a single point. Let
$\fA_1$ denote the corresponding attracting set, which is also a
single point. We have $e_1 [\fA_1] = 0 = e_2[\fA_1]$ as the
corresponding bow varieties are empty. We also have $f_2[\fA_1] = 0$
since the corresponding bow variety does not have torus fixed points.

Next consider the case $\mu=\Lambda-\alpha_1$. We have
\begin{equation*}
  \begin{tikzcd}
    \CC \arrow[rr, "A_1"] \arrow[rd, "b_1"']
    \arrow[out=120,in=60,loop,looseness=3,"B_0=0"]
    && \CC 
    \arrow[out=120,in=60,loop,looseness=3,"B_1"]
    \arrow[rd,"b_2"']
    && 0 
    & \\
    & \CC \arrow[ru,"a_1"'] && \CC 
    && \CC
  \end{tikzcd}
\end{equation*}
We normalize $A_1 = b_2 = 1$, and determine $B_1$ from the equation
$B_1 + a_1 b_1 = 0$. Therefore $\cM_{A_2}(\lambda,\mu) \cong \CC^2$ by
the remaining variables $(a_1,b_1)$. The action is
\begin{equation*}
  (a_1,b_1) \mapsto (t^{m_1} a_1, t^{-m_1} b_1).
\end{equation*}
The attracting set is $\{ a_1 = 0\}$, which is $\CC$. Let us denote it
by $\fA_2$. We have $e_2[\fA_2] = 0 = f_1[\fA_2]$ as the corresponding
bow varieties are empty.

For the case $\mu = \Lambda_1 - \alpha_1 - \alpha_2$, we have
\begin{equation*}
  \begin{tikzcd}
    \CC \arrow[rr, "A_1"] \arrow[rd, "b_1"']
    \arrow[out=120,in=60,loop,looseness=3,"B_0=0"]
    && \CC \arrow[rr, "A_2"]
    \arrow[out=120,in=60,loop,looseness=3,"B_1"]
    \arrow[rd,"b_2"']
    && \CC \arrow[rd, "b_3"'] 
    \arrow[out=120,in=60,loop,looseness=3,"B_2"]
    & \\
    & \CC \arrow[ru,"a_1"'] && \CC \arrow[ru,"a_2"'] && \CC
  \end{tikzcd}
\end{equation*}
We normalize $A_1 = A_2 = b_3 = 1$, and determine $B_1$, $B_2$ from
the equations $B_1 + a_1 b_1 = 0$, $B_2 - B_1 + a_2 b_2 =
0$.
Therefore $\cM_{A_2}(\lambda,\mu) \cong \CC^4$ by the remaining
variables $(a_1,b_1,a_2,b_2)$. The action is
\begin{equation*}
  (a_1,b_1,a_2,b_2) \mapsto 
  (t^{m_1+m_2} a_1, t^{-m_1-m_2} b_1, t^{m_2}a_2, t^{-m_2}b_2).
\end{equation*}
The attracting set is $\{ a_1 = a_2 = 0\}$, which is $\CC^2$. Let us
denote it by $\fA_3$. We have $f_1[\fA_3] = 0 = f_2[\fA_3]$,
$e_1[\fA_3] = 0$ as above.

In order to calculate remaining actions of operators $e_1$, $e_2$,
$f_1$, $f_2$, we take one parameter subgroups with $m_1 = 0$,
$m_2 < 0$ and $m_1 < 0$, $m_2 = 0$ respectively.
\begin{NB}
  We are in the negative chamber $m_1$, $m_2 < 0$ and the boundaries
  are $m_1 = 0$ or $m_2 = 0$.
\end{NB}%

When $m_1 = 0$, the action is trivial in the case
$\mu = \Lambda_1 - \alpha_1-\alpha_2$. Therefore the hyperbolic
restriction does nothing. Hence we have
\begin{equation*}
   f_1 [\fA_1] = [\fA_2], \quad
   e_1 [\fA_2] = [\fA_1].
\end{equation*}

When $m_2 = 0$, the attracting set remains $\fA_2$, and the fixed
point is a single point $(a_1, b_1) = 0$ for $\mu=\Lambda_1-\alpha_1$,
while the attracting set is $\{ a_1 = 0\}\cong \CC^3$ and the fixed
point set is $\{ a_1 = b_1 = 0\} \cong\CC^2$ for
$\mu = \Lambda_1-\alpha_1-\alpha_2$. Therefore we have
\begin{equation*}
  f_2[\fA_2] = [\fA_3], \quad e_2[\fA_3] = [\fA_2].
\end{equation*}
This finishes the calculation, and we see that this gives the
$3$-dimensional standard representation of $\algsl(3)$.

\subsection{Reduction to \texorpdfstring{$A_\infty$}{A∞} case}

We are left to check the case affine $A_1$. As in
\cref{subsec:a_2-case} we may assume $\lambda$ is a fundamental
weight, and $\lambda = \Lambda_0$ by the diagram automorphism. The
following argument works for general $n\ge 2$.

We apply the method used in \cref{subsec:tensor-product} to
$\cM^{\nu^{\Box,\RR}}(\lambda,\mu)$,
$\cM^{\nu^{\Box,\CC}}(\lambda,\mu)$ studied in
\cref{subsec:anoth-choice-stab}. We have
\begin{equation}\label{eq:22}
     \bigoplus_\mu\Phi\circ\pi_*((\mathrm{IC}
   (\cM^{\nu^{\Box,\RR}}(\lambda,\mu))))
   \cong 
   \bigoplus_{\mu'} \Phi(\mathrm{IC}(\cM_{A_\infty}(\lambda',\mu'))),
\end{equation}
where $\lambda'$ is the $0$-th fundamental weight for $A_\infty$.  As
in \cref{prop:tensor}, we define operators $e_i$, $f_i$ on the left
hand side, and ask what they are in the right hand side. In
$\cM_{A_\infty}(\lambda',\mu')$ it is straightforward to check that
the fixed point set $\cM^{\nu^{\Box,\CC}}(\lambda,\mu)^{\chi_i}$ with
respect to the degenerate one-parameter subgroup $\chi_i$ is mapped to
the product of ${A_1}$-type bow varieties for $\chi_{i+mn}$
($m\in\ZZ$) under the isomorphism in \cref{prop:unwind}. Therefore
$e_i$, $f_i$ are given by $\sum_{m\in\ZZ} e_{i+mn}$,
$\sum_{m\in\ZZ} f_{i+mn}$ in the right hand side. This is nothing but
an embedding of $\widehat{\algsl}(n)$ into
$\widehat{\mathfrak{gl}}(\infty)$, see e.g.,
\cite[Lecture~9]{MR3185361}. In particular the relation
$[e_i, f_j] = 0$ for $i\neq j$ and the Serre relation are satisfied
also for $\widehat{\algsl}(2)$.

Let us look at \eqref{eq:21}. It is an isomorphism of
$\widehat{\algsl}(n)$-modules. The left hand side is the restriction
of the Fock space for $\widehat{\mathfrak{gl}}(\infty)$ to
$\widehat{\algsl}(n)$. In the right hand side the first factor is
$V(\Lambda_0)$, while the second factor corresponds to the Fock space
for the Heisenberg subalgebra in $\widehat{\mathfrak{gl}}(\infty)$.

\begin{Remark}\label{rem:Hilb2}
  Consider the case $n=1$. \eqref{eq:22} remains to be true. As we
  mentioned in \cref{rem:Hilb} the left hand side is
  $\bigoplus_{\underline{k}}
  \CC[\overline{S^{\underline{k}}(\CC)}]$. This space is equipped with
  the structure of the Fock space of the Heisenberg algebra so that
  $[\overline{S^{\underline{k}}(\CC)}]$ corresponds to the monomial
  symmetric function for the partition $\underline{k}$
  \cite[\S2]{MR3586508}.

  On the other hand the summand
  $\Phi(\mathrm{IC}(\cM_{A_\infty}(\lambda',\mu')))$ in \eqref{eq:22}
  corresponds to a fixed point in $\cM^{\nu^{\Box,\CC}}(\lambda,\mu)$
  corresponding to a partition given by
  $\lambda'-\mu' = \sum_m \bv(m)\alpha_m$. (See \cref{sec:Maya} how
  $\bv(m)$ corresponds to a Maya diagram. Then we use the standard
  bijection between a partition and a Maya diagram.) This fixed point
  was studied by Vasserot \cite{VasserotC2}: Its class is the Schur
  function for $\underline{k}$.

  Since the correspondence Maya diagrams and Schur functions appears
  in boson-fermion correspondence (see e.g.,
  \cite[Lectures~5,6]{MR3185361}), we see that \eqref{eq:22} respects
  the Heisenberg algebra action where we consider the Heisenberg
  subalgebra in $\widehat{\mathfrak{gl}}(\infty)$ in the right hand side.
\end{Remark}

\subsection{Kashiwara crystal}\label{subsec:crystal}

Recall $\mathcal V_\mu(\lambda)$ has a base parametrized by
irreducible components of the attracting set $\fA_\chi(\lambda,\mu)$
of $\dim = \dim\cM(\lambda,\mu)/2$. (\cref{rem:MV}) In
\cite[Remark~3.26(2)]{2016arXiv160403625B} it was conjectured that the
union of irreducible components
$\bigsqcup\operatorname{Irr}\fA_\chi(\lambda,\mu)$ has a structure of
Kashiwara crystal, isomorphic to $B(\lambda)$, the crystal of the
integrable highest weight module of the quantized enveloping algebra.

As we mentioned in the Introduction, our construction resembles the
construction of Kashiwara crystal structure in \cite{bragai}, it is
straightforward to apply the construction in \cite{bragai} to our
setting. Let us briefly sketch. We use the standard notation for
crystal, e.g., as in \cite{Na-tensor}.

\begin{enumerate}
\item We define Kashiwara operators $\widetilde e$, $\widetilde f$ for
  $\algsl(2)$ from the analysis in \cref{subsec:attrA1}. Namely they
  send $[\fA_\chi(\lambda,\mu)]$ to $[\fA_\chi(\lambda,\mu\pm 2)]$ or
  $0$.

\item We define $\widetilde e_i$, $\widetilde f_i$ in general by
  reduction to $\algsl(2)$ by the hyperbolic restriction with respect
  to $\chi_i$ as in \cref{subsec:definition of operators}. In
  particular, we use the factorization isomorphism appeared in the
  proof of \cref{prop:naturaliso}.

\item We consider irreducible components of attracting sets in
  $\cM^{\nu^{\bullet,\CC}}(\lambda,\mu)$ and
  $\cM^{\nu^{\bullet,\RR}}(\lambda,\mu)$ as in
  \cref{subsec:deformed-case}. They are naturally identified via the
  topologically trivial family
  $\widetilde{\underline{\cM}}(\lambda,\mu)$. Let us denote it by
  $\operatorname{Irr}\fA_\chi^{\nu^\bullet}(\lambda,\mu)$. Then
  $\bigsqcup_\mu\operatorname{Irr}\fA_\chi^{\nu^\bullet}(\lambda,\mu)$
  has a Kashiwara crystal structure. Moreover
  \begin{aenume}
  \item The projection
    $\pi\colon \cM^{\nu^{\bullet,\RR}}(\lambda,\mu)\to
    \cM(\lambda,\mu)$ induces an inclusion
    $\bigsqcup\operatorname{Irr}\fA_\chi(\lambda,\mu)\subset
    \bigsqcup_\mu\operatorname{Irr}\fA_\chi^{\nu^\bullet}(\lambda,\mu)$
    which is an embedding of crystals.
  \item the factorization in \cref{prop:deformed} induces an
    isomorphism
    \begin{equation*}
     \bigsqcup_\mu\operatorname{Irr}\fA_\chi^{\nu^\bullet}(\lambda,\mu)
    \cong 
    \bigsqcup_{\mu^1} 
    \operatorname{Irr}\fA_\chi(\lambda^1,\mu^1)\otimes
    \bigsqcup_{\mu^2}
    \operatorname{Irr}\fA_\chi(\lambda^1,\mu^1) 
    \end{equation*}
    of crystals.
  \end{aenume}
\item If we view $[Y]\in\operatorname{Irr}\fA_\chi(\lambda,\mu)$ as an
  element of $H_{2\dim}(\fA_\chi(\lambda,\mu))\cong V_\mu(\lambda)$,
  Kashiwara operator $\widetilde f_i$ and $f_i$ in the Lie algebra are
  related as
  \begin{equation*}
    f_i[Y] = (\varepsilon_i(Y)+1) \widetilde f_i[Y]
    + \sum_{Y': \varepsilon_i(Y') > \varepsilon_i(Y)+1} c_{Y'} [Y']
  \end{equation*}
  for some constants $c_{Y'}$. (This property was explained in a
  different way in \cite[Prop.~4.1]{bragai}.)
\end{enumerate}

\begin{NB}
  \begin{equation*}
    f (\frac{f^k}{k!}v_\lambda) = (k+1) \frac{f^{k+1}}{(k+1)!}v_\lambda.
  \end{equation*}
\end{NB}%

As in \cite{bragai} the only remaining property we need to check is
the highest weight property: for any
$[Y]\in\operatorname{Irr}\fA_\chi(\lambda,\mu)$ which is not
$[\fA_\chi(\lambda,\lambda)]$, there exists $i$ such that
$\widetilde e_i [Y] \neq 0$. This will be discussed in the next
subsection. At this moment, if we replace
$\bigsqcup\operatorname{Irr}\fA_\chi(\lambda,\mu)$ by the connected
component containing $[\fA_\chi(\lambda,\lambda)]$, it is isomorphic
to the crystal $B(\lambda)$.

\subsection{Irreducibility}

So far we have constructed a $\mathfrak g_{\mathrm{KM}}$-module
structure on $\mathcal V(\lambda)$. It is integrable and has a vector
$v_\lambda$ correspond to the fundamental class of
$\fA_\chi(\lambda,\lambda)$ which is killed by all $e_i$ by definition.
It remains to show that $\mathcal V(\lambda)$ is generated by
$v_\lambda$. By the construction in \cref{subsec:crystal} it follows
once we show that $\bigsqcup\operatorname{Irr}(\fA_\chi(\lambda,\mu))$
has the highest weight property. Conversely if we show that
$\mathcal V(\lambda)$ is generated by $v_\lambda$, there are no other
irreducible components, hence
$\bigsqcup\operatorname{Irr}(\fA_\chi(\lambda,\mu))\cong B(\lambda)$.

We expect that there is a direct argument showing the highest weight
property of crystal, but we give two indirect arguments.

Let us show that the number of irreducible components in
$\fA_\chi(\lambda,\mu)$ is equal to the weight multiplicities. Since
we have constructed a $\mathfrak g_{\mathrm{KM}}$-module structure, we
can assume $\mu$ is dominant. Then the bow variety $\cM(\lambda,\mu)$
is isomorphic to a quiver variety of affine type $A$, where the level
$\ell$ and rank $n$ are swapped. See
\cite[Prop.~7.20]{2016arXiv160602002N}. Moreover the attracting set
$\fA_\chi(\lambda,\mu)$ is the tensor product variety studied in
\cite{Na-tensor}. More precisely its intersection with
$\cM^{\mathrm{s}}(\lambda,\mu)$ is the modified version of the tensor
product variety $\mathfrak Z_0^{\mathrm{s}}(\bv,\bw)$ introduced in
\cite[\S6]{Na-branching}.
\begin{NB}
  Strictly speaking, the tensor product variety depends on the choice
  $\chi$. If we take it from the negative chamber with $s_{n-1}$ is
  much smaller than others, it is clear that the attracting set
  coincides with one for ${\mathfrak Z}_0^{\mathrm{s}}$.
\end{NB}%
It was proved in \cite[\S6]{Na-branching} that the number of
$\operatorname{Irr}{\mathfrak Z}_0^{\mathrm{s}}(\bv,\bw)$ is equal to
the the tensor product multiplicity of
$V_{\mathfrak{gl}(\ell)_{\mathrm{aff}}}({}^t\mu)$ in the tensor
product of fundamental representations. By level-rank duality, this is
equal to the weight multiplicity of $V_\mu(\lambda)$ for
$\algsl(n)_{\mathrm{aff}(n)}$.

The second argument uses the computation of the stalk of
$\mathrm{IC}(\cM(\lambda,\mu))$ when $\mu$ is dominant in
\cite{braverman-2007}. As far as dimension is concerned, the stalk and
hyperbolic restriction give the same answer. Hence the result in
\cite[\S7]{braverman-2007} can be used. Note that
\cite[\S7]{braverman-2007} used a geometric construction of affine Lie
algebra modules via quiver varieties and level rank duality. In this
sense, the second argument is not far away from the first one.

\subsection{Finite dimensional cases}\label{subsec:Krylov}

Suppose that $Q$ is a quiver with symmetrizer as in
\cite{2019arXiv190706552N}, mentioned in the Introduction. We further
assume that it is of finite type. We denote the corresponding Lie
algebra by $\fg^\vee$, and the Langlands dual Lie
algebra by $\fg$.
Then $\cM(\lambda,\mu)$ is isomorphic to a generalized affine
Grassmannian slice $\oW^{\lambda}_\mu$ by
\cite{2016arXiv160403625B,2019arXiv190706552N}.
Here the complex reductive group $\vG$ for the affine Grassmannian has
the Lie algebra $\fg^\vee$ and of adjoint type. Let us fix a Borel
subgroup $B^\vee$ and an opposite Borel $B_-^\vee$.
Let $\Gr_{\vG} = \vG((z))/\vG[[z]]$ be the affine Grassmannian for $\vG$. Let
${\Gr}^{\lambda}_{\vG}$ be the $\vG[[z]]$-orbit through $z^\lambda$,
$\overline{\Gr}^{\lambda}_{\vG}$ its closure.
When $\mu$ is dominant, $\oW^\lambda_\mu$ is a transversal slice to
$\Gr^\mu_{\vG}$ in $\overline{\Gr}^\lambda_{\vG}$. 
For general $\mu$, $\oW^\lambda_\mu$ is the moduli space parametrizing
\begin{aenume}
  \item a $\vG$-bundle $\mathscr P$ on $\proj^1$.
  \item A trivialization of $\mathscr P$ over $\proj^1\setminus\{0\}$
    having a pole of degree $\le\lambda$ at $0\in\proj^1$.
  \item A $B$-structure on $\mathscr P$ of degree $w_0\mu$ having
    fiber $B_-$ at $\infty\in\proj^1$ with respect to the
    trivialization in (b).
\end{aenume}
See \cite[\S2(ii)]{2016arXiv160403625B} for more detail. Here we
mention that it is equipped with
$\mathbf p\colon \oW^\lambda_\mu\to \overline{\Gr}^{\lambda}_{\vG}$ by
(b),(c). It is a closed embedding when $\mu$ is dominant, and gives a
slice to $\Gr^\mu_{\vG}$ in $\overline{\Gr}^\lambda_{\vG}$.

We identify $T = (\CC^\times)^{Q_0}$ with the maximal torus of $\vG$,
which is the intersection $B^\vee\cap B^\vee_-$. (We will not use
$\vG$ except in this subsection. Therefore we do not use the notation
$T^\vee$.)
The $T$-action is identified with the standard one on
$\oW^\lambda_\mu$ under the isomorphism
$\cM(\lambda,\mu)\cong \oW^\lambda_\mu$.
Let us use $\cM(\lambda,\mu)$ instead of $\oW^\lambda_\mu$ hereafter.

We take an anti-dominant cocharacter $\chi\colon\CC^\times\to T$
as above.
By a result of Krylov \cite{2017arXiv170900391K} mentioned in
Introduction,
$\mathcal V_\mu(\lambda) = \Phi(\mathrm{IC}(\cM(\lambda,\mu)))$ is
isomorphic to the hyperbolic restriction considered by
Mirkovi\'c-Vilonen \cite{MV2} via a homomorphism induced by $\mathbf
p$.
Therefore it is equipped with a $\fg$-module structure, same as a
$G$-module structure, as $G$ is simply-connected, by the usual
geometric Satake correspondence.
On the other hand, \cref{conj:1,conj:2} were proved in
\cite{2017arXiv170900391K}, as we will explain below. Therefore the
construction in \cref{subsec:new_conj} is applicable.

\begin{Theorem}\label{thm:Satake}
  \footnote{The author thanks D.~Muthiah who asks him how to prove this
    result.} Operators $e_i$, $f_i$, $h_i$ on
  $\bigoplus \mathcal V_\mu(\lambda)$ given by the construction in
  \cref{subsec:new_conj} are equal to ones given by the usual
  geometric Satake correspondence.
\end{Theorem}

The proof occupies the rest of this subsection.

For type $A_1$, both constructions are explicit. We can directly
check that they are the same.

We consider the one parameter subgroup $\chi_i$ as in
\cref{subsec:new_conj}. We then have a diagram
$\Gr_{L_i} \xleftarrow{\tilde{p}_i} \Gr_{P_i^{-}}
\xrightarrow{\tilde{j}_i} \Gr_\vG$, where $L_i$ (resp.\ $P_i^{-}$) is
the Levi (resp.\ parabolic) subgroup for $\chi_i$.
We further choose connected components $\Gr_{L_i,\overline\mu}$,
$\Gr_{P_i^{-},\overline\mu}$ for $\overline\mu\in\pi_1(L_i)$
corresponding to $\mu$.
This is an analog of a diagram in \cref{subsec:twosteps}, and two
diagrams sit in a commutative diagram
\begin{equation}\label{eq:h:4}
  \begin{tikzcd}
    \cM(\lambda,\mu)^{\chi_i} \arrow["\mathbf p_{L_i}"']{d}
    & \fA_i \arrow["p_i"']{l} \arrow{r}{j_i} \arrow{d} &
    \cM(\lambda,\mu) \arrow["\mathbf p"]{d} \\
    \Gr_{L_i,\overline{\mu}} & \Gr_{P_i^-,\overline{\mu}}
    \arrow{l}{\tilde{p}_i} \arrow["\tilde{j}_i"']{r} & \Gr_\vG \rlap{.}
  \end{tikzcd}
\end{equation}
See \cite[(5.1)]{2017arXiv170900391K}. Here we denote the composite of
$\mathbf p\colon \cM(\lambda,\mu)\to \overline{\Gr}^{\lambda}_{\vG}$
with the inclusion $\overline{\Gr}^{\lambda}_{\vG}\to \Gr_\vG$ also by
$\mathbf p$ for brevity. The leftmost vertical morphism
$\mathbf p_{L_i}$ is the corresponding morphism for $L_i$.
Furthermore $\cM(\lambda,\mu)^{\chi_i}$ is isomorphic to
$\cM_{A_1}(\lambda',\mu')$ by
\cite[Lemma~5.5]{2017arXiv170900391K}. (Note that we have the maximal
$\lambda'$ as we are considering the case when the semisimple part of
$L_i$ is of type $A_1$.)
Since the fixed point set is naturally regarded a generalized slice
for $L_i$, let us change the notation from $\cM_{A_1}(\lambda',\mu')$
to $\cM_{L_i}(\lambda',\mu')$. We also understand $\lambda'$, $\mu'$
as coweights of $L_i$, instead of integers.
\begin{NB}
  Therefore $\lambda'$ is the maximal one among dominant weights
  of $\mu+\bv\alpha_i$ such that corresponding $L_i$-modules appear
  in the restriction of $V(\lambda)$.
\end{NB}%

The multiplicity vector space $M^{\lambda,\mu}_{\kappa',\mu'}$ is the
top degree Borel-Moore homology group of a subvariety $p_i^{-1}(x)$,
where $x$ is a point in the smooth locus of $\cM_{L_i}(\kappa',\mu')$.
Here $\kappa'$ is understood as a coweight of $L_i$.
The left square of \eqref{eq:h:4} is Cartesian by
\cite[Lemma~5.11]{2017arXiv170900391K}. Hence $p_i^{-1}(x)$ is
isomorphic to
$\tilde{p}_i^{-1}(\mathbf p_{L_i}(x))
\cap\overline{\Gr}^\lambda_{\vG}$.
When we move a point $y$ in the ${L_i}[[z]]$-orbit
$\Gr^{\kappa'}_{L_i}$, the top degree Borel-Moore homology group of
$\tilde{p}_i^{-1}(y)\cap\overline{\Gr}^\lambda_{\vG}$ forms a local
system, which is trivial as the $\Gr^{\kappa'}_{L_i}$ is
simply-connected.
(This triviality has been used in the usual geometric Satake
correspondence. See e.g., \cite[\S3.1]{bragai}.)
As a consequence, we have
\begin{Lemma}\label{lem:identification}
  Let $X$ be an irreducible subvariety in $\Gr^{\kappa'}_{L_i}$.
  There is a natural bijection between irreducible components of
  $\tilde{p}_i^{-1}(X)\cap\overline{\Gr}^\lambda_{\vG}$ and those of
  $\tilde{p}_i^{-1}(\Gr^{\kappa'}_{L_i})\cap\overline{\Gr}^\lambda_{\vG}$.
\end{Lemma}

As an application, we obtain a bijection between irreducible
components of $p_i^{-1}(x)$ and those of $p_i^{-1}(x^-)$ for
$x\in \cM_{L_i}^{\mathrm{s}}(\kappa',\mu')$,
$x^-\in\cM_{L_i}^{\mathrm{s}}(\kappa',\mu'-\alpha_i)$, as
$\overline{\mu-\alpha_i} = \overline\mu$ and $\kappa'$ is common.
Hence we get an isomorphism
$M^{\lambda,\mu}_{\kappa',\mu'}\cong
M^{\lambda,\mu-\alpha_i}_{\kappa',\mu'-2}$.
It also shows \cref{conj:2}, as the local system given by the top
degree Borel-Moore homology of $p_i^{-1}(x)$ is trivial over
$\cM^{\mathrm{s}}_{L_i}(\kappa',\mu')$.
In fact, Krylov also compared
$p_{i*}j_i^!\mathrm{IC}(\cM(\lambda,\mu))$ with
$\tilde{p}_{i*}\tilde{j}_i^!\mathrm{IC}(\overline{\Gr}_{\vG}^\lambda)$
by homomorphisms associated with vertical morphisms in
\eqref{eq:h:4}. See \cite[Proof of
Th.~3.4]{2017arXiv170900391K}.

Since the usual geometric Satake correspondence is constructed so that
it is compatible with the restriction to Levi subgroups (see e.g.,
\cite[Prop.~5.3.29]{Beilinson-Drinfeld}), operators $e_i$, $f_i$ are
given by the same method as in \cref{subsec:new_conj}.
Thus we only need to check that the above isomorphism
$M^{\lambda,\mu}_{\kappa',\mu'}\cong
M^{\lambda,\mu-\alpha_i}_{\kappa',\mu'-2}$ is the same as one given by
the factorization.

Let us factorize $\cM(\lambda,\mu-\alpha_i)$
\begin{NB}
  $\lambda - (\mu-\alpha_i) = \alpha + \alpha_i$
\end{NB}%
over
$(\BA^{\underline{\bv}}\times \mathbb
G_m^{\underline{\delta_i}})_{\mathrm{disj}}$ as in \eqref{eq:h:3}.
\begin{NB}
\begin{equation*}
  \cM(\lambda,\mu-\alpha_i)\times_{\BA^{\underline{\bv}+\underline{\delta_i}}}
  (\BA^{\underline{\bv}}\times
  \mathbb G_m^{\underline{\delta_i}})_{\mathrm{disj}}
  \cong (\cM(\lambda,\mu)\times Z^{\alpha_i})\times_{
    \BA^{\underline{\bv}}\times\BA^{\underline{\delta_i}}} 
  (\BA^{\underline{\bv}}\times
  \mathbb G_m^{\underline{\delta_i}})_{\mathrm{disj}},
\end{equation*}
where $\alpha_i^* = -w_0(\alpha_i)$. (Here $w_0$ is the longest
element of the Weyl group.)
\end{NB}%
The factorization induces an isomorphism between $p_i^{-1}(x)$ and
$p_i^{-1}(x^-)$, where $x^-$ corresponds to $(x,p)$ where
$p$ is a generic point in $\cM(\lambda,\lambda-\alpha_i)$.
In particular, we have a bijection between irreducible components of
$p_i^{-1}(x)$ and $p_i^{-1}(x^-)$.
We need to check that this bijection is the same as one
given by \cref{lem:identification}.
Let $X$ be an irreducible component of $p_i^{-1}(x)$ and $X^-$ the
corresponding one of $p_i^{-1}(x^-)$ under the factorization. We show
the following.
\begin{Claim}
  $\mathbf p(X)$ and $\mathbf p(X^-)$ are contained in the same
  irreducible component of
  $\tilde{p}_i^{-1}(\Gr^{\kappa'}_{L_i})\cap\overline{\Gr}^\lambda_{\vG}$.
\end{Claim}

\begin{proof}[Proof of the claim]
By the definition of generalized slices in
\cite[\S2(ii)]{2016arXiv160403625B}, $\cM(\lambda,\mu)$ is an open
subvariety of a larger subvariety $\CG\CZ^{-\mu}_\lambda$, which is
the moduli space of (a),(b),(c) above, but the $B$-structure in (c)
could have defects in $\proj^1\setminus \{\infty\}$. It is equipped
with $\mathbf p\colon\CG\CZ^{-\mu}_\lambda\to\Gr_{\vG}$ such that the
above $\mathbf p$ is its restriction. The diagram \eqref{eq:h:4}
extends to $\CG\CZ^{-\mu}_\lambda$.
We regard $X$, $X^-$ as irreducible components of
$p_i^{-1}((\CG\CZ^{-\mu}_\lambda)^{\chi_i})$ and
$p_i^{-1}((\CG\CZ^{-\mu+\alpha_i}_\lambda)^{\chi_i})$ respectively.

We have a locally closed embedding
$\cM(\lambda,\mu)\times\BA\hookrightarrow \CG\CZ^{-\mu+\alpha_i}_\lambda$ by
regarding the factor $\BA$ as a defect of the $B$-structure. The open
subset $\cM(\lambda,\mu)\times\mathbb G_m$ is contained in the image
of the factorization isomorphism
$\CG\CZ^{-\mu+\alpha_i}_\lambda\approx \CG\CZ^{-\mu}_\lambda\times
(\BA\times\BA)$ via
$\mathbb G_m = \mathbb G_m\times\{0\}\subset \BA\times\BA$:
In fact, the second factor $\BA\times \mathbb G_m$ in \eqref{eq:h:4}
is the Coulomb branch for type $A_1$ quiver with $\bw = 0$, $\bv =
1$. It is the space of based maps $\proj^1\to\proj^1$ of degree
$1$. By allowing defects as above, we get the corresponding zastava
space, isomorphic to $\BA\times\BA$. Since we are restricting to the
inverse image of $\mathbb G_m^{\underline{\delta_i}}$ under the
factorization morphism, the first $\BA$ factor is restricted to
$\mathbb G_m$.
Note also that the factorization for $\CG\CZ^{-\mu+\alpha_i}_\lambda$
is an extension of that for $\cM(\lambda,\mu-\alpha_i)$ by its
construction.

The restriction of
$\mathbf p\colon \CG\CZ^{-\mu+\alpha_i}_\lambda \to \Gr_\vG$ to
$\cM(\lambda,\mu)\times\mathbb G_m$ is equal to
$\mathbf p\colon\cM(\lambda,\mu)\to \Gr_\vG$, composed with the first
projection, as the defect has nothing to do with the data (a),(b).
Therefore $X^-$ contains a subvariety that is mapped to
$\mathbf p(X)$ under $\mathbf p$. Hence the claim is proved.
\end{proof}

Therefore the isomorphism
$M^{\lambda,\mu}_{\kappa',\mu'}\cong
M^{\lambda,\mu-\alpha_i}_{\kappa',\mu'-2}$ is the same as one given by the
factorization. This completes the proof of \cref{thm:Satake}.
  
Let us also mention that Kashiwara crystal structure on the set of
irreducible components of the attracting set in $\cM(\lambda,\mu)$
given in \cref{subsec:crystal} is the same as one constructed
in \cite{2017arXiv170900391K} by the same reason as above.


\appendix{}
\section{Fixed points and Maya diagrams}\label{sec:Maya}

\begin{NB}
  This subsection is a detour, and will not be used later.
\end{NB}%
Let us choose a real parameter $\nu^\RR$ so that
$\nu^\RR_* + \nu^\RR_h < \nu^\RR_{h+1}$ or $\nu^\RR_h < \nu^\RR_{h+1}$
according to $\vin{h}$ is connected to $\vin{\xl_0}$ through triangle
parts or not. (See \cref{rem:inequ}.) A complex parameter $\nu^\CC$ is
arbitrary. It could be $0$.
%
%

Because of this choice of $\nu^\RR$, the condition
($\boldsymbol\nu\bf 1$) is automatically satisfied, and
($\boldsymbol\nu\bf 2$) says that a graded subspace $T$ as in
($\boldsymbol\nu\bf 2$) must be the whole $V$. In particular,
$\nu^\RR$-semistability and $\nu^\RR$-stability are equivalent, and
$\cM^\nu$ is smooth.



Let us study the torus fixed point set $(\cM^\nu)^T$ as in
\cref{prop:torus-fixed-points}. Let us also transform so that the bow
diagram is of form \eqref{eq:16}. The data $(A,B,C,D,a,b)$ decomposes
into a direct sum corresponding to $\CC_{\xl_i}$ ($0\le i\le
n-1$). And a summand corresponding to a bow diagram of a form
\eqref{eq:4}.

Let us move $\xl_i$ to the right by Hanany-Witten transitions until the
vector space at the right of $\xl_i$ becomes $0$. Then we get data as
\begin{equation*}
  \xymatrix@C=1.2em{
        \CC^{k_m} \ar@<.5ex>[r]^{C_{m-1}}
        & \CC^{k_{m-1}} \ar@<.5ex>[l]^{D_{m-1}} \ar@<.5ex>[r]^{C_{m-2}}
        & \cdots \ar@<.5ex>[l]^{D_{m-2}} \ar@<.5ex>[r]^{C_2}
        & \CC^{k_2} \ar@<.5ex>[l]^{D_2} \ar@<.5ex>[r]^{C_1}
        & \CC^{k_1} \ar@<.5ex>[l]^{D_1} \ar[dr]_{b} \ar@(ur,ul)_{B}
        &
        \\
        &&&&& \CC \rlap{.}}
    \begin{NB}
    \xymatrix@C=1.2em{
    & \CC^{k_0}  \ar@(ur,ul)_{B} \ar@<.5ex>[r]^{C_0}
    & \CC^{k_{-1}} \ar@<.5ex>[r]^{C_{-1}} \ar@<.5ex>[l]^{D_0}
    & \cdots \ar@<.5ex>[l]^{D_{-1}} \ar@<.5ex>[r]^{C_{2-m}}
    & \CC^{k_{1-m}} \ar@<.5ex>[l]^{D_{2-m}} \ar@<.5ex>[r]^{C_{1-m}}
        & \CC^{k_{-m}}. \ar@<.5ex>[l]^{D_{1-m}}
        \\
        \CC \ar[ur]_{a} &&&&&}
    \end{NB}%
\end{equation*}

Let us first take the reduction by $\GL(k_2)\times\cdots\times
\GL(k_m)$ and next take the reduction by $\GL(k_1)$. By the first step
we get the product of (a) a quiver variety of type $A_{m-1}$ with
dimension vectors $\underline{\bv} = (k_{m},\dots, k_{2})$,
$\underline{\bw} = (0,\dots,0,k_1)$, and (b) a pair $(B,b)$ of a
$k_1\times k_1$-matrix and a co-vector in $\CC^{k_1}$ such that $b$ is
co-cyclic with respect to $B$.
In the second step we set $B = - C_1 D_1$ and take the quotient by
$\GL(k_1)$.

\begin{NB}
    The condition (S1) says that $b$ is a co-cyclic vector with
    respect to $B$.
Therefore we can make $B = {}^t\! J_{k_1}$, $a = {}^t e_{k_1}$ by
the action of $\GL(k_1)$, and this normalization kills the action of
$\GL(k_1)$. Hence this data can be considered as a point of a quiver
variety of type $A_{m-1}$ with dimension vectors $\underline{\bv} =
(k_{m},\dots, k_{2})$, $\underline{\bw} = (0,\dots,0,k_1)$ with an
additional constraint $C_1 D_1 = -{}^t\! J_{k_1}$.
\end{NB}%
By the standard argument (cf.\ \cite[Th.~7.3]{Na-quiver}) $D_i$
($1\le i\le m-1$) is surjective. In particular, $\dim S_i$ is
decreasing. However $C_1 D_1 = -B$ has a cocyclic vector, we must have
$k_{s} - k_{s+1} = 0$ or $1$.
We fill $\medcirc$ for two way parts with $k_{s} - k_{s+1} = 1$ as
$\circfilled$.

The defining equation determines the characteristic polynomial of $C_1
D_1$ (e.g., it is $z^{k_1}$ if the complex parameter $\nu^\CC$ is
$0$). Then $B = - C_1 D_1$ together with $b$ forms a single free
$\GL(k_1)$-orbit.
Note also $C_1$, $D_1$, \dots are determined automatically once $k_1$,
$k_2$, \dots are specified, if $B$ is fixed. Namely the corresponding
bow variety is a single point.

Returning back to \eqref{eq:4} by Hanany-Witten transitions, we find that
$\medcirc$ for $h$ is filled as $\circfilled$ if and only if
\begin{equation*}
    \begin{cases}
    N_h = 1 & \text{if $h=h_j$ with $j > 0$},\\
    N_h = 0 & \text{if $h=h_j$ with $j \le 0$}.
    \end{cases}
\end{equation*}
Though we move $\xl_i$ only finite amount, we extend the above rule to
any $\medcirc$. Hence we have a sequence of $\medcirc$'s going to
infinite in both left and right with some filled as $\circfilled$ such
that they are filled for $h_j$ with sufficiently negative $j$, not filled sufficiently positive $j$.
We have an infinite sequence for each $\xl_i$ ($0\le i\le n-1$), and
we arrange them as follows:
\begin{equation}\label{eq:Maya}
   \text{\scriptsize $n$ rows }\Big\{
   \Yvcentermath1
   \cdots
   \underbrace{\overset{-3/2}{\young(\rf\rf\rf,\rf\rf\rf)}}_{\text{$\ell$ columns}}
   \,
   \overset{-1/2}{\young(\rf\hf\rf,\rf\rf\hf)}
   \,
   \overset{1/2}{\young(\rf\rf\rf,\hf\rf\hf)}
   \,
   \overset{3/2}{\young(\hf\hf\hf,\hf\hf\hf)}
   \cdots,
\end{equation}
where $h_1$ is the first column in the block $1/2$, $h_2$ is the
second one, and $h_0$ is the last column in the block $-1/2$, and so on. 
This is a variant of a Maya diagram.

Conversely a diagram above gives a torus fixed point: Reading
$(i+1)$-th row, we determine the bow diagram corresponding to
$\xl_{i}$ ($0\le i \le n-1$) including dimensions $R(\zeta)$. Then we
take the sum over $i$.

Note that $R(\vout{\xl_i})$ (resp.\ $R(\vin{\xl_i})$) is equal to the
number of $\graysquare$ (resp.\ $\square$) in blocks $1/2$,
$3/2$, \dots (resp.\ $-1/2$, $-3/2$, \dots) of $\xl_i$. In
particular,
\(
    N_{\xl_i}
    \begin{NB}
        = R(\vout{\xl_i}) - R(\vin{\xl_i})
    \end{NB}
\)
is the difference between these numbers. Since a summand corresponding to
$\xl_j$ ($j\neq i$) has isomorphic $A_{\xl_i}$, $N_{\xl_i}$ is the
same for this summand and the original bow diagram.
In other words, if a diagram corresponding to a $T$-fixed point in
$\cM^\nu$ of a bow diagram of form \eqref{eq:16}, we have a constraint
\begin{equation}\label{eq:15}
    N_{\xl_i} = 
    \begin{aligned}
        & \text{the number of $\graysquare$ in the $(i+1)$-th
          row in blocks $1/2$, $3/2$,\dots}\\
        & \quad - \text{the number of $\square$ in the $(i+1)$-th row
          in blocks $-1/2$, $-3/2$,\dots}
    \end{aligned}
\end{equation}

Let us consider numbers of $\graysquare$, $\square$ in each column on
the other hand. In blocks $1/2$, $3/2$, \dots, $\graysquare$
contributes $1$ to $N_h$. In blocks $-1/2$, $-3/2$, \dots, $\square$ contributes $-1$. Therefore
\begin{equation}\label{eq:17}
    N_{h_\sigma} = 
    \begin{aligned}
        & \text{the number of $\graysquare$ in the $\sigma$-th
          column in blocks $1/2$, $3/2$,\dots}\\
        & \quad - \text{the number of $\square$ in the $\sigma$-th row
          in blocks $-1/2$, $-3/2$,\dots}
    \end{aligned}
\end{equation}

Though $N_{\xl_i}$, $N_{h_\sigma}$ determine $R(\zeta)$ only up to
over all shifts, the number $\bv_0$ in \eqref{eq:16} is given by
\begin{equation}\label{eq:18}
  \bv_0
  =
  \begin{aligned}[t]
    & (\text{the number of $\square$ in blocks $-1/2$}) \\
    &\quad + 2 (\text{the number of $\square$ in blocks $-3/2$}) \\
    &\quad\quad + 3 (\text{the number of $\square$ in blocks $-5/2$})
    + \cdots.
  \end{aligned}
\end{equation}
\begin{NB}
\[
  = \sum_{i=0}^{n-1} R(\vin{\xl_i}) + \text{contribution after one round}
  + \text{contribution after two rounds} + \cdots
\]
\end{NB}%
Therefore
\begin{Theorem}
  The fixed point set $(\cM^\nu)^T$ is in bijection to the set of
  diagrams \eqref{eq:Maya} with constraint
  \cref{eq:15,,eq:17,,eq:18}.
\end{Theorem}

In particular, we have a natural bijection between fixed point sets of
the bow variety and another bow variety given by the bow diagram with
$\boldsymbol\times$, $\boldsymbol\medcirc$ swapped. It includes the
case of Higgs and Coulomb branches of the same quiver gauge theory of
affine type $A$. This should be equal to the natural bijection given as
a consequence of Hikita conjecture
\cite{2015arXiv150102430H,2015arXiv150303676N}.

\bibliographystyle{myamsalpha}
\bibliography{nakajima,mybib,coulomb}

\end{document}